\documentclass{article}
\usepackage[%
journal=FoCM,    
lang=american,   
]{ems-journal}

\usepackage[margin=0pt]{subfig}
\usepackage{mathtools}
\usepackage{bm}
\usepackage{placeins,siunitx}
\usepackage{tikz,pgfplots}
\usetikzlibrary{arrows.meta,backgrounds}
\usepgfplotslibrary{groupplots}
\usepgfplotslibrary{fillbetween}
\pgfplotsset{compat=1.18}
\usepackage{thm-restate}
\usepackage{multirow}
\usepackage{tabularx}
\usepackage{dsfont}
\usepackage[left]{lineno}

\newcommand{\mypara}[1]{\paragraph{#1}}
\newcommand{\norm}[1]{\lVert #1\rVert}
\newcommand{\abs}[1]{\vert #1\vert}
\newcommand{\EE}{\mathbb{E}}
\newcommand{\RR}{\mathbb{R}}
\newcommand{\NN}{\mathbb{N}}
\newcommand{\YY}{\mathbb{Y}}
\newcommand{\PP}{\mathbb{P}}
\DeclareMathOperator{\tr}{tr}

\theoremstyle{plain}
\newtheorem{theorem}{Theorem}[section]
\newtheorem{lemma}[theorem]{Lemma}
\newtheorem{proposition}[theorem]{Proposition}
\newtheorem{corollary}[theorem]{Corollary}

\theoremstyle{definition}
\newtheorem{definition}[theorem]{Definition}
\newtheorem{example}[theorem]{Example}
\newtheorem{remark}[theorem]{Remark}

\numberwithin{equation}{section}

\begin{document}
	\title{Phase Transition from Random Features to Fully Trainable Networks}
	\titlemark{Phase Transition from Random Features to Fully Trainable Networks}
	
	
	%
	
	\emsauthor{1}{
		\givenname{Xin}
		\surname{He}
		\mrid{}  
		\zblid{} 
		\orcid{}}{X.~He} 
	\emsauthor{2}{
		\givenname{Huan}
		\surname{Lao}
		\mrid{}
		\zblid{}
		\orcid{}}{H.~Lao}
	\emsauthor{3}{
		\givenname{Xiliang}
		\surname{Lu}
		\mrid{}
		\zblid{}
		\orcid{}}{X.~Lu}
	\emsauthor*{4}{
		\givenname{Shijun}
		\surname{Zhang}
		\mrid{}
		\zblid{}
		\orcid{}}{S.~Zhang}
	\Emsaffil{1}{
		\department{School of Mathematics and Statistics}
		\organisation{Wuhan University}
		\rorid{033vjfk17}
		\address{Bayi Road 299}
		\zip{430072}
		\city{Wuhan}
		\country{China}
		\affemail{\detokenize{xin_he@whu.edu.cn}}}
	\Emsaffil{2}{
		\department{School of Mathematics and Statistics}
		\organisation{Wuhan University}
		\rorid{033vjfk17}
		\address{Bayi Road 299}
		\zip{430072}
		\city{Wuhan}
		\country{China}
		\affemail{laohuan0618@whu.edu.cn}}
	
	\Emsaffil{3}{
		\department{1}{School of Mathematics and Statistics}
		\organisation{1}{Wuhan University}
		\rorid{1}{033vjfk17}
		\address{1}{Bayi Road 299}
		\zip{1}{430072}
		\city{1}{Wuhan}
		\country{1}{China}
		\department{2}{Computational Science Hubei Key Laboratory} 
		\organisation{2}{Wuhan University}%
		\rorid{2}{}
		\address{2}{}%
		\zip{2}{430072}
		\city{2}{Wuhan}
		\country{2}{China} 
		\affemail{2}{xllv.math@whu.edu.cn}}
	\Emsaffil{4}{
		\department{Department of Applied Mathematics}
		\organisation{The Hong Kong Polytechnic University}
		\rorid{0030zas98}
		\address{Hung Hom, Kowloon}
		\zip{}
		\city{Hong Kong SAR}
		\country{China}
		\affemail{shijun.zhang@polyu.edu.hk}}
	
	\classification[41A25, 68T07]{41A46}
	
	
    \begin{abstract}
Constructing accurate and computationally tractable approximation spaces is a fundamental task in scientific computing. Partially trainable neural networks bridge fixed random features and fully adaptive representations. This paper quantifies a phase transition pattern in approximation bounds and observes analogous behavior in training errors: a rapid but smooth initial improvement followed by diminishing returns, with a sharper transition under greater initialization mismatch. Under polynomial spectral decay with exponent $\alpha>1$ and leverage compatibility, we establish a squared $L_2$ approximation rate $\mathcal{O}((\kappa_P\log(m/\delta)/m)^\alpha)$ for two-layer networks of width $m$ and trainable proportion $P$, where $\kappa_P=(P/\kappa_1^*+(1-P)/\kappa_0^*)^{-1}$. The bound holds uniformly over the unit ball of a reproducing kernel Hilbert space (RKHS) with probability at least $1-\delta$. The constants $\kappa_0^*$ and $\kappa_1^*$ quantify the fixed and trainable groups' sampling mismatches relative to the target RKHS. For spherical biased-ReLU features in dimension $d$, we verify these assumptions with $\alpha=(d+2)/(d-1)$ and $\kappa_0^*=4\sqrt{2}\tau(d+1)/d$ for initialization scale $\tau\ge1$, and $\kappa_1^*=1$ under ridge leverage sampling. For $d=5$ and $\tau=10$, a trainable proportion of approximately $4.1\%$ captures $90\%$ of the reduction in the approximation bound from $P=0$ to $P=1$. A structured three-layer example illustrates how a few trainable first-layer neurons exploit fixed downstream blocks through nonlinear composition. Numerical experiments across diverse settings, including networks with up to 15 layers, exhibit similar phase transitions in training errors.
\end{abstract}

\makeatletter
\g@addto@macro\ems@authlist{\textsuperscript{*}}
\makeatother

\begingroup
\makeatletter
\let\@makefnmark\relax
\makeatother
\setlength{\footnotemargin}{0pt}
\footnotetext{%
    \textsuperscript{*}Corresponding author:
    \href{mailto:shijun.zhang@polyu.edu.hk}
    {shijun.zhang@polyu.edu.hk}.%
}
\endgroup

\maketitle
    
    \thispagestyle{empty}
	\tableofcontents
	\section{Introduction}
The pursuit of accurate, efficient, and stable representations lies
at the heart of scientific computing. A key task is to construct
suitable approximation spaces. Many classical methods, such as
polynomial, Fourier, and wavelet approximation, prescribe a family
of functions and approximate the target through their linear
combinations. Random feature methods (RFMs)
\cite{pao1994learning,rahimi2007random} follow a similar principle,
using fixed sampled features whose ability to capture the target
depends on how well the initialization distribution matches it
\cite{bach2017equivalence,rudi2017generalization,nelsen2024operator}.
Feed-forward neural networks (FNNs) offer a more adaptive approach:
training the hidden parameters allows the feature functions themselves
to adapt to the target, yielding strong expressive power
\cite{pinkus1999approximation,zhong2016overview,shen2022optimal}.
RFMs and fully trainable FNNs thus represent two endpoints, connected
by partially trainable networks in which a proportion of hidden
parameters are trainable and the rest randomly fixed. This paper
investigates the phase transition between these endpoints, addressing
the question:
\begin{center}
    \begin{minipage}{0.94\linewidth}
        \itshape
        How does the performance of partially trainable neural networks
        vary as the trainable proportion increases?
    \end{minipage}
\end{center}

To investigate this question, this paper studies the $P$-trainable networks, where 
the same proportion $P\in[0,1]$ of neurons in each hidden layer are trainable and the rest are randomly fixed. Although arbitrary subsets of network parameters can be made trainable, the present formulation assigns trainability at the neuron level. Our approximation analysis shows that the effect of the trainable proportion depends on how well the fixed and trainable neuron groups match the target function class. The corresponding theoretical bounds and numerical experiments together identify a distinct phase transition pattern: A small trainable group can offset the mismatch between the initialization distribution and the target function class, yielding substantial performance gains at small $P$ followed by diminishing returns, and greater initialization mismatch sharpens the transition.

Building on the approximation results for two-layer $P$-trainable networks with general activation functions, the analysis of spherical biased-ReLU features provides the explicit rates that quantify this dependence on initialization. Figure \ref{fig:intro-spherical-phase-transition} illustrates qualitative agreement between the predicted and observed phase transitions across initialization scales $\tau\ge1$. Within this family, the chosen admissible mismatch constant increases linearly with $\tau$. Increasing $\tau$ generally sharpens phase transitions. More concretely, for $d=5$ and $\tau=10$, a trainable proportion of only about $4.1\%$ captures $90\%$ of the total reduction in the approximation bound from $P=0$ to $P=1$.

A structured three-layer example further demonstrates how a small trainable group uses randomly fixed downstream neurons through nonlinear composition enabled by depth. Extensive numerical experiments across target functions and architectures exhibit similar phase transition patterns, with representative width and depth comparisons in Figure~\ref{fig:representative-result-introduction-phase-transition}.

\begin{figure}[!t]
	\centering
	\subfloat[Threshold for $90\%$ gain.]{%
		\includegraphics[width=0.47\linewidth]
		{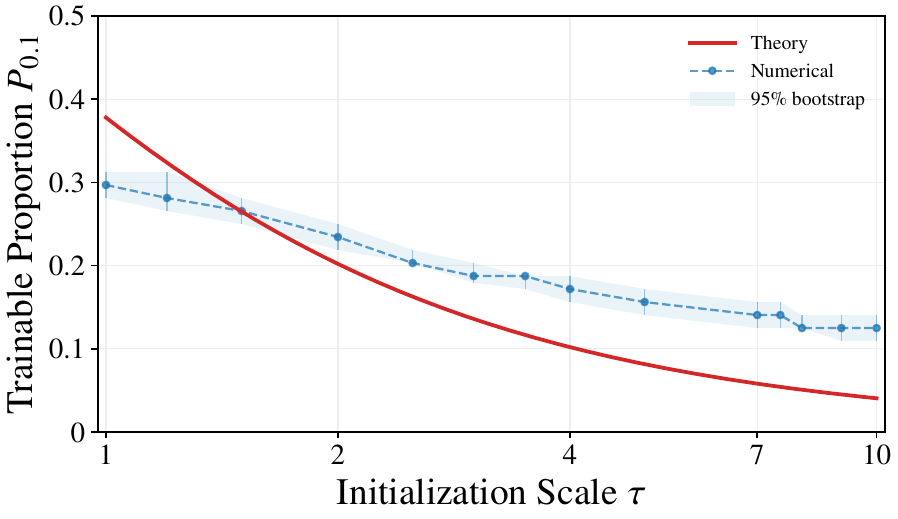}%
		\label{fig:intro-threshold}%
	}
	\hfill
	\subfloat[Normalized gain.]{%
		\includegraphics[width=0.47\linewidth]
		{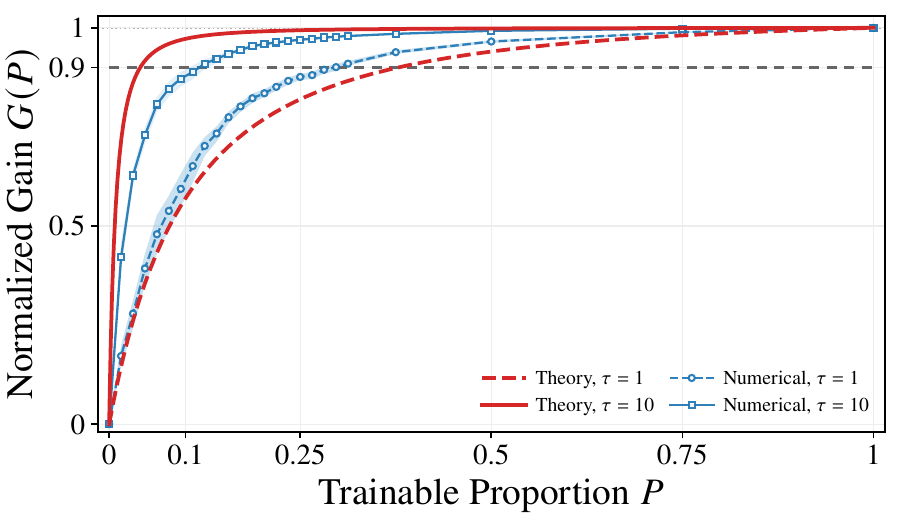}%
		\label{fig:intro-gain}%
	}
	\caption{Phase transitions for two-layer spherical biased-ReLU $P$-trainable networks with $d=5$. Red curves are derived from the approximation bound, while blue curves show training results for a fixed RKHS target. Increasing the fixed initialization scale $\tau\ge1$ raises the mismatch constant from its minimum at $\tau=1$. Both theory and training show qualitatively similar transition shapes and sharper transitions at larger $\tau$. Quantitative differences are expected because training also depends on finite sampling, the optimizer, and the training budget. Threshold $P_{0.1}$ marks $90\%$ normalized gain. Shaded bands indicate $95\%$ bootstrap intervals. See experimental details in Appendix~\ref{subsec:tau-experiments}.}
	\label{fig:intro-spherical-phase-transition}
	\setcounter{subfigure}{0}
	\subfloat[Depth variation from 3 to 7 at width 64.]{%
		\includegraphics[width=0.47\linewidth]
		{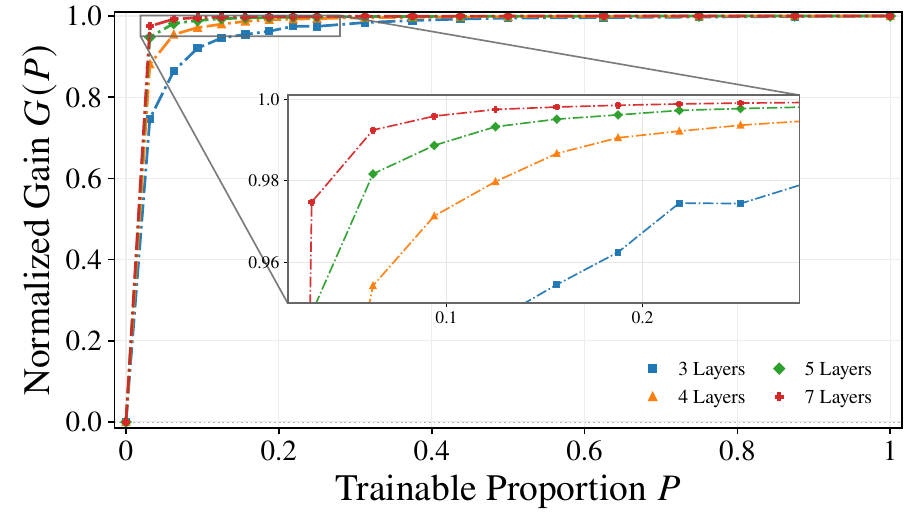}%
	}
	\hfill
	\subfloat[Width variation from 32 to 256 at depth 6.]{%
		\includegraphics[width=0.47\linewidth]
		{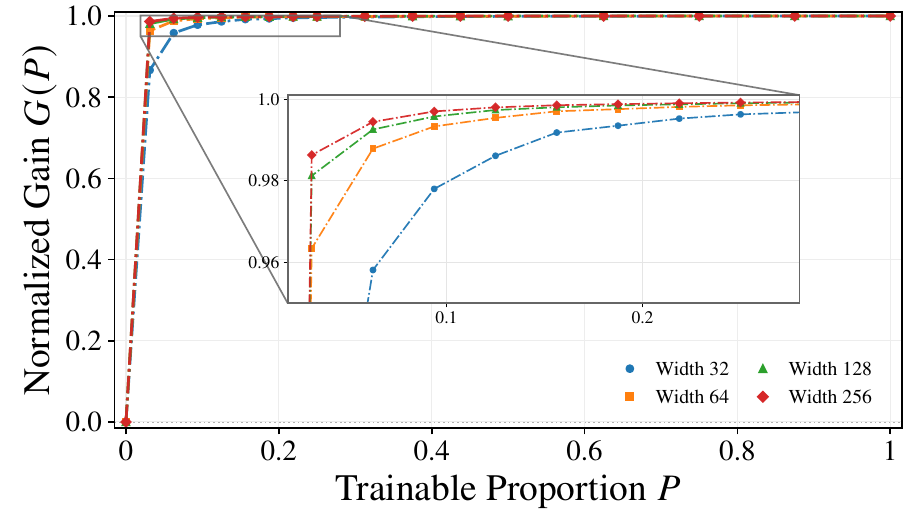}%
	}
	\caption{Representative numerical normalized gains for $P$-trainable ReLU networks on a surface target function. Gains are normalized by the observed loss reduction between $P=0$ and $P=1$. Both depth and width comparisons show rapid gains at small $P$, followed by diminishing gains.}
	\label{fig:representative-result-introduction-phase-transition}
\end{figure}

\subsection{Main Results}
The main contributions of this work are summarized as follows.

\mypara{(\romannumeral1) Small-$P$ Phase Transition for Shallow Networks.} For two-layer $P$-trainable networks, the analysis samples the trainable group from an adjustable distribution to quantify how the trainable proportion and initialization mismatch affect approximation bounds.
\begin{itemize}[labelindent=0pt, leftmargin=*, topsep=0pt, itemsep=0pt]
	\item \textit{General approximation rates.} Theorem \ref{thm:general-approx.-bounds} gives a sufficient width condition for approximating the RKHS unit ball uniformly to any prescribed tolerance with high probability. Under polynomial spectral decay with exponent $\alpha>1$ and leverage compatibility, Theorem \ref{thm:polynomial-rates} then gives the approximation rate for width $m$
	\[ \mathcal{O}\bigg(\Big(\Big(\frac{P}{\kappa_{1}^{*}}+\frac{1-P}{\kappa_{0}^{*}}\Big)\frac{m}{\log(m/\delta)}\Big)^{-\alpha}\bigg),\quad P\in[0,1], \]
	with probability at least $1-\delta$, where $\delta\in(0,1)$. The constants $\kappa_{1}^{*}$ and $\kappa_0^*$ control how well the sampling distributions of the trainable and randomly fixed groups match the target function class, respectively. Both groups contribute to the approximation bound, and their relative contributions depend on the mismatch ratio $\kappa_0^*/\kappa_1^*$.
	\item \textit{Phase transition sharpness.} Corollary \ref{cor:phase-transition} shows that when $\kappa_0^*/\kappa_1^*\gg1 $, a trainable proportion of order $(\kappa_0^*/\kappa_1^*)^{-1}$ suffices to capture any fixed fraction of the reduction between the endpoint bounds. A larger mismatch ratio produces a sharper transition in the bound.
	\item \textit{Spectral and mismatch verification.} For spherical biased-ReLU features on $\mathbb S^{d-1}$, Propositions \ref{prop:spectrum-biased-relu-kernel} and \ref{prop:anisotropic-spherical-mismatch} establish the spectral exponent $\alpha=(d+2)/(d-1)$ and the admissible mismatch constant $\kappa_0^*=4\sqrt2\tau(d+1)/d$. With $\kappa_1^*=1$ from ridge leverage sampling of the trainable group, Theorem \ref{thm:sphericial-polynomial-rates} gives the explicit spherical rate. Increasing the fixed initialization scale \(\tau\ge1\) generally sharpens predicted and observed transitions (Figure~\ref{fig:intro-spherical-phase-transition}).
\end{itemize}
\mypara{(\romannumeral2) Similar Transitions in Deep Networks.} A deep neural network example and numerical experiments show similar small-$P$ transitions and further examine the roles of depth and randomly fixed neurons.
\begin{itemize}[labelindent=0pt, leftmargin=*, topsep=0pt, itemsep=0pt]
	\item \textit{Gains from nonlinear composition.} Example \ref{exm:structed-three-layer-compar} provides explicit approximation upper bounds for a structured three-layer network in which a few trainable first-layer neurons use several randomly fixed downstream neurons. The bounds exhibit a similar small-$P$ phase transition. At selected trainable neuron counts in the first hidden layer, the three-layer construction also gives a smaller error upper bound than its two-layer counterpart, highlighting the gains from nonlinear composition enabled by depth.
	\item \textit{Phase transitions in training performance.} Across three targets and different widths and depths, most tested configurations in Section \ref{sec:numerical-illustration} recover a substantial fraction of the observed loss reduction between the two endpoints at small $P$, followed by diminishing gains. The transition strength varies with the target and architecture, with representative depth and width results in Figure \ref{fig:representative-result-introduction-phase-transition}. Besides, Table \ref{tab:randomly-fixed-part} shows lower training losses with more randomly fixed neurons at a fixed trainable neuron count per hidden layer, supporting the expressive role for the fixed group.
    
\end{itemize}
\subsection{Related Work}
The following discussion reviews approximation and learning theory for RFMs, partially trainable networks, and fully trainable FNNs, as well as studies of pruning and parameter freezing. For each topic, it explains the connections to our work and the differences in model construction, analysis, or research objectives.

\mypara{Approximation Theory for Random Feature Methods.} RFMs approximate the infinite-dimensional feature representation of a kernel using finitely many random features sampled from a prescribed distribution \cite{rahimi2007random, rahimi2008weighted}. Early analyses focused on approximating the kernel itself, particularly with random Fourier feature analysis \cite{sutherland2015error,sriperumbudur2015optimal}. Subsequent operator-based work considered general integral feature maps and studied RFMs via a framework of function approximation and statistical learning. In this framework, the required feature number is governed by the spectrum and effective dimension of the associated integral operator \cite{bach2017equivalence,rudi2017generalization,long2024optimal}. Beyond sampling from the original feature distribution, leverage-based sampling schemes have been developed at both the population and empirical levels to reduce the required number of features \cite{bach2017equivalence,avron2017random,li2019towards,liu2020random}. When the feature map is specialized to ReLU neurons, approximation bounds have also been established for two-layer networks with randomly sampled hidden parameters \cite{hsu2021approximation,gonon2023approximation}. 

RFMs also have limitations in approximation capacity. Under standard Gaussian inputs, approximating even a single ReLU neuron can require exponentially many features or exponentially large output coefficients in the input dimension \cite{yehudai2019power}. Approximation lower bounds derived using ridgelet analysis show that two-layer RFMs with bounded hidden parameters can fail to achieve universal approximation when the sampling support omits a nonzero part of the target’s ridgelet spectrum \cite{li2023powerful}.

Our analysis extends the operator framework in \cite{bach2017equivalence} from a single sampling distribution to two independent groups of fixed sizes sampled from separate distributions, giving unified approximation bounds over $P\in[0,1]$. Beyond the sampling support restrictions studied in \cite{li2023powerful}, our mismatch factor quantifies insufficient sampling mass relative to the target RKHS and its effect on the sufficient network width. The bounds further show how a small, appropriately chosen trainable group can compensate for the initialization mismatch.

\mypara{Learning Theory for Partially Trainable Networks.}
Partially trainable networks have been studied under various constraints on
which parameters are optimized. Semi-random feature models combine randomly
fixed and trainable components within individual units
\cite{kawaguchi2018deep}, while other work establishes universal approximation
for networks with fixed random weights and trainable biases
\cite{williams2025expressivity}. Optimization and function-space analyses
also provide convergence and complexity results for models with a randomly
fixed first hidden layer \cite{chen2022on,chen2026functional}.
Related approximation studies investigate representations with a small
number of target-dependent parameters
\cite{shijun:intrinsic:parameters,2026arXiv260831157Z},
as well as multigrade and layer-wise approximation constructions
\cite{2026arXiv260116884Z,shijun:Geometric:Layer-wise:rates:sinReLU}.

Our analysis considers a different form of partial trainability: hidden
neurons within each layer are divided into trainable and randomly fixed
groups, and we study the resulting finite-width approximation properties.

\mypara{Approximation Theory for Feed-Forward Neural Networks.}
Classical universal approximation theorems establish the density of FNNs
with suitable activations in spaces of continuous functions
\cite{cybenko1989approximation,hornik1989multilayer}.
Quantitative approximation theory relates approximation accuracy to network
size, target regularity, and architectural constraints.
For shallow networks, Barron-type and variation spaces provide important
frameworks for this analysis
\cite{barron1993universal,siegel2023characterization,devore2025weighted,2019arXiv190608039E},
with approximation rates and optimality results available under various
assumptions on the activation and error norm
\cite{siegel2020approximation,siegel2022high,ma2022uniform,siegel2024sharp}.
More general approximation-space frameworks have also been developed for
deep networks \cite{2019arXiv190501208G}.

Quantitative bounds for H\"older, Sobolev, and Besov classes have been
established using ReLU and related rectified-power networks
\cite{yarotsky2017error,lu2021deep,shen2022optimal,siegel2023optimal,suzuki2018adaptivity,shijun:Characterized:by:Numer:Neurons,2019arXiv190207896G,mao2026approximation}.
Other results address analytic and bandlimited functions
\cite{E:2018:exponential:DNN:analytic,bandlimit},
as well as smooth functionals \cite{SONG2023424}.
Further studies investigate how activation choices and constraints on
network size or parameter magnitude affect approximation capabilities
\cite{shijun:arbitrary:error:with:fixed:size,shijun:three:layers,shijun:2023:beyond:ReLU:to:diverse:actfun,jiao2023deepnocurse,2026arXiv260616975L,2026arXiv260925710L}.

Beyond target regularity, approximation theory also examines the role of
input geometry
\cite{shaham2018provable,chen2019efficient,Jian-Feng-Cai-2021}
and architectural structure, including convolutional networks,
dynamical-systems formulations, Transformers, and structured network models
\cite{ZHOU2019,NEURIPS2022_136302ea,li2022deep,cheng2025interpolation,cheng2026unified,zhang2025structured}.
Connections with finite element methods and hierarchical basis
constructions provide additional perspectives
\cite{ce76a17e-6e58-3d02-9e13-a9c53bfaa2fd,HE2022105}.

Our work complements these studies by examining how partial trainability
affects approximation within a given architecture and target function class.

\mypara{Lottery Ticket Hypothesis and Parameter Freezing.} The lottery ticket hypothesis states that a randomly initialized dense network contains a smaller subnetwork that, when trained in isolation, can attain performance comparable to the original network \cite{frankle2019iclr-lottery}. Related studies show that suitably pruned randomly initialized networks may contain subnetworks capable of approximating target networks with little or no weight training \cite{malach2020proving,ramanujan2020shiddennetwork}. Instead, parameter freezing trains only a selected subset of parameters while holding the remainder fixed, often at their random initializations \cite{rosenfeld2019intriguing, wimmer2020freezenet}. Studies of parameter freezing investigate parameter selection rules and freezing ratios, often with an emphasis on practical performance and training efficiency.

Our setting is related to these studies but differs in both construction and objective. A prescribed proportion of hidden neurons is trainable. The remaining hidden parameters stay fixed after initialization, while all neurons are retained in the forward computation. The objective is to study approximation capacity as the trainable proportion varies, rather than to search for subnetworks or design parameter selection rules.

\mypara{Organization.} The rest of the paper is organized as follows. Section \ref{sec:setup} introduces the model, approximation task and analytical framework. Section \ref{sec:main-results} presents the general approximation bound, explicit polynomial rates, and the resulting phase transition phenomenon. Section \ref{sec:specific-activation-functions} specializes the analysis to specific spherical biased-ReLU features. Section \ref{sec:three-layer-illustration} provides the structured three-layer illustration. Section \ref{sec:numerical-illustration} reports the numerical illustrations, and Section \ref{sec:discuss-conclusion} concludes the paper. Supplementary proofs are provided in the appendices.
	\section{Preliminaries and Problem Setting}
\label{sec:setup}

\mypara{Notation.} Let $\mathbb{N}_{+}\coloneq\{1,2,3,\ldots\}$ and $ \mathbb{N}\coloneq\mathbb{N}_{+}\cup\{0\}$. For $n\in\mathbb N_{+}$, write $[n]=\{1,\ldots,n\}$,
with $[0]=\varnothing$. Empty sums are understood to be zero. We denote the set of real numbers by $\mathbb{R}$. Write $A\lesssim B$ and $B\gtrsim A$ when $A\le cB$ for a constant $c>0$. And $A\asymp B$ means $A\lesssim B$ and $B\lesssim A$. Vectors written as tuples are understood as column vectors. For ${x}\in\mathbb{R}^{d}$, $\norm{x}_{p}$ denotes the $\ell_{p}$-norm for $p\in[1,+\infty)$. For a measurable function $f:\mathbb{R}^{d}\to\mathbb{R}$, $\norm{f}_{L_{p}(d\mu)}$ denotes the $L_{p}$-norm with respect to $\mu$. Table \ref{tab:summary-of-notations} in the appendix summarizes the main notations for quick reference.

\subsection{Model Space: $P$-Trainable Neural Networks}
\label{subsec:network-definition}
\mypara{Feed-forward Neural Networks (FNNs).}
Let $W\in\mathbb{N}_{+}$ and $L\ge 2$. We consider a feed-forward neural network
$f_L\in\mathcal{F}(W,L):\mathbb{R}^{d}\to\mathbb{R}$ of width $W$ and
depth $L$, with $L-1$ hidden layers. Set $m_0=d$ and $m_L=1$. For $\ell\in[L-1]$, let $\theta_{\ell,j}\in\mathbb{R}^{m_{\ell-1}+1}$ denote the parameter vector of the $j$-th neuron in the $\ell$-th hidden
layer and $\phi:\mathbb{R}\to\mathbb{R}$ be the activation function. The network is then defined recursively by
\begin{equation}
	\label{eq:Full-FNN}
	\begin{aligned}
		f_0(x)&=({x},1),\\
		f_\ell({x})&=
		\bigl(
		\phi({x};\theta_{\ell,1}),\ldots,
		\phi({x};\theta_{\ell,m_\ell}),1
		\bigr), \quad \ell\in[L-1],\\
		f_L(x)&=\sum\nolimits_{j=1}^{m_{L-1}}
		\beta_j\phi(x;\theta_{L-1,j}),
	\end{aligned}
\end{equation}
where $\beta=(\beta_1,\ldots,\beta_{m_{L-1}})\in\mathbb{R}^{m_{L-1}}$ denotes the output coefficient vector. The width of the network is defined as $W=\max\{m_0,m_1,\ldots,m_L\}$. With a slight abuse of notation, we define
the feature induced by the $j$-th neuron as $\phi(x;\theta_{\ell,j})\coloneqq\phi(\langle\theta_{\ell,j},f_{\ell-1}(x)\rangle)$.

\mypara{$\bm{P}$-Trainable Neural Networks.}
Let $P\in[0,1]$ denote the proportion of trainable neurons in each
hidden layer. For every $\ell\in[L-1]$, define $m_{\ell,\mathrm{free}}(P)=\lfloor Pm_\ell\rfloor$ and $m_{\ell,\mathrm{fix}}(P)=m_\ell-m_{\ell,\mathrm{free}}(P)$. When the dependence on $P$ is clear, write $m_{\ell,\mathrm{free}}$ and $m_{\ell,\mathrm{fix}}$ without their arguments. The hidden-layer widths $\{m_1,\ldots,m_{L-1}\}$ are fixed
when comparing different values of $P$. For indexing, the trainable neurons are listed before the randomly fixed neurons in each hidden layer. Specifically, \[ {\theta}_{\ell,j}\in\mathbb{R}^{m_{\ell-1}+1}, ~1\le j\le m_{\ell,\mathrm{free}},~\text{and}~\theta_{\ell,j}=\bar{\theta}_{\ell,j-m_{\ell,\mathrm{free}}}\sim\rho_{0,\ell},~m_{\ell,\mathrm{free}}+1\le j\le m_{\ell},\]
where ${\theta}_{\ell,j}$ denotes the parameter of a trainable neuron while $\bar{\theta}_{\ell,j-m_{\ell,\mathrm{free}}}$ denotes the parameter sampled from a prescribed distribution $\rho_{0,\ell}$. In the two-layer setting, write $\rho_0=\rho_{0,1}$.

With $\phi(x;\theta_{\ell,j})\coloneqq\phi(\langle\theta_{\ell,j},f^{(P)}_{\ell-1}(x)\rangle)$, the $P$-trainable network is defined by
\begin{equation} 
	\label{eq:p-trainable-neural-network} 
	\begin{aligned} 
		f_0^{(P)}(x)&=(x,1),\\f_\ell^{(P)}(x) &=\Big(\bigl( \phi(x;{\theta}_{\ell,j}) \bigr)_{j=1}^{m_{\ell,\mathrm{free}}},\bigl( \phi(x;\bar{\theta}_{\ell,j}) \bigr)_{j=1}^{m_{\ell,\mathrm{fix}}},1\Big),\quad\ell\in[L-1],\\ 
		f_L^{(P)}(x) &= \sum_{j=1}^{m_{L-1,\mathrm{free}}} \beta_j\phi(x;{\theta}_{L-1,j}) + \sum_{j=1}^{m_{L-1,\mathrm{fix}}} \beta_{m_{L-1,\mathrm{free}}+j}\phi(x;\bar{\theta}_{L-1,j}), 
	\end{aligned}
\end{equation}
where the output coefficient vector
$\beta=(\beta_1,\ldots,\beta_{m_{L-1}})=({\beta}_{\mathrm{free}},{\beta}_{\mathrm{fix}})$ is partitioned according to the trainable and fixed hidden neurons. All entries of $\beta$ remain trainable. For a given realization of the fixed hidden parameters, the resulting class of $P$-trainable neural networks is denoted by $\mathcal{F}_P(W,L)$. Figure \ref{fig:P-network-architecture} illustrates the architecture.

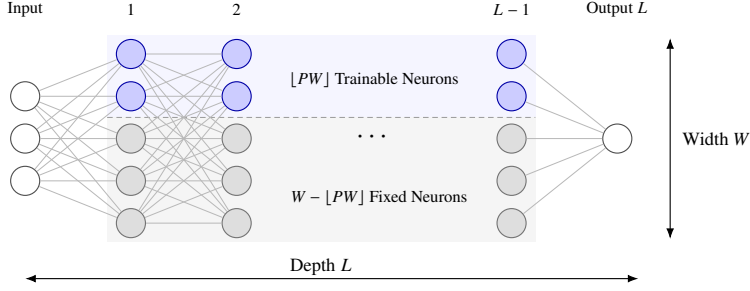
\begin{figure}[t]
	\centering
	\resizebox{0.8\linewidth}{!}{%
		\begin{tikzpicture}[
			>=Latex,
			neuron/.style={circle,draw,minimum size=5.5mm,inner sep=0pt,line width=.6pt},
			free/.style={neuron,draw=blue!65!black,fill=blue!20},
			fixed/.style={neuron,draw=black!55,fill=black!13},
			other/.style={neuron,draw=black!70,fill=white},
			connection/.style={draw=black!28,line width=.35pt}
			]
			\fill[blue!4] (1.55,.4) rectangle (9.65,1.95);
			\fill[black!4] (1.55,-1.95) rectangle (9.65,.4);
			\draw[densely dashed,black!40] (1.55,.4)--(9.65,.4);
			
			\foreach \i in {1,...,3}
			\node[other] (input-\i) at (0,{1.6-.8*\i}) {};
			\foreach \x/\name in {2/hone,4/htwo,9.2/hlast}{
				\foreach \j in {1,...,5}{
					\ifnum\j<3
					\node[free] (\name-\j) at (\x,{2.4-.8*\j}) {};
					\else
					\node[fixed] (\name-\j) at (\x,{2.4-.8*\j}) {};
					\fi
				}
			}
			\node[other] (output) at (11.2,0) {};
			
			\foreach \i in {1,...,3}
			\foreach \j in {1,...,5}
			\draw[connection] (input-\i)--(hone-\j);
			\foreach \i in {1,...,5}
			\foreach \j in {1,...,5}
			\draw[connection] (hone-\i)--(htwo-\j);
			\foreach \i in {1,...,5}
			\draw[connection] (hlast-\i)--(output);
			
			\node[inner sep=2pt,font=\Large] at (6.6,0) {$\cdots$};
			\node[font=\small] at (6.6,1.12)
			{$\lfloor PW\rfloor$ Trainable Neurons};
			\node[font=\small] at (6.7,-1.12)
			{$W-\lfloor PW\rfloor$ Fixed Neurons};
			
			\node[font=\small] at (0,2.45) {Input};
			\node[font=\small] at (2,2.45) {$1$};
			\node[font=\small] at (4,2.45) {$2$};
			\node[font=\small] at (9.2,2.45) {$L-1$};
			\node[font=\small] at (11.2,2.45) {Output $L$};
			
			\draw[<->,line width=.6pt] (12.2,1.9)--(12.2,-1.9)
			node[midway,right=3pt] {Width $W$};
			\draw[<->,line width=.6pt] (0,-2.65)--(11.6,-2.65);
			\node[above=0pt] at (5.6,-2.7) {Depth $L$};
		\end{tikzpicture}
	}
	\caption{A $P$-trainable network with hidden layers of equal width $W$. Each hidden layer contains $\lfloor PW\rfloor$ trainable neurons and $W-\lfloor PW\rfloor$ fixed neurons. Thus, the network has $\mathcal{O}(PW^2L+W)$ trainable parameters out of $\mathcal{O}(W^2L)$ total parameters.}
	\label{fig:P-network-architecture}
\end{figure}
Notably, there are two limiting cases. At $P=0$, all hidden parameters are randomly initialized and held fixed, so only the output coefficients are trained. Thus, $\mathcal{F}_0(W,L)$ corresponds to an RFM class. At $P=1$, all hidden parameters are trainable, so the $P$-trainable network class coincides with the fully trainable FNN class, and hence  $\mathcal{F}_1(W,L)=\mathcal{F}(W,L)$. 

Although the model \eqref{eq:p-trainable-neural-network} is defined for general depth $L$, the main approximation theory concerns two-layer networks in Sections \ref{sec:main-results} and \ref{sec:specific-activation-functions}. Section \ref{sec:three-layer-illustration} provides a structured three-layer illustration, and Section \ref{sec:numerical-illustration} examines deeper networks numerically.

\begin{remark}[Neuron-wise Trainability]
	Partial training can also be defined at the level of individual parameter entries. Here, each hidden neuron is assigned to the trainable or fixed group through its entire \textit{incoming} parameter vector $\theta_{\ell,j}$, including the bias. In deeper networks, the outputs of fixed neurons may still change as earlier trainable layers change. 
\end{remark}

\subsection{Approximation Task: Two Complementary Views}
For the two-layer analysis, set $m=m_1$ and assume that
$Pm$ is an integer, so that $m_{\mathrm{free}}=Pm$
and $m_{\mathrm{fix}}=(1-P)m$. Let $\mathcal{X}\subseteq\mathbb{R}^d$ be the input domain and let $f^*:\mathcal{X}\to\mathbb{R}$ be a target function. The approximation task by a two-layer $P$-trainable network is described from two complementary viewpoints. 

\mypara{Finite-dimensional View.} Given input points $\{x_{1},\ldots,x_{n}\}\subset\mathcal{X}$ and the corresponding target values $\{f^{*}(x_{i})\}_{i=1}^{n}$, the task is to choose the output coefficients and trainable hidden parameters. By \eqref{eq:p-trainable-neural-network}, a two-layer $P$-trainable network $f\in\mathcal{F}_{P}(W,2)$ has the form 
\[ f_2^{(P)}({x})=\sum_{j=1}^{m_{\mathrm{free}}} \beta_j\phi({x};{\theta}_{j}) + \sum_{j=1}^{m_{\mathrm{fix}}} \beta_{m_{\mathrm{free}}+j}\phi(x;\bar{\theta}_{j}). \]

With the randomly sampled hidden parameters held fixed during optimization, the expected best approximation error on these input points is
\begin{equation}
	\label{eq:finite-appro.-task}
	\mathcal{E}_{n,m}(f^{*},\mathcal{F}_P(W,2))=\EE_{\bar\theta_{1:m_{\mathrm{fix}}}}\Big[\inf_{\beta,\theta_{1:m_{\mathrm{free}}}}\frac{1}{n}\sum_{i=1}^{n}\big(f^{*}(x_{i})-f_{2}^{(P)}(x_{i})\big)^{2}\Big].
\end{equation}
When $P=0$, only the output coefficients $\beta$ are optimized, which gives a linear least-squares problem. When $m_{\mathrm{free}}>0$, jointly optimizing the output coefficients and trainable hidden parameters generally yields a nonlinear, non-convex problem.

\mypara{Population View.} Let $\mu$ be a probability measure on the input space $\mathcal{X}$. For any target function $f^{*}\in L_{2}(\mathcal{X},d\mu)$, its squared $L_2(d\mu)$ norm is defined as
\[ \norm{f^{*}}_{L_{2}(d\mu)}^{2}=\int_{\mathcal{X}}\abs{f^{*}(x)}^{2}d\mu(x). \]
Replacing the average in \eqref{eq:finite-appro.-task} with $\norm{\cdot}_{L_{2}(d\mu)}^{2}$ gives the population approximation problem
\begin{equation}
	\label{eq:population-appro.-task}
	\begin{aligned}
		&\mathcal{E}_{m}(f^{*},\mathcal{F}_P(W,2))=\EE_{\bar\theta_{1:m_{\mathrm{fix}}}}\Big[\inf_{\beta,\theta_{1:m_{\mathrm{free}}}}\norm{{f^{*}}-{f_{2}^{(P)}}}_{L_{2}(d\mu)}^{2}\Big].
	\end{aligned}
\end{equation}

The finite-dimensional problem \eqref{eq:finite-appro.-task} evaluates approximation on the prescribed input points, while the
population problem \eqref{eq:population-appro.-task} measures approximation in $L_{2}(d\mu)$. These formulations define the expected best approximation errors over the fixed initialization, while the subsequent results are stated as high-probability bounds over the sampled parameters. 

\subsection{Target Function Space: The Reproducing Kernel Hilbert Space}
Let $\mu$ be the probability measure on the input space $\mathcal X$ and $\rho$ on the parameter space $\Theta$.
\begin{restatable}[Operator Regularity]{assumption}{assumpoperatorRegularity}
	\label{assumption:operator}
	The feature map \(\phi\) is jointly measurable and belongs to \(L_2(\mathcal X\times\Theta,\mu\otimes\rho)\).
\end{restatable}
Under Assumption \ref{assumption:operator}, the feature family $\{\phi(\cdot;\theta): \theta\in\Theta\}$ induces the positive semi-definite kernel $k_{\rho}$ on a $\mu$ full measure subset\footnote{For notational simplicity, write $\mathcal X$ for this subset and leave the excluded set of $\mu$-measure zero implicit.} of $\mathcal{X}$,
\begin{equation} \label{eq:induced-kernel} k_{\rho}(x,x') \coloneq \int_{\Theta} \phi(x;\theta)\phi(x';\theta)\,d\rho(\theta) = \mathbb{E}_{\theta\sim\rho} [ \phi(x;\theta)\phi(x';\theta)]. 
\end{equation}
Let $\mathcal{H}_\rho\subseteq L_2(\mathcal{X},d\mu)$ denote the reproducing kernel Hilbert space (RKHS) induced by $k_{\rho}$. We assume that the target $f^*$ in the main analysis belongs to the RKHS $\mathcal{H}_{\rho}$.

\mypara{Integral Operator Induced by the Feature Map.} 
To compare finite sampled features with the integral representation of targets in $\mathcal H_\rho$, define the feature operator $S_{\rho}:L_2(\Theta, d\rho)\to L_2(\mathcal{X}, d\mu)$ and its adjoint $S_{\rho}^{*}:L_2(\mathcal{X}, d\mu)\to L_2(\Theta, d\rho)$ by 
\begin{equation}
	\nonumber
	(S_{\rho}a)(x)\coloneq
	\int_{\Theta}
	a(\theta)\phi(x; \theta)d\rho(\theta)~ \text{and}~(S_{\rho}^{*}g)(\theta)=\langle \phi(\cdot;\theta),g(\cdot)\rangle_{L_2(d\mu)}.
\end{equation}
Here $a\in L_2(\Theta, d\rho)$ is a coefficient function and $g\in L_2(\mathcal{X}, d\mu)$ is an arbitrary square integrable function. The integral operator associated with $k_{\rho}$ is defined as 
\[
T_{\rho}\coloneq S_{\rho}S_{\rho}^{*}: L_2(\mathcal{X}, d\mu)\to L_2(\mathcal{X}, d\mu).
\]
Equivalently, for any $g\in L_2(\mathcal{X}, d\mu)$,
\begin{equation}
	\label{eq:kernel-integral-operator}
	(T_{\rho}g)(x)
	=\int_{\mathcal{X}}k_{\rho}(x,x')g(x')\,d\mu(x')=\int_{\Theta}
	\phi(x;\theta)\langle\phi(\cdot;\theta),g(\cdot)\rangle_{L_2(d\mu)}d\rho(\theta).
\end{equation}
Therefore, $T_{\rho}=\mathbb{E}_{\theta\sim\rho}
[\phi(\cdot;\theta)\otimes_{L_2(d\mu)}\phi(\cdot;\theta)]$, where $a\otimes_{L_2(d\mu)} b: L_{2}(\mathcal{X}, d\mu)\to L_{2}(\mathcal{X}, d\mu)$ is an operator such that $(a\otimes_{L_2(d\mu)} b)g=\langle b,g\rangle_{L_{2}(d\mu)} a$.

\mypara{Spectral Framework.} Under Assumption \ref{assumption:operator}, $S_{\rho}$ is a Hilbert-Schmidt operator and hence compact. Thus $T_{\rho}=S_{\rho}S_{\rho}^{*}$ is self-adjoint, positive semi-definite, trace-class and compact \cite{simon1979trace}. By the spectral theorem for compact self-adjoint operators \cite{konig1986eigenvalues}, the set of positive eigenvalues $\{\eta_k\}_{k\geq1}$ of $T_{\rho}$ is finite or countably infinite. These eigenvalues can be arranged in non-increasing order and counted with multiplicity, $\eta_1\ge\eta_2\ge\cdots>0$. The corresponding eigenfunctions $\{e_k\}_{k\geq1}$ form an orthonormal basis of $(\ker T_{\rho})^{\perp}$. Therefore, the integral operator $T_{\rho}$ and its positive fractional powers have the following expansions,
\begin{equation}
	\label{eq:spectral-decomposition}
	T_{\rho}=\sum_{k\geq1}\eta_k(e_k\otimes e_k),~\text{and}~T_{\rho}^{s}
	=\sum_{k\geq1}\eta_k^{s}(e_k\otimes e_k),~ s>0.
\end{equation}
In the population analysis, RKHS functions that agree $\mu$-almost everywhere are identified, and $\mathcal H_\rho$ is equipped with the norm defined below. Under this convention, $\mathcal{H}_{\rho}=\operatorname{Range}(T_{\rho}^{1/2})=\operatorname{Range}(S_{\rho})$. Consequently, for every target $f^{*}\in\mathcal{H}_{\rho}$, there exists a unique $q_f^{*}\in(\ker T_{\rho})^{\perp}$ and a unique $a_{f}^{*}\in (\ker S_{\rho})^{\perp}$ such that
\begin{equation}
	\label{eq:target-function-statement}
	f^*=T_{\rho}^{1/2}q_f^{*}=S_{\rho}a_{f}^{*}=\int_{\Theta}a_{f}^{*}(\theta)\phi(\cdot;\theta)d\rho(\theta),
\end{equation}
These equalities hold in $L_{2}(\mathcal{X},d\mu)$, with the norm defined by
\[ \|f^*\|_{\mathcal{H}_{\rho}}
=\inf_{\substack{a\in L_2(d\rho)\\S_\rho a=f^*\ \mathrm{in}\ L_2(d\mu)}}
\|a\|_{L_2(d\rho)}=\|q_f^{*}\|_{L_2(d\mu)}=\norm{a_{f}^{*}}_{L_2(d\rho)}. \]
This representation underlies subsequent analysis of the population approximation error.
\subsection{Distributional Approximation Strategy}
\label{subsec:distributional-approximation-strategy}
Although the trainable parameters can be chosen freely, the direct optimization over these parameters is generally nonlinear and non-convex. To facilitate the approximation analysis, all hidden-layer parameters are constructed through sampling. 
\begin{definition}[Two-group Sampling]
	\label{def:sampling-scheme}
	Let $\rho$ be the reference probability measure defining the target RKHS. Let $\rho_0$ be the prescribed initialization distribution for the fixed neurons and $\rho_1$ an adjustable distribution used to construct the trainable hidden parameters. Assume that both distributions are absolutely continuous with respect to $\rho$, and write
	\[d\rho_i(\theta)=\gamma_i(\theta)d\rho(\theta),i\in\{0,1\}, \]
	where $\gamma_i$ is a nonnegative measurable Radon-Nikodym density satisfying $\int_{\Theta}\gamma_i(\theta)d\rho(\theta)=1$ for each $i\in\{0,1\}$. Given fixed group sizes $m_{\mathrm{free}}$ and $m_{\mathrm{fix}}$, set $m=m_{\mathrm{free}}+m_{\mathrm{fix}}\geq1$ and the trainable proportion $P=m_{\mathrm{free}}/m$. The hidden parameters are sampled according to
	\begin{equation}
		\label{eq:sampling-scheme}
		{\theta}_{j}\stackrel{\mathrm{i.i.d.}}{\sim}\rho_{1}, j\in[m_{\mathrm{free}}],~\text{and}~\bar{\theta}_{j}\stackrel{\mathrm{i.i.d.}}{\sim}\rho_{0}, j\in[m_{\mathrm{fix}}].
	\end{equation}
	 The two groups $\{\theta_{j}\}_{j=1}^{m_{\mathrm{free}}}$ and $\{\bar\theta_{j}\}_{j=1}^{m_{\mathrm{fix}}}$ are mutually independent and have fixed sizes. Their associated mixture measure is $\rho_{P}=P\rho_{1}+(1-P)\rho_{0}$, satisfying
	 \begin{equation}
	 	\label{eq:distributional-constr.}
	 	d\rho_{P}(\theta)=\gamma_{P}(\theta)d\rho(\theta),\text{ with density } \gamma_{P}(\theta)=P\gamma_{1}(\theta)+(1-P)\gamma_{0}(\theta).
	 \end{equation}
\end{definition}

For the approximation analysis, assume that each $\gamma_i,i\in\{0,1\}$ is positive for $\rho$-almost every $\theta$ satisfying $\|\phi(\cdot;\theta)\|_{L_2(d\mu)}>0$. This ensures both sampling distributions retain the nonzero features used to define $T_\rho$, while allowing zero densities where the feature vanishes.

The two-group sampling in Definition \ref{def:sampling-scheme} differs from i.i.d. sampling directly from $\rho_{P}$. The former uses fixed group sizes while the latter would randomize the group counts. Both sampling
distributions remain fixed within each construction. Every realized construction belongs to the $P$-trainable network class $\mathcal{F}_{P}(W,2)$, so its approximation error provides an upper bound on the best approximation error over this network class. Besides, the fixed initialization distribution $\rho_0$ is prescribed, while the choice of the adjustable distribution $\rho_1$
is discussed in the next section. 

By \eqref{eq:population-appro.-task}, optimizing only the output coefficients and taking expectations over both groups thus yields
\begin{equation}
	\label{eq:surrogate-approximation-task}
	\begin{aligned}
		&\mathcal{E}_{m}(f^{*},\mathcal{F}_P(W,2))\\
		&\qquad\le \EE_{\theta_{1:m_{\mathrm{free}}},\bar\theta_{1:m_{\mathrm{fix}}}}\Big[\inf_{\beta}\Big\Vert{f^{*}}  -\sum_{j=1}^{m_{\mathrm{free}}}\beta_j\phi(\cdot;\theta_j)-\sum_{j=1}^{m_{\mathrm{fix}}}\beta_{m_{\mathrm{free}}+j}\phi(\cdot;\bar\theta_j)\Big\Vert_{L_{2}(d\mu)}^{2}\Big].
	\end{aligned}
\end{equation}

\mypara{Ridge Regression.} Conditional on the sampled hidden features, the network output is linear in $\beta$. We analyze \eqref{eq:surrogate-approximation-task} by choosing the output coefficients $\beta$ by population ridge regression. The objective is specified in the next section.

\begin{remark}[Analytical Construction]
	The distributional construction is used only for the approximation analysis and does not require the practical training algorithm to optimize explicitly over probability distributions.
\end{remark}
	\section{General Approximation Results}
\label{sec:main-results}
This section establishes a general approximation bound for two-layer $P$-trainable networks with targets in $\mathcal H_\rho$. Under polynomial spectral decay and leverage compatibility assumptions, this bound yields an explicit rate and quantifies the phase transition as $P$ varies. Note that $T_{\rho}=0$ is trivial for approximation results, since $\mathcal{H}_{\rho}=\{0\}$. Henceforth, assume $T_{\rho}\neq 0$.

\subsection{Approximation Bounds for General Feature Maps}
\label{subsec:approxi-bounds-general-feature-maps}
Following the importance reweighting construction in \cite{bach2017equivalence}, define $\phi_P(x;\theta)=\gamma_P^{-1/2}(\theta)\phi(x;\theta)$ when $\|\phi(\cdot;\theta)\|_{L_2(d\mu)}>0$, where $\gamma_P$ is defined in \eqref{eq:distributional-constr.}. For parameters inducing the zero function in $L_2(d\mu)$, set $\phi_P(\cdot;\theta)=0$. The following lemma shows that this reweighting preserves the kernel and its integral operator.
\begin{lemma}[Invariance of Kernel and Operator]
	Under Assumption \ref{assumption:operator}, let the two groups be sampled according to \eqref{eq:sampling-scheme}, with $P=m_{\mathrm{free}}/m$. The reweighted features $\phi_P$ satisfy the following properties.
	\begin{enumerate}[labelindent=0pt, leftmargin=*,topsep=0pt,partopsep=0pt]
		\item The following empirical kernel is unbiased for $(\mu\otimes\mu)$-almost every
		$(x,x')\in\mathcal{X}\times\mathcal{X}$ \[ \hat{k}_{\rho}^{(P)}(x,x^{\prime})=\frac{1}{m}\Big[\sum_{j=1}^{m_{\mathrm{free}}}\phi_{P}(x;\theta_{j})\phi_{P}(x^{\prime};\theta_{j})+\sum_{j=1}^{m_{\mathrm{fix}}}\phi_{P}(x;\bar\theta_{j})\phi_{P}(x^{\prime};\bar\theta_{j})\Big], \] with $\EE\big[\hat{k}_{\rho}^{(P)}(x,x^{\prime})\big]=k_{\rho}(x,x^{\prime})= \int_{\Theta}\phi_{P}(x;\theta)\phi_{P}(x^{\prime};\theta)d\rho_{P}(\theta). $
		\item The integral operator induced by $\phi_{P}$ is still $T_{\rho}$ and the original eigenpairs
		$\{(\eta_k,e_k)\}_{k\ge1}$ remain unchanged, satisfying $T_{\rho}=\EE_{\theta\sim\rho_{P}}[\phi_{P}(\cdot;\theta)\otimes\phi_{P}(\cdot;\theta)]=\sum_{k\ge1}\eta_k(e_k\otimes e_k).$
	\end{enumerate}
	\label{lemma:kernel-operator-invariance}
\end{lemma}
Lemma \ref{lemma:kernel-operator-invariance} ensures that the reweighted features preserve $T_\rho$. Using these sampled features, the output coefficients are chosen by population ridge regression. For any $\lambda>0$,
\begin{equation}
	\label{eq:the-real-ridge-regression-task}
	\hat{\beta}_{\lambda}=\underset{\beta\in\RR^{m}}{\arg\min}~ \Bigg\{\Big\Vert f^{*}-\sum_{j=1}^{m_{\mathrm{free}}}\beta_{j}\phi_{P}(\cdot;\theta_{j})-\sum_{j=1}^{m_{\mathrm{fix}}}\beta_{m_{\mathrm{free}}+j}\phi_{P}(\cdot;\bar\theta_{j})\Big\Vert_{L_{2}(d\mu)}^{2}+\lambda m \norm{\beta}_{2}^{2}\Bigg\}.
\end{equation}

For each realization, reweighting rescales the output coefficients and preserves the span of the sampled
features. To analyze the approximation error of this ridge construction, the size of each reweighted feature is measured relative to the regularized operator $T_\rho+\lambda I$. Specifically, define the ridge leverage score
\begin{equation}
	\label{eq:leverage-score}
	s_{\lambda,P}(\theta)= \langle
		\phi_{P}(\cdot;\theta),(T_\rho+\lambda I)^{-1}\phi_{P}(\cdot;\theta)
		 \rangle_{L_{2}(d\mu)},
\end{equation}
where $\phi_{P}=\gamma_{P}^{-1/2}\phi$, with the convention that $s_{\lambda,P}(\theta)=0$ whenever $\|\phi(\cdot;\theta)\|_{L_2(d\mu)}=0$. The maximal ridge leverage value under the sampling
distribution $\rho_P$ is
\begin{equation}
	\label{eq:definition-of-d-max}
	d_{\max}(\gamma_{P},\lambda)
	=\operatorname*{ess\,sup}_{\theta\sim\rho_{P}}s_{\lambda,P}(\theta).
\end{equation}

The following theorem uses the regularization scale $\lambda=\varepsilon/4$ to give a sufficient width condition for squared approximation error at most $\varepsilon$ with confidence level $1-\delta$.
\begin{theorem}[General Approximation Bounds]
	\label{thm:general-approx.-bounds} Under Assumption \ref{assumption:operator}, let the error tolerance $\varepsilon>0$, $m_{\mathrm{free}}, m_{\mathrm{fix}}\in\NN$ with $m=m_{\mathrm{free}}+m_{\mathrm{fix}}\ge 1$, and the trainable proportion $P=m_{\mathrm{free}}/m \in[0,1]$. Assume that $d_{\max}(\gamma_{P},\varepsilon/4)<\infty$. Let the hidden parameters be sampled according to
	Definition \ref{def:sampling-scheme}. For any $\delta\in(0,1)$, if 
	\begin{equation}
		\label{eq:sample-condition}
		m\ge \frac{40}{9}d_{\max}(\gamma_{P},\varepsilon/4)\log\frac{14d_{\max}(\gamma_{P},\varepsilon/4)}{\delta},
	\end{equation}
	then, with probability at least $1-\delta$, 
	\begin{equation}
		\label{eq:main-result-GAB}
		\sup_{\norm{f^{*}}_{\mathcal H_\rho}\le 1}
		\inf_{\norm{\beta}_{2}^{2}\le\frac{4}{m}}
		\Big\Vert f^{*}-\sum_{j=1}^{m_{\mathrm{free}}}\beta_{j}\phi_{P}(\cdot;\theta_{j})-\sum_{j=1}^{m_{\mathrm{fix}}}\beta_{m_{\mathrm{free}}+j}\phi_{P}(\cdot;\bar\theta_{j})\Big\Vert_{L_{2}(d\mu)}^{2}\le \varepsilon.
	\end{equation}
\end{theorem}
Theorem \ref{thm:general-approx.-bounds} reveals that a single realization of the sampled hidden parameters supports the general approximation bounds uniformly over the unit ball of $\mathcal{H}_{\rho}$, with output coefficients chosen separately for each target function. The constraint $\norm{\beta}_2^2\le 4/m$ controls the output coefficients in the reweighted representation in \eqref{eq:the-real-ridge-regression-task}.

The width condition \eqref{eq:sample-condition} depends on the sampling scheme through $d_{\max}(\gamma_{P},\varepsilon/4)$. To further understand what governs the approximation capabilities of $P$-trainable networks, the next subsection examines how $d_{\max}(\gamma_{P},\varepsilon/4)$ depends on the trainable proportion $P$, the sampling distributions $\rho_{0}$ and $\rho_{1}$, and the regularization scale $\lambda$.

\subsection{Approximation Rates under Polynomial Spectral Decay}
\label{subsec:Poly-Rates}
The quantity $d_{\max}(\gamma_{P},\lambda)$ can be written as the product of a spectral complexity of $T_\rho$ and a sampling mismatch factor. Specifically, the spectral complexity is measured by the effective dimension, defined for $\lambda>0$ as
\begin{equation}
	\label{eq:effective-dimension}
	d(\lambda)=\int_\Theta s_{\lambda,P}(\theta)d\rho_P(\theta)=\operatorname{tr}(T_\rho(T_\rho+\lambda I)^{-1}).
\end{equation}
Since reweighting preserves $T_{\rho}$, the effective dimension depends only on $T_{\rho}$ and $\lambda$, independent of the trainable proportion $P$ and the sampling distributions $\rho_{0}$ and $\rho_{1}$.

For $d(\lambda)>0$, define the mismatch factor of the sampling distribution $\rho_{P}$ by
\begin{equation}
	\label{eq:definition-of-kappa-first}
	\kappa_{\gamma_{P}}(\lambda)=\frac{d_{\max}(\gamma_{P},\lambda)}{d(\lambda)}=\frac{\operatorname*{ess\,sup}_{\theta\sim\rho_{P}}{s_{\lambda,P}(\theta)}}{\EE_{\theta\sim\rho_{P}}[s_{\lambda,P}(\theta)]}\ge 1.
\end{equation}
The mismatch factor $\kappa_{\gamma_{P}}(\lambda)$ compares the maximal score with its sampling average. Since the sufficient width condition \eqref{eq:sample-condition} in Theorem \ref{thm:general-approx.-bounds} is controlled by $d_{\max}(\gamma_P,\lambda)$, a natural sampling benchmark is a density minimizing this maximal ridge leverage for fixed $T_\rho$ and $\lambda$. Such a benchmark is provided by the optimal ridge leverage density with respect to $\rho$,
\begin{equation}
	\label{eq:leverage-score-distribution}
	\gamma_{\lambda}^{*}(\theta)=\frac{\gamma_{P}(\theta)s_{\lambda,P}(\theta)}{d(\lambda)}=\frac{ \langle
		\phi(\cdot;\theta),(T_\rho+\lambda I)^{-1}\phi(\cdot;\theta)
		\rangle_{L_{2}(d\mu)}}{d(\lambda)}.
\end{equation}
By construction, we have $\int_{\Theta}\gamma_{\lambda}^{*}(\theta)d\rho(\theta)=1$, $d_{\max}(\gamma_{\lambda}^{*},\lambda)=d(\lambda)$, and $\kappa_{\gamma_{\lambda}^{*}}(\lambda)=1$. Since $d_{\max}(\gamma,\lambda)\ge d(\lambda)$ for any admissible density $\gamma$, the density $\gamma_{\lambda}^{*}$ gives the smallest sufficient width prescribed by \eqref{eq:sample-condition} for uniform approximation over the unit ball of the target RKHS $\mathcal H_\rho$. It therefore provides an optimal sampling benchmark adapted to the entire target function class. Moreover, the mismatch factor admits the density representation
\[
\kappa_{\gamma_P}(\lambda)
=\operatorname*{ess\,sup}_{\theta\sim\rho_P}\frac{\gamma_\lambda^*(\theta)}{\gamma_P(\theta)}.
\]
Thus, $\kappa_{\gamma_P}(\lambda)$ measures the sampling mismatch through the largest local ratio. Values close to $1$ indicate that the current sampling density closely matches the ridge leverage benchmark. Large values indicate that some parameter regions receive insufficient sampling mass. Figure \ref{fig:sampling-mismatch-22} illustrates these two cases.  
\begin{figure}[!t]
	\centering
	\subfloat[Well-balanced sampling with small $\kappa_{\gamma_P}(\lambda)$.]{%
		\includegraphics[width=0.47\linewidth]{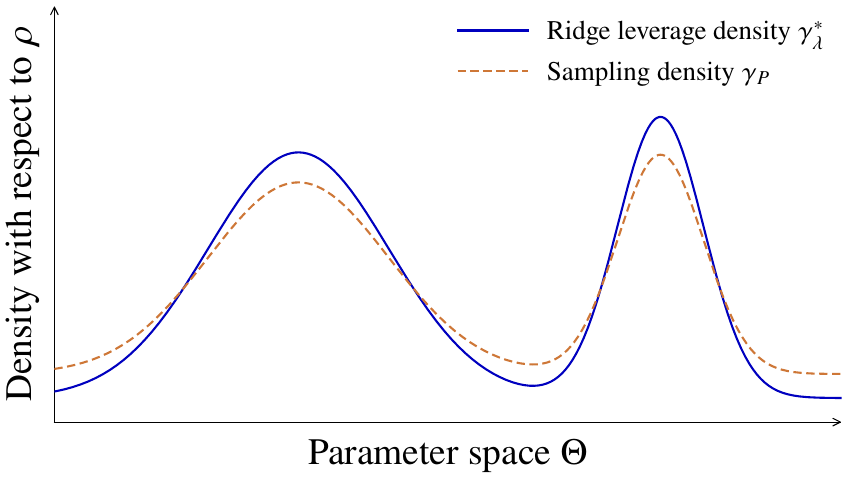}%
	}\hfill
	\subfloat[Local undersampling with large $\kappa_{\gamma_P}(\lambda)$.]{%
		\includegraphics[width=0.47\linewidth]{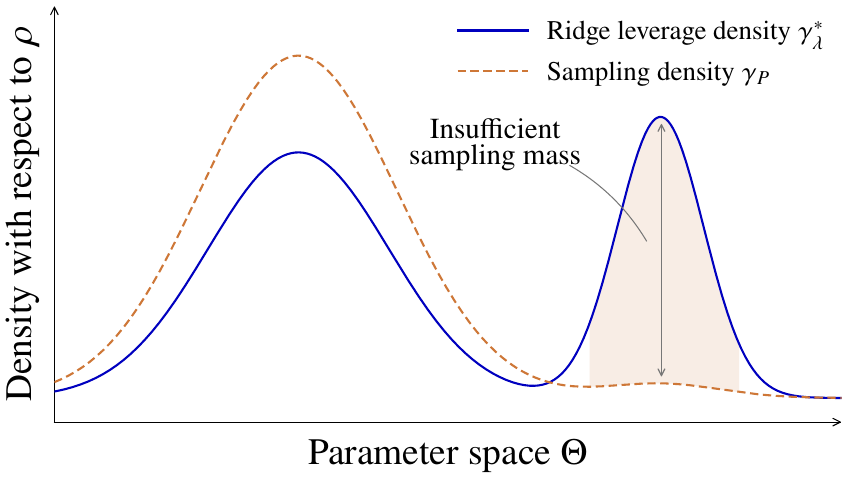}%
	}
	\caption{Schematic illustration of the mismatch factor $\kappa_{\gamma_P}(\lambda)$. The sampling density closely follows the ridge leverage density in (a), but assigns insufficient mass near the right peak in (b).}
	\label{fig:sampling-mismatch-22}
\end{figure}

Since the dependence on the trainable proportion $P$ enters through the mixture density $\gamma_{P}=P\gamma_{1}+(1-P)\gamma_{0}$, the following lemma bounds $d_{\max}(\gamma_P,\lambda)$ in terms of $P$ and the mismatch factors of the two component sampling densities.
\begin{lemma}[Mixture Leverage Bound]
	\label{lemma:maximal-leverage-bound}
	Let the mixture density be  $\gamma_{P}=P\gamma_{1}+(1-P)\gamma_{0}$ with $P\in[0,1]$. For any $\lambda>0$ and $\kappa_{\gamma_{i}}(\lambda)<\infty, i\in\{0,1\}$, we have
	\begin{equation}
		\label{eq:bound-dmax-kappa(lamda)}
		d_{\max}(\gamma_{P},\lambda)\le \Big(\frac{P}{\kappa_{\gamma_{1}}(\lambda)}+\frac{1-P}{\kappa_{\gamma_{0}}(\lambda)} \Big)^{-1}d(\lambda).
	\end{equation} 
\end{lemma}
To derive explicit approximation rates, the next assumption further controls these two factors over the relevant regularization scales.

\begin{restatable}[Leverage Compatibility]{assumption}{assumQuality}
	\label{assumption:mismatch-used}
	There exist $\lambda_{0}>0$ and constants $1\le\kappa_1^{*}\le\kappa_0^{*}<\infty$ such that, for each $0<\lambda\le\lambda_0$, the sampling densities satisfy $\kappa_{\gamma_{i}}(\lambda)\le\kappa_i^{*}$ for $i\in\{0,1\}$. Here, $\gamma_0$ is prescribed, while $\gamma_1$ may depend on $\lambda$ and is fixed in each sampling construction.
\end{restatable}
Together with the mixture leverage bound in Lemma \ref{lemma:maximal-leverage-bound}, this assumption controls the effect of mixture sampling through the constants $\kappa_1^*$ and $\kappa_0^*$. The remaining step is to determine how the effective dimension $d(\lambda)$ grows as $\lambda$ decreases.

\begin{restatable}[Polynomial Spectral Decay]{assumption}{assumpPoly}
	\label{assumption:eigendecay-used}
	There exist constants $c_-,c_+>0$ and $\alpha>1$ such that
	$c_-k^{-\alpha}\le\eta_k\le c_+k^{-\alpha}$ for every $k\ge1$.
\end{restatable}
This assumption specifies the eigenvalue decay of the integral operator $T_{\rho}$ and implies $d(\lambda)\asymp\lambda^{-1/\alpha}$ for $0<\lambda\le\lambda_0$, as shown in the proof of Theorem \ref{thm:polynomial-rates}. Together with the mixture leverage bound in Lemma \ref{lemma:maximal-leverage-bound}, Assumptions \ref{assumption:mismatch-used} and \ref{assumption:eigendecay-used} are used to derive explicit approximation rates. Further interpretation of assumptions is discussed in Appendix \ref{subsec:assumption-discussions}. 
\begin{theorem}[Polynomial Approximation Rates]
	\label{thm:polynomial-rates}
	Under Assumptions \ref{assumption:operator}, \ref{assumption:mismatch-used} and \ref{assumption:eigendecay-used}, let $m\in\NN_{+}$ and $P\in[0,1]$ satisfy $Pm\in\NN$, and set $W=\max\{d,m,1\}$. For any $\delta\in(0,1)$, under the sampling scheme in Definition \ref{def:sampling-scheme}, the rate holds with probability at least $1-\delta$,
	\begin{equation}
		\label{eq:poly-rates-main-theorem}
		\sup_{\norm{f}_{\mathcal H_\rho}\le 1}\inf_{f_{2}^{(P)}\in\mathcal{F}_P(W,2)}\norm{f-f_{2}^{(P)}}_{L_{2}(d\mu)}^{2}\lesssim \Big( \Big(\frac{P}{\kappa_1^{*}}+\frac{1-P}{\kappa_{0}^{*}} \Big)\frac{9m}{40\log(63m/(20\delta))} \Big)^{-\alpha},
	\end{equation}
	where the implicit constant is independent of $m, P$ and $\delta$.	 
\end{theorem}
The stated rate \eqref{eq:poly-rates-main-theorem} does not require the regularization scale to satisfy $\lambda\le\lambda_0$. When $\lambda>\lambda_0$, the zero approximant used in the proof gives a squared error bounded by a constant times $\lambda$. The same rate therefore follows after adjusting the implicit constant, without an additional restriction on $m$. Accordingly, the adjustable density $\gamma_1$ for the trainable group is selected at the scale used in the proof of Theorem \ref{thm:polynomial-rates}.

By scaling, the result extends to every target $f\in\mathcal H_\rho$,
\[ \inf_{f_{2}^{(P)}\in\mathcal{F}_P(W,2)}\norm{f-f_{2}^{(P)}}_{L_{2}(d\mu)}^{2}\lesssim \Big( \Big(\frac{P}{\kappa_1^{*}}+\frac{1-P}{\kappa_{0}^{*}} \Big)\frac{9m}{40\log(63m/(20\delta))} \Big)^{-\alpha}
\norm{f}_{\mathcal{H}_{\rho}}^{2}. \]
The reference measure $\rho$ therefore affects the bound through its kernel spectrum, the associated mismatch constants, and the RKHS norm of the target function.

For a fixed reference measure $\rho$, Theorem \ref{thm:polynomial-rates} provides an explicit approximation rate in terms of the trainable proportion $P$, the hidden-layer width $m$, and the spectral decay exponent $\alpha$. For fixed $P,\delta$ and mismatch constants, the bound decreases at the rate of  $\mathcal{O}(({m}/{\log m})^{-\alpha})$ as $m\to\infty$. For fixed $m$ and $\delta$, the dependence on $P$ affects the bound through the prefactor $(P/\kappa_{1}^{*}+(1-P)/\kappa_{0}^{*})^{-\alpha}$, which is non-increasing in $P$ since $\kappa_1^*\le \kappa_0^*$. Choosing $\gamma_1=\gamma_\lambda^*$ at each regularization choice gives $\kappa_1^*=1$. More generally, when $\kappa_1^*\ll\kappa_0^*$, a substantial reduction in the bound can occur even for small $P$, as quantified in the subsequent phase transition Corollary \ref{cor:phase-transition}.

\begin{remark}[Probability Statement When $P=1$]
	At $P=1$, the fully trainable network class is deterministic, but the proof still constructs the trainable parameters by sampling
	from $\rho_1$. The failure probability $\delta$ refers to this sampled construction. A successful realization yields an approximation bound over the deterministic network class. The high-probability formulation is retained for consistency across admissible $P\in[0,1]$.
\end{remark}

Substituting $P=m_{\mathrm{free}}/m$ into Theorem \ref{thm:polynomial-rates} gives an equivalent approximation bound in terms of the numbers of trainable and fixed neurons.
\begin{corollary}[Rate by Neuron Counts]
	\label{cor:Rate-of-neuron-number}
	Under the same assumptions and sampling scheme, let $m_{\mathrm{free}}, m_{\mathrm{fix}}\in\NN$ with $m=m_{\mathrm{free}}+m_{\mathrm{fix}}\ge 1$, and set $P=m_{\mathrm{free}}/m$. For $\delta\in(0,1)$, with probability at least $1-\delta$, 
	\begin{equation}
		\nonumber
		\sup_{\norm{f}_{\mathcal H_\rho}\le 1}\inf_{f_{2}^{(P)}\in\mathcal{F}_P(W,2)}\norm{f-f_{2}^{(P)}}_{L_{2}(d\mu)}^{2}\lesssim \Big( \Big(\frac{m_{\mathrm{free}}}{\kappa_1^{*}}+\frac{m_{\mathrm{fix}}}{\kappa_{0}^{*}} \Big)\frac{9}{40\log(63m/(20\delta))} \Big)^{-\alpha}.
	\end{equation}
\end{corollary}
The two neuron groups enter the bound through the weighted counts $m_{\mathrm{free}}/\kappa_1^*$ and $m_{\mathrm{fix}}/\kappa_0^*$. Their contributions are comparable when $m_{\mathrm{free}}\asymp m_{\mathrm{fix}}
(\kappa_1^*/\kappa_0^*)$. Thus, when $\kappa_1^*\ll\kappa_0^*$, a small trainable group can substantially reduce the bound. The reciprocal mismatch ratio $\kappa_1^*/\kappa_0^*$ determines the scale of the relative trainable count, while the hidden-layer width $m$ still affects the overall
approximation accuracy guaranteed by the bound.

\begin{remark}[Influence of Input Dimension]
	The input dimension $d$ affects the approximation bounds through the spectral decay exponent $\alpha$ and the mismatch constants $\kappa_0^*$ and $\kappa_1^*$. Section \ref{sec:specific-activation-functions} makes this dependence explicit for the spherical
	biased-ReLU features.
\end{remark}

\subsection{Small-$P$ Phase Transition}
To quantify the effect of $P$ on the approximation rate of two-layer $P$-trainable networks, fix the hidden-layer width $m$ and the confidence level $1-\delta$, and let $\mathcal{E}(P)$ denote the rate expression on the right-hand side of \eqref{eq:poly-rates-main-theorem}. For $\kappa_1^*<\kappa_0^*$, define the normalized gain by
\begin{equation}
	\label{eq:the-normalized-gain}
	G(P)=\frac{\mathcal{E}(0)-\mathcal{E}(P)}{\mathcal{E}(0)-\mathcal{E}(1)}.
\end{equation}
This quantity $G(P)$ measures the fraction of the total bound reduction between $P=0$ and $P=1$ attained at trainable proportion $P$.
\begin{corollary}[Phase Transition Improvement]
	Under the conditions of Theorem \ref{thm:polynomial-rates}, let $r={\kappa_0^{*}}/{\kappa_1^{*}}>1$. Then the normalized gain satisfies
	\begin{equation}
		\label{eq:normalized-gain}
		G(P)=\frac{1-\bigl(1+(r-1)P\bigr)^{-\alpha}}{1-r^{-\alpha}},~\text{for}~ P\in[0,1].
	\end{equation}
	Moreover, $G(0)=0, G(1)=1, G'(P)>0$, and $G''(P)<0$. Therefore, increasing the trainable proportion $P$ monotonically reduces the approximation bound, with diminishing marginal
	improvement.
	\label{cor:phase-transition}
\end{corollary}
The normalized gain depends only on the mismatch ratio $r$ and the spectral exponent $\alpha$. This dependence also quantifies how a small proportion of trainable neurons can capture most of the total improvement. For any tolerance level $\varepsilon\in(0,1)$, define
\[
P_{\varepsilon}(r)=\frac{ (\varepsilon+(1-\varepsilon)r^{-\alpha} )^{-1/\alpha}-1}{r-1},
\]
which is the smallest trainable proportion required to achieve the normalized gain of $1-\varepsilon$,
\[ G(P)\ge G(P_{\varepsilon})= 1-\varepsilon~\Longleftrightarrow~P\ge P_{\varepsilon}. \]

For a fixed width $m$, the corresponding minimum trainable count is $\lceil mP_\varepsilon(r)\rceil$. Furthermore, $P_{\varepsilon}\le{(\varepsilon^{-1/\alpha}-1)}/{(r-1)}$. For fixed $\varepsilon$ and $\alpha$, the threshold of $P_{\varepsilon}$ therefore decreases at the rate $\mathcal O(r^{-1})$ as $r$ increases. This quantifies the phase transition phenomenon in the bound: when the trainable distribution has a substantially smaller mismatch constant than the fixed initialization distribution ($\kappa_{1}^{*}\ll\kappa_{0}^{*}$), a small trainable proportion can capture most of the total reduction available between $P=0$ and $P=1$. 
\begin{figure}[!t]
	\centering
	\resizebox{0.6\linewidth}{!}{%
		\begin{tikzpicture}[>=Latex,font=\small]
			\fill[blue!9] (0,3)--(6.2,3)--
			(6.2000,0.0215497)--(6.1750,0.0220533)--(6.1500,0.0225687)--(6.1250,0.0230962)--(6.1000,0.0236360)--
			(6.0750,0.0241885)--(6.0500,0.0247540)--(6.0250,0.0253328)--(6.0000,0.0259251)--(5.9750,0.0265313)--
			(5.9500,0.0271517)--(5.9250,0.0277867)--(5.9000,0.0284366)--(5.8750,0.0291018)--(5.8500,0.0297826)--
			(5.8250,0.0304794)--(5.8000,0.0311926)--(5.7750,0.0319225)--(5.7500,0.0326695)--(5.7250,0.0334342)--
			(5.7000,0.0342168)--(5.6750,0.0350178)--(5.6500,0.0358377)--(5.6250,0.0366768)--(5.6000,0.0375357)--
			(5.5750,0.0384148)--(5.5500,0.0393146)--(5.5250,0.0402356)--(5.5000,0.0411782)--(5.4750,0.0421431)--
			(5.4500,0.0431307)--(5.4250,0.0441416)--(5.4000,0.0451763)--(5.3750,0.0462354)--(5.3500,0.0473195)--
			(5.3250,0.0484292)--(5.3000,0.0495650)--(5.2750,0.0507277)--(5.2500,0.0519178)--(5.2250,0.0531359)--
			(5.2000,0.0543829)--(5.1750,0.0556593)--(5.1500,0.0569659)--(5.1250,0.0583034)--(5.1000,0.0596725)--
			(5.0750,0.0610739)--(5.0500,0.0625086)--(5.0250,0.0639771)--(5.0000,0.0654805)--(4.9750,0.0670194)--
			(4.9500,0.0685948)--(4.9250,0.0702075)--(4.9000,0.0718584)--(4.8750,0.0735484)--(4.8500,0.0752786)--
			(4.8250,0.0770497)--(4.8000,0.0788629)--(4.7750,0.0807191)--(4.7500,0.0826194)--(4.7250,0.0845648)--
			(4.7000,0.0865565)--(4.6750,0.0885954)--(4.6500,0.0906828)--(4.6250,0.0928199)--(4.6000,0.0950078)--
			(4.5750,0.0972477)--(4.5500,0.0995410)--(4.5250,0.1018889)--(4.5000,0.1042926)--(4.4750,0.1067537)--
			(4.4500,0.1092733)--(4.4250,0.1118531)--(4.4000,0.1144943)--(4.3750,0.1171986)--(4.3500,0.1199673)--
			(4.3250,0.1228022)--(4.3000,0.1257047)--(4.2750,0.1286765)--(4.2500,0.1317193)--(4.2250,0.1348348)--
			(4.2000,0.1380247)--(4.1750,0.1412909)--(4.1500,0.1446351)--(4.1250,0.1480594)--(4.1000,0.1515655)--
			(4.0750,0.1551555)--(4.0500,0.1588314)--(4.0250,0.1625952)--(4.0000,0.1664492)--(3.9750,0.1703953)--
			(3.9500,0.1744360)--(3.9250,0.1785734)--(3.9000,0.1828098)--(3.8750,0.1871477)--(3.8500,0.1915894)--
			(3.8250,0.1961375)--(3.8000,0.2007945)--(3.7750,0.2055630)--(3.7500,0.2104456)--(3.7250,0.2154451)--
			(3.7000,0.2205642)--(3.6750,0.2258057)--(3.6500,0.2311727)--(3.6250,0.2366679)--(3.6000,0.2422945)--
			(3.5750,0.2480556)--(3.5500,0.2539542)--(3.5250,0.2599936)--(3.5000,0.2661770)--(3.4750,0.2725079)--
			(3.4500,0.2789896)--(3.4250,0.2856255)--(3.4000,0.2924193)--(3.3750,0.2993744)--(3.3500,0.3064947)--
			(3.3250,0.3137837)--(3.3000,0.3212453)--(3.2750,0.3288834)--(3.2500,0.3367018)--(3.2250,0.3447046)--
			(3.2000,0.3528958)--(3.1750,0.3612794)--(3.1500,0.3698596)--(3.1250,0.3786406)--(3.1000,0.3876267)--
			(3.0750,0.3968222)--(3.0500,0.4062313)--(3.0250,0.4158585)--(3.0000,0.4257083)--(2.9750,0.4357849)--
			(2.9500,0.4460930)--(2.9250,0.4566371)--(2.9000,0.4674215)--(2.8750,0.4784510)--(2.8500,0.4897300)--
			(2.8250,0.5012630)--(2.8000,0.5130545)--(2.7750,0.5251091)--(2.7500,0.5374312)--(2.7250,0.5500253)--
			(2.7000,0.5628957)--(2.6750,0.5760467)--(2.6500,0.5894825)--(2.6250,0.6032074)--(2.6000,0.6172253)--
			(2.5750,0.6315402)--(2.5500,0.6461559)--(2.5250,0.6610761)--(2.5000,0.6763042)--(2.4750,0.6918436)--
			(2.4500,0.7076974)--(2.4250,0.7238684)--(2.4000,0.7403593)--(2.3750,0.7571725)--(2.3500,0.7743102)--
			(2.3250,0.7917741)--(2.3000,0.8095656)--(2.2750,0.8276860)--(2.2500,0.8461360)--(2.2250,0.8649160)--
			(2.2000,0.8840257)--(2.1750,0.9034649)--(2.1500,0.9232324)--(2.1250,0.9433269)--(2.1000,0.9637462)--
			(2.0750,0.9844879)--(2.0500,1.0055488)--(2.0250,1.0269254)--(2.0000,1.0486131)--(1.9750,1.0706072)--
			(1.9500,1.0929021)--(1.9250,1.1154914)--(1.9000,1.1383684)--(1.8750,1.1615254)--(1.8500,1.1849541)--
			(1.8250,1.2086455)--(1.8000,1.2325899)--(1.7750,1.2567769)--(1.7500,1.2811953)--(1.7250,1.3058333)--
			(1.7000,1.3306784)--(1.6750,1.3557175)--(1.6500,1.3809366)--(1.6250,1.4063212)--(1.6000,1.4318562)--
			(1.5750,1.4575260)--(1.5500,1.4833143)--(1.5250,1.5092044)--(1.5000,1.5351789)--(1.4750,1.5612203)--
			(1.4500,1.5873106)--(1.4250,1.6134314)--(1.4000,1.6395640)--(1.3750,1.6656896)--(1.3500,1.6917894)--
			(1.3250,1.7178441)--(1.3000,1.7438347)--(1.2750,1.7697421)--(1.2500,1.7955474)--(1.2250,1.8212318)--
			(1.2000,1.8467767)--(1.1750,1.8721639)--(1.1500,1.8973754)--(1.1250,1.9223937)--(1.1000,1.9472020)--
			(1.0750,1.9717835)--(1.0500,1.9961225)--(1.0250,2.0202036)--(1.0000,2.0440122)--(0.9750,2.0675344)--
			(0.9500,2.0907571)--(0.9250,2.1136677)--(0.9000,2.1362548)--(0.8750,2.1585075)--(0.8500,2.1804158)--
			(0.8250,2.2019708)--(0.8000,2.2231640)--(0.7750,2.2439882)--(0.7500,2.2644368)--(0.7250,2.2845040)--
			(0.7000,2.3041850)--(0.6750,2.3234758)--(0.6500,2.3423731)--(0.6250,2.3608745)--(0.6000,2.3789782)--
			(0.5750,2.3966833)--(0.5500,2.4139895)--(0.5250,2.4308972)--(0.5000,2.4474075)--(0.4750,2.4635219)--
			(0.4500,2.4792429)--(0.4250,2.4945730)--(0.4000,2.5095156)--(0.3750,2.5240745)--(0.3500,2.5382538)--
			(0.3250,2.5520580)--(0.3000,2.5654923)--(0.2750,2.5785617)--(0.2500,2.5912720)--(0.2250,2.6036288)--
			(0.2000,2.6156385)--(0.1750,2.6273071)--(0.1500,2.6386413)--(0.1250,2.6496476)--(0.1000,2.6603327)--
			(0.0750,2.6707036)--(0.0500,2.6807671)--(0.0250,2.6905302)--(0.0000,2.7000000)--cycle;
			
			\draw[->,black!80] (0,0)--(6.45,0) node[right] {$r$};
			\draw[->,black!80] (0,0)--(0,3.25) node[above] {$P$};
			\draw[blue!75!black,line width=1.1pt]
			plot coordinates {
				(0.0000,2.7000000) (0.0250,2.6905302) (0.0500,2.6807671) (0.0750,2.6707036) (0.1000,2.6603327)
				(0.1250,2.6496476) (0.1500,2.6386413) (0.1750,2.6273071) (0.2000,2.6156385) (0.2250,2.6036288)
				(0.2500,2.5912720) (0.2750,2.5785617) (0.3000,2.5654923) (0.3250,2.5520580) (0.3500,2.5382538)
				(0.3750,2.5240745) (0.4000,2.5095156) (0.4250,2.4945730) (0.4500,2.4792429) (0.4750,2.4635219)
				(0.5000,2.4474075) (0.5250,2.4308972) (0.5500,2.4139895) (0.5750,2.3966833) (0.6000,2.3789782)
				(0.6250,2.3608745) (0.6500,2.3423731) (0.6750,2.3234758) (0.7000,2.3041850) (0.7250,2.2845040)
				(0.7500,2.2644368) (0.7750,2.2439882) (0.8000,2.2231640) (0.8250,2.2019708) (0.8500,2.1804158)
				(0.8750,2.1585075) (0.9000,2.1362548) (0.9250,2.1136677) (0.9500,2.0907571) (0.9750,2.0675344)
				(1.0000,2.0440122) (1.0250,2.0202036) (1.0500,1.9961225) (1.0750,1.9717835) (1.1000,1.9472020)
				(1.1250,1.9223937) (1.1500,1.8973754) (1.1750,1.8721639) (1.2000,1.8467767) (1.2250,1.8212318)
				(1.2500,1.7955474) (1.2750,1.7697421) (1.3000,1.7438347) (1.3250,1.7178441) (1.3500,1.6917894)
				(1.3750,1.6656896) (1.4000,1.6395640) (1.4250,1.6134314) (1.4500,1.5873106) (1.4750,1.5612203)
				(1.5000,1.5351789) (1.5250,1.5092044) (1.5500,1.4833143) (1.5750,1.4575260) (1.6000,1.4318562)
				(1.6250,1.4063212) (1.6500,1.3809366) (1.6750,1.3557175) (1.7000,1.3306784) (1.7250,1.3058333)
				(1.7500,1.2811953) (1.7750,1.2567769) (1.8000,1.2325899) (1.8250,1.2086455) (1.8500,1.1849541)
				(1.8750,1.1615254) (1.9000,1.1383684) (1.9250,1.1154914) (1.9500,1.0929021) (1.9750,1.0706072)
				(2.0000,1.0486131) (2.0250,1.0269254) (2.0500,1.0055488) (2.0750,0.9844879) (2.1000,0.9637462)
				(2.1250,0.9433269) (2.1500,0.9232324) (2.1750,0.9034649) (2.2000,0.8840257) (2.2250,0.8649160)
				(2.2500,0.8461360) (2.2750,0.8276860) (2.3000,0.8095656) (2.3250,0.7917741) (2.3500,0.7743102)
				(2.3750,0.7571725) (2.4000,0.7403593) (2.4250,0.7238684) (2.4500,0.7076974) (2.4750,0.6918436)
				(2.5000,0.6763042) (2.5250,0.6610761) (2.5500,0.6461559) (2.5750,0.6315402) (2.6000,0.6172253)
				(2.6250,0.6032074) (2.6500,0.5894825) (2.6750,0.5760467) (2.7000,0.5628957) (2.7250,0.5500253)
				(2.7500,0.5374312) (2.7750,0.5251091) (2.8000,0.5130545) (2.8250,0.5012630) (2.8500,0.4897300)
				(2.8750,0.4784510) (2.9000,0.4674215) (2.9250,0.4566371) (2.9500,0.4460930) (2.9750,0.4357849)
				(3.0000,0.4257083) (3.0250,0.4158585) (3.0500,0.4062313) (3.0750,0.3968222) (3.1000,0.3876267)
				(3.1250,0.3786406) (3.1500,0.3698596) (3.1750,0.3612794) (3.2000,0.3528958) (3.2250,0.3447046)
				(3.2500,0.3367018) (3.2750,0.3288834) (3.3000,0.3212453) (3.3250,0.3137837) (3.3500,0.3064947)
				(3.3750,0.2993744) (3.4000,0.2924193) (3.4250,0.2856255) (3.4500,0.2789896) (3.4750,0.2725079)
				(3.5000,0.2661770) (3.5250,0.2599936) (3.5500,0.2539542) (3.5750,0.2480556) (3.6000,0.2422945)
				(3.6250,0.2366679) (3.6500,0.2311727) (3.6750,0.2258057) (3.7000,0.2205642) (3.7250,0.2154451)
				(3.7500,0.2104456) (3.7750,0.2055630) (3.8000,0.2007945) (3.8250,0.1961375) (3.8500,0.1915894)
				(3.8750,0.1871477) (3.9000,0.1828098) (3.9250,0.1785734) (3.9500,0.1744360) (3.9750,0.1703953)
				(4.0000,0.1664492) (4.0250,0.1625952) (4.0500,0.1588314) (4.0750,0.1551555) (4.1000,0.1515655)
				(4.1250,0.1480594) (4.1500,0.1446351) (4.1750,0.1412909) (4.2000,0.1380247) (4.2250,0.1348348)
				(4.2500,0.1317193) (4.2750,0.1286765) (4.3000,0.1257047) (4.3250,0.1228022) (4.3500,0.1199673)
				(4.3750,0.1171986) (4.4000,0.1144943) (4.4250,0.1118531) (4.4500,0.1092733) (4.4750,0.1067537)
				(4.5000,0.1042926) (4.5250,0.1018889) (4.5500,0.0995410) (4.5750,0.0972477) (4.6000,0.0950078)
				(4.6250,0.0928199) (4.6500,0.0906828) (4.6750,0.0885954) (4.7000,0.0865565) (4.7250,0.0845648)
				(4.7500,0.0826194) (4.7750,0.0807191) (4.8000,0.0788629) (4.8250,0.0770497) (4.8500,0.0752786)
				(4.8750,0.0735484) (4.9000,0.0718584) (4.9250,0.0702075) (4.9500,0.0685948) (4.9750,0.0670194)
				(5.0000,0.0654805)
			};
			\draw[blue!75!black,densely dashed,line width=.9pt]
			plot coordinates {
				(5.0000,0.0654805) (5.0250,0.0639771) (5.0500,0.0625086) (5.0750,0.0610739) (5.1000,0.0596725)
				(5.1250,0.0583034) (5.1500,0.0569659) (5.1750,0.0556593) (5.2000,0.0543829) (5.2250,0.0531359)
				(5.2500,0.0519178) (5.2750,0.0507277) (5.3000,0.0495650) (5.3250,0.0484292) (5.3500,0.0473195)
				(5.3750,0.0462354) (5.4000,0.0451763) (5.4250,0.0441416) (5.4500,0.0431307) (5.4750,0.0421431)
				(5.5000,0.0411782) (5.5250,0.0402356) (5.5500,0.0393146) (5.5750,0.0384148) (5.6000,0.0375357)
				(5.6250,0.0366768) (5.6500,0.0358377) (5.6750,0.0350178) (5.7000,0.0342168) (5.7250,0.0334342)
				(5.7500,0.0326695) (5.7750,0.0319225) (5.8000,0.0311926) (5.8250,0.0304794) (5.8500,0.0297826)
				(5.8750,0.0291018) (5.9000,0.0284366) (5.9250,0.0277867) (5.9500,0.0271517) (5.9750,0.0265313)
				(6.0000,0.0259251) (6.0250,0.0253328) (6.0500,0.0247540) (6.0750,0.0241885) (6.1000,0.0236360)
				(6.1250,0.0230962) (6.1500,0.0225687) (6.1750,0.0220533) (6.2000,0.0215497)
			};
			
			\foreach \x/\lab in {0/1,2.5/10,5/100}{
				\draw (\x,0)--(\x,-.065) node[below=2pt] {$\lab$};
			}
			\draw (0,3)--(-.065,3) node[left=2pt] {$1$};
			\draw (0,2.7)--(-.065,2.7) node[left=2pt] {$1-\varepsilon$};
			\draw[black!45,densely dashed,line width=.65pt]
			(0,0.0654805)--(5,0.0654805);
			\draw[blue!75!black,fill=blue!75!black]
			(5,0.0654805) circle (1pt);
			\draw (-.055,0.0654805)--(0,0.0654805);
			\node[font=\scriptsize,anchor=south east,inner sep=1pt]
			at (-.09,0.0654805)
			{$P_\varepsilon(100)\approx 0.0218$};
			\node[font=\small] at (3.15,2.28)
			{$G(P)\geq 1-\varepsilon$};
			\node[font=\small] at (1.12,.57)
			{$G(P)<1-\varepsilon$};
			\node[anchor=west,blue!75!black] at (2.7,1.10)
			{$P=P_\varepsilon(r)$};
			\node[font=\scriptsize,anchor=west] at (.15,3.16)
			{$\alpha=2,\ \varepsilon=0.1$};
		\end{tikzpicture}
	}
	\caption{Minimum trainable proportion for $90\%$ normalized gain, with $\alpha=2$ and $\varepsilon=0.1$. Shaded region above $P_{\varepsilon}(r)$ satisfies $G(P)\ge 1-\varepsilon$. The endpoint at $r=1$ represents the limit $1-\varepsilon$ as $r\to1_{+}$.}
	\label{fig:tau-phase-transition-pattern}
\end{figure}
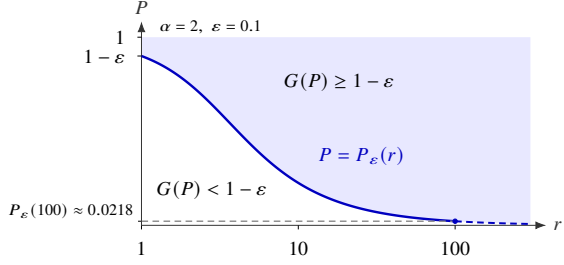
Figure \ref{fig:tau-phase-transition-pattern} illustrates this threshold for $\alpha=2$ and $\varepsilon=0.1$. As $r$ increases, a smaller trainable proportion suffices to attain $90\%$ of the normalized gain. At $r=100$, $P_{\varepsilon}\approx0.0218$, giving a proportion threshold of approximately $2.2\%$.
\begin{remark}[Phase Transition in the Upper Bound]
	The normalized gain describes the reduction in the theoretical upper bound. Observed errors also depend on the target function, input sampling, and optimization, and may therefore deviate from this curve. Nevertheless, the subsequent numerical experiments exhibit a similar qualitative phase transition.
\end{remark}
	\section{Analysis of Spherical Biased-ReLU Features}
\label{sec:specific-activation-functions}
Since the preceding results in Section \ref{sec:main-results} apply to general feature maps under assumptions on sampling compatibility and spectral decay, this section verifies these assumptions for spherical biased-ReLU features.

\subsection{The Target RKHS and Kernel Spectrum}
\mypara{Target Function Class.} Let $d\ge 2$ and consider the spherical input and parameter spaces
\[
\begin{aligned}
	\mathcal{X}&=\mathbb{S}^{d-1}=\{x\in\RR^{d},\norm{x}_{2}=1\}\subset\RR^{d},&&\mu=\operatorname{Unif}(\mathbb S^{d-1}),\\
	\Theta&=\mathbb S^d\subset\RR^{d+1},&&\rho=\operatorname{Unif}(\mathbb S^d), 
\end{aligned} \]
where $\operatorname{Unif}(\cdot)$ denotes the normalized rotation-invariant surface measure. Throughout this section, $\rho$ serves as the reference measure defining the kernel $k_\rho$, the integral operator $T_\rho$, and the associated target RKHS $\mathcal H_\rho$.

For $\theta=(\omega,b)\in\Theta$, define the  biased ReLU feature by
\[ \phi(x;\theta)=\sigma_{R}(\omega^{\top}x+b)=(\theta^{\top}\tilde{x})_{+},~\text{where}~\sigma_{R}(t)=\max\{t,0\},~\text{and}~\tilde{x}=(x,1). \]
Since $\norm{\theta}_{2}^{2}=1$ and $\norm{\tilde{x}}_{2}=\sqrt{2}$, we have $0\le\phi(x;\theta)\le\sqrt{2}$. These bounded features satisfy the integrability Assumption \ref{assumption:operator}. Thus, the associated operator $T_\rho$ taking the form of \eqref{eq:kernel-integral-operator} is self-adjoint, positive semi-definite, trace-class, and compact on $L_2(\mathbb S^{d-1},d\mu)$.

\mypara{Scaled Spherical Initialization.} To examine the initialization mismatch while keeping the target RKHS fixed, we consider the following family of sampling distributions for hidden neurons. Let $\omega_{0}\sim\mathcal{N}(0,I_d)$ and
$b_{0}\sim\mathcal{N}(0,1)$ be independent. For $\tau \geq 1$, define
\begin{equation}
	\nonumber
	\theta=\frac{(\omega_{0},\tau b_{0})}
	{\sqrt{\norm{\omega_{0}}_2^2+\tau^2b_{0}^2}}=(\omega,b)\in \Theta=\mathbb S^d,
\end{equation}
and denote its distribution by $\rho_0^{(\tau)}$. Its density  with respect to $\rho$ is
\begin{equation}
	\nonumber
	\gamma_0^{(\tau)}(\theta)=\frac{d\rho_{0}^{(\tau)}}{d\rho}(\theta)=\frac{1}{\tau} \Big(1-b^2+\frac{b^2}{\tau^2} \Big)^{-(d+1)/2}.
\end{equation}
In particular, $\rho_0^{(1)}=\rho$. The density satisfies $\gamma_0^{(\tau)}(\theta)\ge {1}/{\tau}$ for all $\theta\in\Theta$, since $\tau\geq1$ and $\abs{b}\le 1$. The randomly fixed neurons are sampled from $\rho_0^{(\tau)}=\gamma_0^{(\tau)}\rho$. For each regularization scale $\lambda>0$, the trainable parameters are constructed using the optimal ridge leverage distribution introduced in \eqref{eq:leverage-score-distribution}, that is, $\rho_{1}=\gamma_1\rho$ with $ \gamma_1=\gamma_{\lambda}^{*}$ and $\kappa_{{1}}^{*}=1$. The corresponding mixture density is therefore
\[
\gamma_P(\theta)=P\gamma_1(\theta)+(1-P)\gamma_0^{(\tau)}(\theta).
\]

With the reweighted features defined in Subsection
\ref{subsec:approxi-bounds-general-feature-maps}, Lemma \ref{lemma:kernel-operator-invariance} ensures that this sampling construction preserves $k_\rho$ and $T_\rho$. Varying $\tau$ only changes the initialization distribution while leaving the target space $\mathcal H_\rho$ unchanged.

\begin{figure}[H]
	\centering
	\resizebox{\linewidth}{!}{%
		\begin{tikzpicture}[>=Latex,font=\small,text=black]
			\def\R{1.42}
			\def\tilt{12}
			\def\N{90}
			\def\poleY{\R}
			
			\draw[black!65,densely dashed,line width=.45pt]
			(-3.8,\poleY)--(2.55,\poleY);
			\draw[black!65,densely dashed,line width=.45pt]
			(-3.8,-\poleY)--(2.55,-\poleY);
			\node[anchor=east] at (-3.8,\poleY) {$b=1$};
			\node[anchor=east] at (-3.8,-\poleY) {$b=-1$};
			
			\foreach \cx/\scalefactor in {-2.55/1,2.55/3.2}{
				\begin{scope}[shift={(\cx,0)}]
					\shade[inner color=white,outer color=blue!13] (0,0) circle (\R);
					\draw[blue!42!black,densely dashed,line width=.35pt]
					(-\R,0) arc[start angle=180,end angle=0,
					x radius=\R,y radius={\R*sin(\tilt)}];
					\draw[blue!42!black,line width=.45pt]
					(-\R,0) arc[start angle=180,end angle=360,
					x radius=\R,y radius={\R*sin(\tilt)}];
					\draw[blue!46!black,line width=.55pt] (0,0) circle (\R);
					
					\foreach \i in {0,...,89}{
						\pgfmathsetmacro{\z}{1-2*(\i+.5)/\N}
						\pgfmathsetmacro{\rxy}{sqrt(1-\z*\z)}
						\pgfmathsetmacro{\phi}{\i*137.507764}
						\pgfmathsetmacro{\xx}{\rxy*cos(\phi)}
						\pgfmathsetmacro{\yy}{\rxy*sin(\phi)}
						\pgfmathsetmacro{\den}{sqrt(1-\z*\z+\scalefactor*\scalefactor*\z*\z)}
						\pgfmathsetmacro{\xp}{\xx/\den}
						\pgfmathsetmacro{\yp}{\yy/\den}
						\pgfmathsetmacro{\zp}{\scalefactor*\z/\den}
						\pgfmathsetmacro{\depth}{\yp*cos(\tilt)+\zp*sin(\tilt)}
						\ifdim\depth pt>0pt
						\fill[color={rgb,255:red,60;green,84;blue,136}]
						({\R*\xp},{\R*(\zp*cos(\tilt)-\yp*sin(\tilt))})
						circle (.95pt);
						\fi
					}
					\node[inner sep=0pt] at (0,\poleY) {$\star$};
					\node[inner sep=0pt] at (0,-\poleY) {$\star$};
					\ifdim\scalefactor pt<2pt
					\node[anchor=base] at (0,-1.91) {$\rho=\rho_0^{(1)},~\tau=1$};
					\else
					\node[anchor=base] at (0,-1.91) {$\rho_0^{(\tau)},~\tau>1$};
					\fi
				\end{scope}
			}
			\draw[->,thick] (-.72,0)--(.72,0)
			node[midway,above=2pt] {$\tau\uparrow$};
			\draw[->,thick] (4.38,0)--(5.82,0)
			node[midway,above=2pt] {Mismatch $\uparrow$};
			
			\begin{scope}[shift={(6.23,-1.42)}]
				\draw[->,black!80] (0,0)--(2.84,0) node[right,text=black] {$P$};
				\draw[->,black!80] (0,0)--(0,2.84)
				node[above,text=black] {$G(P)$};
				\draw[blue!45!black,densely dashed,line width=.9pt,
				domain=0:1,samples=60,smooth,variable=\p]
				plot ({2.65*\p},{2.65*2*\p/(1+\p)});
				\draw[blue!75!black,line width=1.1pt,
				domain=0:1,samples=60,smooth,variable=\p]
				plot ({2.65*\p},{2.65*12*\p/(1+11*\p)});
				\draw[->,black,line width=.55pt] (1.2,1.83)--(.8,2.2);
				\node[anchor=west,font=\scriptsize] at (.12,2.65)
				{Larger mismatch};
				\node[anchor=west,font=\scriptsize] at (.7,1)
				{Smaller mismatch};
				\node[anchor=base,font=\small] at (1.42,-.49)
				{Sharper phase transition};
			\end{scope}
		\end{tikzpicture}
	}
	\caption{Scaled spherical initialization. Increasing $\tau$ shifts fixed-neuron parameters toward the poles $b=\pm1$. The right panel schematically shows how greater mismatch can sharpen the gain at small $P$.}
	\label{fig:tau-illustration-drawing}
\end{figure}
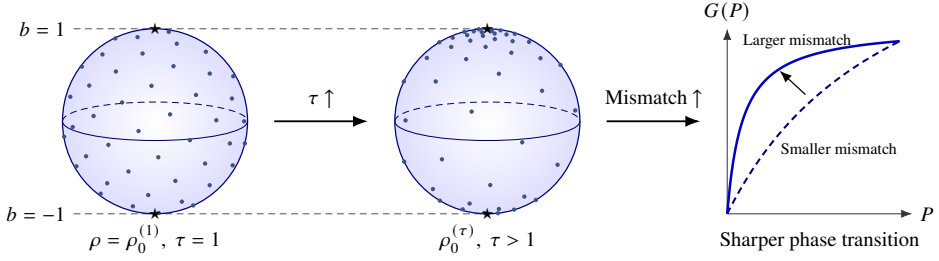
Figure \ref{fig:tau-illustration-drawing} illustrates the role of $\tau$. At $\tau=1$, fixed neurons follow the uniform reference distribution $\rho$. Increasing $\tau$ concentrates their initialization near $b=\pm1$. The right panel schematically illustrates how greater initialization mismatch with the target RKHS can sharpen the normalized gain as $P$ increases. Thus the scaled spherical family provides a concrete setting for analyzing the effect of initialization mismatch, and the framework also accommodates other admissible initialization distributions.

\mypara{Eigen-decomposition of $T_{\rho}$.} The spectrum of $T_\rho$ is characterized using spherical harmonics. Following \cite{atkinson2012spherical,bach2017breaking}, let $\YY_k$ be the space of spherical harmonics of degree $k$ on $\mathbb S^{d-1}$, and let $\{Y_{k,j}\}_{j=1}^{N(d,k)}$ be a real orthonormal basis of $\YY_{k}$ in $L_2(\mathbb S^{d-1}, d\mu)$, where
\begin{equation}
	\label{eq:degree-N-Formula}
	N(d,0)=1,~N(d,k)=\frac{2k+d-2}{k}\binom{k+d-3}{d-2},~d\ge2,k\ge 1.
\end{equation}
Every $f\in L_2(\mathbb S^{d-1},d\mu) $ admits the following expansion, with convergence in $L_2(d\mu)$,
\begin{equation}
	\nonumber
    f(x)=\sum_{k=0}^{+\infty}\sum_{j=1}^{N(d,k)}\langle f, Y_{k,j} \rangle_{L_{2}(d\mu)}Y_{k,j}(x)=\sum_{k=0}^{+\infty} N(d,k) \int_{\mathcal{X}}f(y)P_{k}(x^{\top}y)d\mu(y).
\end{equation}
Here $P_{k}$ is the normalized Legendre polynomial of degree $k$ in dimension  $d$, with $P_{k}(1)=1$.

By the rotation invariance of $\rho$ and results in \cite[Section 3.1]{bach2017breaking}, the induced kernel $k_{\rho}(x,x^{\prime})$ now has the explicit form
\[ k_{\rho}(x,x^{\prime})=\EE_{\theta\sim\rho}[\phi(x;\theta)\phi(x^{\prime};\theta)]=h \Big(\frac{1+x^{\top}x^{\prime}}{2} \Big),\]
where $h(t)= \frac{1}{\pi(d+1)}(\sqrt{1-t^{2}}+(\pi-\arccos t )t)$. Thus, $k_\rho$ is a continuous rotation-invariant kernel on $\mathbb S^{d-1}$, allowing its spectrum to be analyzed through the spherical harmonic decomposition.

\begin{proposition}[Spectrum and Eigenvalue Decay of ReLU Kernel]
	\label{prop:spectrum-biased-relu-kernel}
	Under the spherical setting described above, the following properties hold:
	\begin{enumerate}[labelindent=0pt, leftmargin=*,topsep=0pt,partopsep=0pt]
		\item Every spherical harmonic $Y_{k,j}\in\YY_{k}$ is an eigenfunction of $T_{\rho}$, with $T_{\rho}Y_{k,j}=\xi_{k}Y_{k,j}, j=1,\ldots,N(d,k)$, where the eigenvalue associated with degree $k$ satisfies
		\begin{equation}
			\label{eq:relu-degree-eigenvalue}
			\xi_{k}=\frac{\Gamma(d/2)}{\sqrt{\pi}\Gamma((d-1)/2)}\int_{-1}^{1}h \Big(\frac{1+t}{2} \Big)P_{k}(t)(1-t^{2})^{(d-3)/2}dt>0.
		\end{equation}
		Here $\Gamma$ is the Gamma function and $h(t)= \frac{1}{\pi(d+1)}(\sqrt{1-t^{2}}+(\pi-\arccos t )t)$.
		\item Consequently, the operator and kernel admit the decompositions:
		\[ T_{\rho}=\sum_{k=0}^{+\infty}\xi_{k}\sum_{j=1}^{N(d,k)}Y_{k,j}\otimes Y_{k,j},~\text{and}~k_{\rho}(x,x^{\prime})=\sum_{k=0}^{+\infty}\xi_{k}\sum_{j=1}^{N(d,k)} Y_{k,j}(x)Y_{k,j}(x^{\prime}). \]
		\item The degree-wise eigenvalues satisfy $\xi_{k}\asymp k^{-(d+2)}$ as $k\to\infty$. Moreover, the eigenvalues $\{\eta_k\}_{k\ge1}$ of $T_{\rho}$ arranged in non-increasing order and counted with multiplicity satisfy $\eta_k\asymp k^{-\frac{d+2}{d-1}}$.
	\end{enumerate}
\end{proposition}
Proposition \ref{prop:spectrum-biased-relu-kernel} verifies the polynomial spectral decay Assumption \ref{assumption:eigendecay-used} with $\alpha=\frac{d+2}{d-1}$. It also characterizes the considered target RKHS as
\begin{equation}
	\label{eq:space-defition-of-RKHS-spericial}
	\mathcal{H}_{\rho}=\Bigg\{ f=\sum_{k=0}^{\infty}\sum_{j=1}^{N(d,k)}c_{k,j}Y_{k,j}: \norm{f}_{\mathcal{H}_{\rho}}^{2}=\sum_{k=0}^{\infty}\sum_{j=1}^{N(d,k)}\frac{c_{k,j}^{2}}{\xi_{k}}<\infty\Bigg\}.
\end{equation}

\subsection{Initialization Mismatch and Approximation Rates}
With the spectral decay established, the remaining step is to verify Assumption \ref{assumption:mismatch-used} by bounding the initialization mismatch $\kappa_{\gamma_{0}^{(\tau)}}(\lambda)$. This provides the compatibility constant needed to specialize the approximation rates.

\mypara{Quantification of the Initialization Mismatch.} For the scaled spherical initialization $\rho_0^{(\tau)}=\gamma_0^{(\tau)}\rho$, the mismatch factor defined in \eqref{eq:definition-of-kappa-first} is
\begin{equation}
    \nonumber
	\kappa_{\gamma_0^{(\tau)}}(\lambda)= \frac{d_{\max}(\gamma_{0}^{(\tau)},\lambda)}{d(\lambda)}=\frac{1}{d(\lambda)}\operatorname*{ess\,sup}_{\theta\sim\rho_{0}^{(\tau)}}\frac{  \langle
		\phi(\cdot;\theta),(T_\rho+\lambda I)^{-1}\phi(\cdot;\theta)
		 \rangle_{L_{2}(d\mu)}}{\gamma_{0}^{(\tau)}(\theta)}.
\end{equation}
The following proposition bounds this factor explicitly in terms of $d$ and $\tau$.
\begin{proposition}[Mismatch of Scaled Spherical Initialization]
	\label{prop:anisotropic-spherical-mismatch}
	Under the spherical biased-ReLU setting, fix $\tau\ge 1$ and let the randomly fixed neurons follow $\rho_0^{(\tau)}=\gamma_0^{(\tau)}\rho$. Then, for every $\varepsilon>0$, there exists
	$\lambda_{\varepsilon,\tau}>0$ such that
	\begin{equation}
		\label{eq:anisotropic-spherical-mismatch-upper-bound}
		\nonumber
		\frac{\sqrt{2}\tau(d+1)}{d}-\varepsilon\le \kappa_{\gamma_0^{(\tau)}}(\lambda)\le\frac{2\sqrt{2}\tau(d+1)}{d}+\varepsilon,~\text{for}~
		\lambda\in(0,\lambda_{\varepsilon,\tau}].
	\end{equation}
\end{proposition}
To absorb the $\varepsilon$ term in the upper bound, we choose $\kappa_0^*={4\sqrt{2}\tau(d+1)}/{d}$ as an admissible compatibility constant for subsequent analysis, valid for all sufficiently small $\lambda>0$. These estimates show that the initialization mismatch is of linear order in $\tau$. Choosing the optimal ridge leverage distribution for the trainable neurons gives $\kappa_1^*=1$. Now the leverage compatibility Assumption \ref{assumption:mismatch-used} is verified.

Substituting these constants and the spectral decay exponent $\alpha=(d+2)/(d-1)$ into the general polynomial rate yields the following result.
\begin{theorem}[Rates of Spherical Biased-ReLU Features]
	\label{thm:sphericial-polynomial-rates} Under the spherical biased-ReLU setting, fix $\tau\ge 1$. Let $m_{\mathrm{free}}, m_{\mathrm{fix}}\in\NN$ with $m=m_{\mathrm{free}}+m_{\mathrm{fix}}\ge 1$, and set $P=m_{\mathrm{free}}/m$ and $W=\max\{d,m,1\}$. For any $\delta\in(0,1)$, there exists $\lambda>0$ such that, when the two groups are sampled according to Definition \ref{def:sampling-scheme}, with fixed-group distribution $\rho_0^{(\tau)}$ and trainable-group distribution $\rho_1=\gamma_\lambda^*\rho$, the following bound holds with probability at least $1-\delta$:
	\begin{equation}
		\nonumber
		\sup_{\norm{f}_{\mathcal H_\rho}\le 1}\inf_{f_{2}^{(P)}\in\mathcal{F}_P(W,2)}\norm{f-f_{2}^{(P)}}_{L_{2}(d\mu)}^{2}\lesssim \Big( \Big({P}+\frac{(1-P)d}{4\sqrt{2}\tau(d+1)} \Big)\frac{9m}{40\log(63m/(20\delta))} \Big)^{-\frac{d+2}{d-1}}.
	\end{equation}
\end{theorem}
The corresponding admissible mismatch ratio in Corollary \ref{cor:phase-transition} is currently $r=\kappa_{0}^{*}/\kappa_{{1}}^{*}=4\sqrt2\tau(d+1)/d$. Increasing $\tau$ raises this ratio and produces a sharper predicted transition at small $P$. The rate thus quantifies how the fixed initialization and trainable proportion $P$ affect the approximation accuracy.

Figure \ref{fig:spherical-tau-gain} illustrates the predicted normalized gain in a representative high-dimensional setting with $d=20$, for which $\alpha=22/19$. In particular, the proportion required to attain $90\%$ of the gain decreases from approximately $0.563$ at $\tau=1$ to $0.100$ at $\tau=10$.
\begin{figure}[H]
	\centering
	\resizebox{\linewidth}{!}{%
		\begin{tikzpicture}[font=\small,text=black]
			\def\inputDim{20}
			\pgfmathsetmacro{\spectralExponent}{(\inputDim+2)/(\inputDim-1)}
			\begin{groupplot}[
				group style={group size=4 by 1,horizontal sep=0.7cm,ylabels at=edge left},
				width=3.5cm,height=3.5cm,scale only axis,
				xmin=0,xmax=1.04,ymin=0,ymax=1.10,
				axis lines=left,axis line style={black,line width=.5pt},
				xtick={0,0.5,1},ytick={0,0.5,0.9,1},
				tick align=outside,tick style={black,line width=.4pt},
				tick label style={font=\scriptsize,text=black},
				scaled ticks=false,
				xlabel={$P$},xlabel style={font=\small,text=black},
				ylabel={$G(P)$},ylabel style={font=\small,text=black},
				title style={font=\small,text=black,yshift=-6pt},
				clip=true,clip mode=individual,
				]
				\pgfplotsinvokeforeach{1,2,5,10}{
					\nextgroupplot[title={$\tau=#1$}]
					\pgfkeys{/pgf/fpu=true,/pgf/fpu/output format=fixed}
					\pgfmathsetmacro{\contrast}{4*sqrt(2)*#1*(\inputDim+1)/\inputDim}
					\pgfmathsetmacro{\threshold}{
						((0.1+0.9*\contrast^(-\spectralExponent))^(-1/\spectralExponent)-1)
						/(\contrast-1)}
					\pgfkeys{/pgf/fpu=false}
					\addplot[black!40,densely dashed,line width=.5pt]
					coordinates {(0,1) (1,1) (1,0)};
					\addplot[black!40,densely dashed,line width=.5pt]
					coordinates {(0,0.9) (\threshold,0.9) (\threshold,0)};
					\addplot[blue!75!black,line width=1pt,domain=0:1,
					samples=351,variable=\t]
					({\t^2},
					{(1-(1+(\contrast-1)*\t^2)^(-\spectralExponent))
						/(1-\contrast^(-\spectralExponent))});
					\addplot[only marks,mark=*,mark size=1.3pt,color={rgb,255:red,214;green,39;blue,40}]
					coordinates {(\threshold,0.9) (1,1)};
					\node[anchor=center,font=\small,text=black]
					at (rel axis cs:0.5,0.45)
					{$P_{0.1}\approx\pgfmathprintnumber[fixed,fixed zerofill,precision=3]{\threshold}$};
				}
			\end{groupplot}
		\end{tikzpicture}
	}
	\caption{Normalized gain predicted by the spherical biased-ReLU bound in dimension $d=20$ ($\alpha=22/19$) for increasing $\tau$. Gray dashed guides mark $90\%$ gain and the endpoint $G(1)=1$.}
	\label{fig:spherical-tau-gain}
\end{figure}
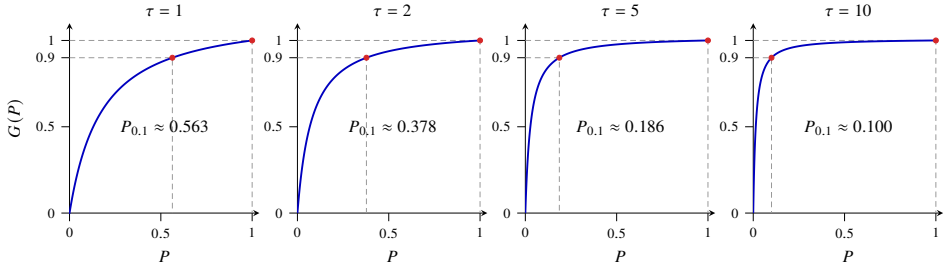

Besides, Figure \ref{fig:intro-spherical-phase-transition} in the introduction provides a numerical illustration for a fixed target in the spherical RKHS unit ball with $d=5$, using a two-layer ReLU network of width $64$. The observed $90\%$ gain threshold decreases overall as $\tau$ increases, and the gain curve for $\tau=10$ rises more rapidly at small $P$ than that for $\tau=1$. These observations are qualitatively consistent with the dependence on initialization predicted by the bound. The target construction and training settings for this $d=5$ experiment are provided in Appendix \ref{subsec:tau-experiments}.

\mypara{Mechanism and Scope.} Varying the fixed-neuron initialization while keeping the target RKHS unchanged isolates the effect of sampling mismatch. The spherical biased-ReLU setting precisely demonstrates how a small trainable group can offset the fixed group's mismatch and substantially reduce the approximation bound, consistent with the general theory in Section \ref{sec:main-results}. The framework can also be applied to other input domains, feature maps, and target spaces, when the corresponding operator properties, spectral decay, and sampling compatibility are verified. Such extensions may require different proof techniques. Nevertheless, the experiments in Section \ref{sec:numerical-illustration} show similar phase transitions across diverse settings.
	\section{A Structured Three-Layer Illustration}
\label{sec:three-layer-illustration}
Although the $P$-trainable architecture is well-defined at arbitrary depth, shared upstream parameters couple downstream features through nonlinear composition in deeper networks. This complicates a direct deep-layer extension of the probabilistic kernel argument. The following structured three-layer example provides a complementary illustration of the phase transition pattern in approximation bounds and the advantage of nonlinear composition.

\begin{example}[Structured Three-Layer Illustration]
	\label{exm:structed-three-layer-compar}
	Let $\mathcal{X}=[-1,1]^{3}$, $\mu=\operatorname{Unif}([-1,1]^3)$ and consider the target function
	\[
	f^*(x)=(x_2^2-1/3)+\frac12(x_3^2-1/3).
	\]
	Let $\phi(x;\theta)=\operatorname{ReLU}(x;\theta)=(\theta^{\top}\tilde{x})_+$, where $\tilde{x}=(x,1)$. Consider the three-layer network
	\begin{equation}
        \nonumber
		f_{3}(x)
		=\beta_0+
		\sum_{j=1}^{10}\sum_{\ell=1}^{20}\beta_{j,\ell}
		\omega_{j,\ell}\big(\phi(x;\theta_{2j-1})
		+\phi(x;\theta_{2j})-b_{j,\ell}\big)_{+}.
	\end{equation}
	Here $\beta_0$ is a freely trainable output bias. The 20 first-layer neurons form 10 disjoint pairs,
	each feeding a separate block of 20 second-layer neurons. Each second-layer neuron receives inputs only from its assigned pair, with the same weight on both incoming edges.
	
	\mypara{Parameter Configuration.} This structured illustration varies the trainable count only in
	the first hidden layer, rather than using a common trainable proportion across hidden layers. For $m_{\mathrm{free}}\in\{0,\ldots,20\}$, the first $m_{\mathrm{free}}$ first-layer parameters are trainable in $\mathbb R^4$. The remaining parameters
	$\theta_i=(\theta_{i,1},\theta_{i,2},
	\theta_{i,3},\theta_{i,4})$
	are sampled independently from $\rho$, with
	\begin{equation}
        \nonumber
		\begin{aligned}
			\mathbb P\bigl((\theta_{i,1},\theta_{i,2},\theta_{i,3})=\mathbf e_1\bigr)
			&=\mathbb P\bigl((\theta_{i,1},\theta_{i,2},\theta_{i,3})=-\mathbf e_1\bigr)
			=0.4999,\\
			\mathbb P\bigl((\theta_{i,1},\theta_{i,2},\theta_{i,3})=\mathbf e_k\bigr)
			&=\mathbb P\bigl((\theta_{i,1},\theta_{i,2},\theta_{i,3})=-\mathbf e_k\bigr)
			=0.00005,~k=2,3,
		\end{aligned}
	\end{equation}
	where $\mathbf e_k$ is the $k$-th standard basis vector. Each bias $\theta_{i,4}\sim\operatorname{Unif}[-1,1]$
	is independent of the weight vector. The second-layer parameters are independent draws of
	$\omega_{j,\ell}\sim\operatorname{Unif}[1,2]$ and
	$b_{j,\ell}\sim\operatorname{Unif}[0,1]$, independently of the first-layer initialization. All sampled parameters remain fixed, while the trainable first-layer parameters and output parameters are optimized.
	
	\mypara{Two-layer Comparison.} To additionally examine the approximation effect of nonlinear composition alongside the phase transition, consider the two-layer network
	\[ f_2(x)=\beta_0^{(2)}+\sum_{i=1}^{220}\beta_i^{(2)}\phi(x;\theta_i^{(2)}), \]
	with the same total hidden-neuron and trainable-neuron counts. Its first $m_{\mathrm{free}}$ hidden parameters are trainable. The remaining hidden parameters are also sampled independently from $\rho$ and fixed. All output parameters are optimized. Figure \ref{fig:three-layer-comparison} compares these two architectures with four trainable hidden neurons.
		\begin{figure}[t]
		\centering
		\resizebox{\linewidth}{!}{%
			\begin{tikzpicture}[
				>=Latex,font=\small,text=black,
				neuron/.style={circle,minimum size=3mm,inner sep=0pt,line width=.5pt},
				free/.style={neuron,draw=blue!65!black,fill=blue!20},
				fixed/.style={neuron,draw=black!55,fill=black!13},
				other/.style={neuron,draw=black!70,fill=white},
				module/.style={rectangle,rounded corners=2pt,draw=black!55,
					minimum width=1.75cm,minimum height=.86cm,
					inner sep=0pt,line width=.5pt},
				free module/.style={module,draw=blue!65!black},
				connection/.style={->,draw=black!50,line width=.55pt},
				output connection/.style={connection,draw=blue!65!black},
				panel/.style={draw=black!55,densely dashed,
					rounded corners=3pt,line width=.55pt}
				]
				
				\begin{scope}
					\draw[panel] (-.55,-3.02) rectangle (7.5,2.64);
					\node at (0,2.22) {Input};
					\node at (2.15,2.22) {Hidden Layer 1};
					\node at (4.6,2.22) {Hidden Layer 2};
					\node at (6.8,2.22) {Output};
					
					\node[module,minimum width=.65cm,minimum height=1.45cm]
					(a-input) at (0,-.2) {};
					\foreach \dy in {-.42,0,.42}
					\node[other] at (0,{-.2+\dy}) {};
					\node[free] (a-output) at (6.8,-.2) {};
					
					\foreach \j/\cy/\kind in {
						1/1.4/free,2/.4/free,3/-.6/fixed,10/-1.95/fixed}{
						\ifnum\j<3
						\node[free module] (a-pair-\j) at (2.15,\cy) {};
						\else
						\node[module] (a-pair-\j) at (2.15,\cy) {};
						\fi
						\foreach \dy in {-.22,.22}
						\node[\kind] at (1.6,{\cy+\dy}) {};
						\node[align=center] at (2.4,\cy) {Pair\\$\j$};
						
						\node[module] (a-block-\j) at (4.6,\cy) {};
						\foreach \dy in {-.27,.27}
						\node[fixed] at (4.05,{\cy+\dy}) {};
						\foreach \dy in {-.065,0,.065}
						\fill[black] (4.05,{\cy+\dy}) circle (.4pt);
						\node[align=center] at (4.85,\cy) {Block\\$\j$};
					}
					
					\foreach \cx in {2.15,4.6}
					\foreach \cy in {-1.16,-1.27,-1.38}
					\fill[black] (\cx,\cy) circle (.55pt);
					
					\begin{scope}[on background layer]
						\foreach \j in {1,2,3,10}{
							\draw[connection] (a-input.east)--(a-pair-\j.west);
							\draw[connection] (a-pair-\j.east)--(a-block-\j.west);
							\draw[output connection] (a-block-\j.east)--(a-output);
						}
					\end{scope}
					
					\node at (2.15,-2.75) {20 Neurons};
					\node at (4.6,-2.75) {200 Neurons};
					\node[font=\normalsize] at (3.475,-3.43)
					{(a) Three-layer Network $f_{3}$};
				\end{scope}
				
				\draw[->,line width=.9pt] (7.9,-.2)--(11.2,-.2);
				\node[align=center] at (9.55,.2)
				{Move 200 Fixed Neurons};
				\node[align=center] at (9.55,-.7) {To Hidden Layer 1 \\ (\textcolor{red}{Same Neuron Budget})};
				
				\begin{scope}[shift={(12.15,0)}]
					\draw[panel] (-.55,-3.02) rectangle (5.4,2.64);
					\node at (0,2.22) {Input};
					\node at (2.3,2.22) {Hidden Layer 1};
					\node at (4.7,2.22) {Output};
					
					\node[module,minimum width=.65cm,minimum height=1.45cm]
					(b-input) at (0,-.2) {};
					\foreach \dy in {-.42,0,.42}
					\node[other] at (0,{-.2+\dy}) {};
					\node[free] (b-output) at (4.7,-.2) {};
					
					\node[free module,minimum width=1.8cm,minimum height=1.5cm]
					(b-free) at (2.3,1.1) {};
					\foreach \dy in {-.54,-.18,.18,.54}
					\node[free] at (1.75,{1.1+\dy}) {};
					\node[align=center] at (2.55,1.1) {Free\\4};
					
					\node[module,minimum width=1.8cm,minimum height=2.25cm]
					(b-fixed) at (2.3,-1.25) {};
					\foreach \dy in {-.88,-.5,.5,.88}
					\node[fixed] at (1.75,{-1.25+\dy}) {};
					\foreach \dy in {-.12,0,.12}
					\fill[black] (1.75,{-1.25+\dy}) circle (.5pt);
					\node[align=center] at (2.55,-1.25) {Fixed\\216};
					
					\begin{scope}[on background layer]
						\foreach \part in {free,fixed}{
							\draw[connection] (b-input.east)--(b-\part.west);
							\draw[output connection] (b-\part.east)--(b-output);
						}
					\end{scope}
					
					\node at (2.3,-2.75) {220 Neurons};
					\node[font=\normalsize] at (2.425,-3.43)
					{(b) Two-layer Network $f_{2}$};
				\end{scope}
				
				\node[free] at (5.5,-4.04) {};
				\node[anchor=west] at (5.78,-4.04) {Trainable Neuron};
				\node[fixed] at (10.05,-4.04) {};
				\node[anchor=west] at (10.33,-4.04) {Randomly Fixed Neuron};
				
			\end{tikzpicture}%
		}
		\caption{Architectural comparison with 220 hidden neurons and four trainable neurons. In (a), 20 first-layer neurons form 10 pairs, each feeding a block of 20 fixed neurons with tied incoming weights. In (b), the second-layer budget is reallocated to the single hidden layer. Only representative fixed neurons are shown.}
		\label{fig:three-layer-comparison}
	\end{figure}
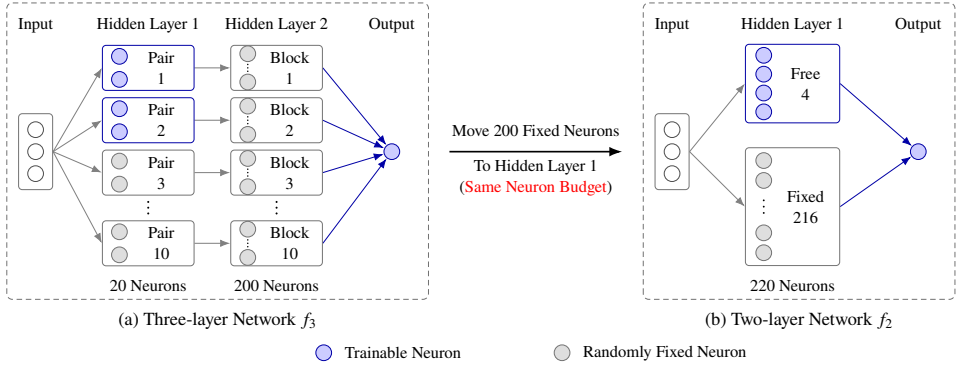
	
	\mypara{Error Analysis and Normalized Gain.} The approximation error is defined as 
	\[
	\mathcal E(f^{*},f_{d})=\mathbb E\inf_{\theta_{\mathrm{free}},\beta}\|f^*-f_d\|_{L_2(d\mu)}^2,~d\in\{2,3\},
	\]
	where $\beta$ denotes all output parameters, including bias $\beta_0$,
	and the expectation is over the fixed random parameters. Let $\mathcal E^{(d)}(m_{\mathrm{free}})$ denote its explicit upper bound derived in Proposition \ref{prop:structured-depth-upper-bounds}. For each architecture, we normalize the reduction in this upper
	bound relative to the endpoint $m_{\mathrm{free}}=20$:
	\[
	G^{(d)}(m_{\mathrm{free}})
	=\frac{\mathcal E^{(d)}(0)-\mathcal E^{(d)}(m_{\mathrm{free}})}
	{\mathcal E^{(d)}(0)-\mathcal E^{(d)}(20)},~d\in\{2,3\}.
	\]
	Table \ref{tab:three-layer-comparison} reports the upper bounds and normalized gains for two or four trainable neurons.

    \begin{table}[!t]
	\centering
	\caption{Approximation error upper bounds and normalized gains with two or four trainable hidden neurons. Both baseline bounds equal $1/9$. Each gain is normalized relative to its architecture's bound at $m_{\mathrm{free}}=20$.}
	\label{tab:three-layer-comparison}
    \small
	\setlength{\tabcolsep}{2pt}
	\renewcommand{\arraystretch}{1.2}
	\begin{tabular*}{\linewidth}{
			@{\extracolsep{\fill}}ccccc@{}
		}
		\toprule
		\textbf{Trainable Neurons}
		& \textbf{Architecture}
		& \textbf{Error Bound}
		& \textbf{Gain}
		& \textbf{Bound Ratio} \\
		\midrule
		
		\multirow{2}{*}{$2$ in the first layer}
		& Three-layer $f_3$
		& $2.22\times10^{-2}$
		& $0.7998$
		& \multirow{2}{*}{
			$\displaystyle
			\frac{\mathcal E^{(3)}(2)}{\mathcal E^{(2)}(2)}
			\approx0.9539$
		} \\
		
		& Two-layer $f_2$
		& $2.33\times10^{-2}$
		& $0.7902$
		& \\
		
		\midrule
		
		\multirow{2}{*}{$4$ in the first layer}
		& Three-layer $f_3$
		& $2.82\times10^{-5}$
		& $0.9997$
		& \multirow{2}{*}{
			$\displaystyle
			\frac{\mathcal E^{(3)}(4)}{\mathcal E^{(2)}(4)}
			\approx0.0206$
		} \\
		
		& Two-layer $f_2$
		& $1.37\times10^{-3}$
		& $0.9877$
		& \\
		
		\bottomrule
	\end{tabular*}
\end{table}
\end{example}
\mypara{Phase Transition.} Table \ref{tab:three-layer-comparison} shows that a small trainable group captures most of the normalized reduction in the approximation bounds. With only two trainable hidden neurons, both constructions attain approximately \(80\%\) of the reduction measured relative to their respective bounds at \(m_{\mathrm{free}}=20\). With four trainable neurons, the normalized gains reach \(0.9997\) and \(0.9877\), respectively. These results illustrate a rapid reduction in the bounds at small trainable counts, consistent with the phase transition pattern in the two-layer theory. Moreover, with four trainable neurons, the three-layer error upper bound is only \(2.06\%\) of its two-layer counterpart under matched neuron budgets. 

\mypara{Role of Depth.} The difference in the absolute upper bounds arises from nonlinear composition and feature reuse across layers. Both constructions approximate the quadratic components of the target function by piecewise linear interpolation, but use their trainable neurons differently. In the three-layer construction, four trainable first-layer neurons can generate \(|x_2|\) and \(|x_3|\). Each feeds a fixed block of 20 second-layer neurons, providing 20 basis functions for approximating its square. Thus, four trainable neurons adapt 40 downstream basis functions through the fixed second layer. In contrast, the two-layer construction uses its trainable neurons directly to implement interpolation basis functions, while leaving the fixed features unused. Therefore, this example illustrates how a small trainable upstream group can use fixed downstream random features through nonlinear composition to obtain a smaller explicit approximation bound.

The construction and error estimates of this example follow from a direct interpolation argument, detailed in Appendix \ref{subsec:proof-of-Example}. Beyond this structured example, a general theory characterizing phase transitions and quantifying the approximation advantages of nonlinear composition in multilayer \(P\)-trainable networks remains to be developed. It is currently unclear whether the probabilistic kernel approach can be extended to address these questions or whether different methods are needed.

	\section{Numerical Evidence of Phase Transitions}
\label{sec:numerical-illustration}
The preceding sections derive approximation bounds for two-layer $P$-trainable networks and analyze a structured three-layer construction. This section complements these results by examining the attained training losses across different target functions, widths, and depths. The numerical results show that the small-$P$ phase transition pattern also appears in the training performance under a fixed optimization budget.

\subsection{Experimental Setting}
\label{subsubsec:phase-transition-setting}
\mypara{Target Functions.} Three target functions cover different oscillation scales and input dimensions, as illustrated in Figure~\ref{fig:target-functions}.

The first target combines a relatively low-frequency oscillatory term with a Gaussian term and is sampled at 12,000 uniformly spaced points on $[-4,4]$:
\[ f_1^*(x)=x\sin x+e^{-0.1x^2}. \]

The second target contains multiple frequency components with higher-frequency oscillations and is sampled at 12,000 uniformly spaced points on $[-2\pi,2\pi]$:
\[ f_2^*(x)=0.1\sin(21x)+0.2\cos(17x+0.4)+0.35\sin(13x- 0.7)+0.5\cos(9x)+\sin(5x). \]

The third target extends the experiments to two dimensions. It combines a separable trigonometric term with a radial Gaussian term and is sampled on a $250\times250$ uniform grid over $[-2,2]^2$:
\[ f_3^*(x,y)=\sin(\pi x)\cos(\pi y)+e^{-(x^2+y^2)}. \]
All main experiments use deterministic uniform grids without shuffling. For each target, the training inputs, target values, and sample order are identical across configurations.
\begin{figure}[!t]
	\centering
	\subfloat[Target $f_1^*$.]{%
		\includegraphics[width=0.32\linewidth]{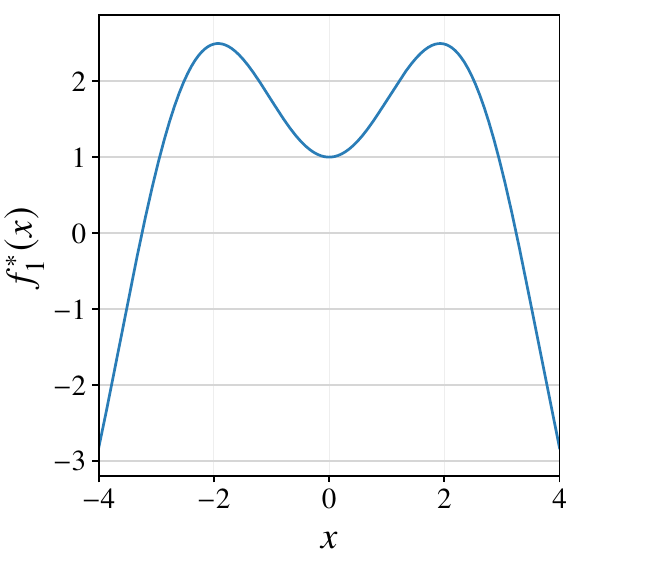}%
		\label{fig:target-f1}%
	}\hfill
	\subfloat[Target $f_2^*$.]{%
		\includegraphics[width=0.32\linewidth]{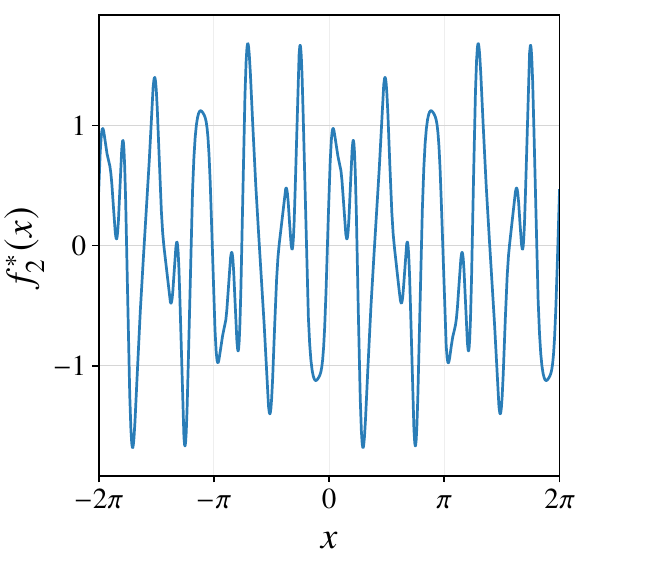}%
		\label{fig:target-f2}%
	}\hfill
	\subfloat[Target $f_3^*$.]{%
		\includegraphics[width=0.32\linewidth]{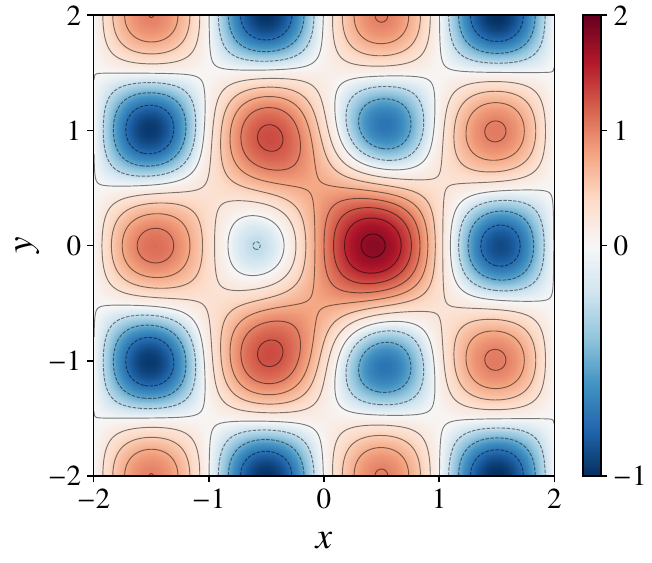}%
		\label{fig:target-f3}%
	}
	\caption{Target functions in the numerical experiments. Targets
		$f_1^*$ and $f_2^*$ are shown as curves, while $f_3^*$ is shown as a heatmap, with color indicating the function value.}
	\label{fig:target-functions}
\end{figure}

\mypara{Architecture Configurations.} 
The experiments follow the neuron-wise trainability convention in Section \ref{sec:setup}, considering $P$-trainable networks with equal  hidden-layer width $W\in\{32,64,128,256\}$. The depth $L$ includes the output layer, so each network has $L-1$ hidden layers. In each hidden layer, $PW$ neurons are trainable  through their incoming weights and biases, while the remaining hidden parameters stay fixed after initialization. All neurons remain active in the forward computation, and the output layer, including its bias, is fully trainable. The main experiments use ReLU activation.

\mypara{Experimental Set-up.} The following settings are used to compare trainable proportions under shared initialization and a fixed training budget.
\begin{itemize}[labelindent=0pt, leftmargin=*, topsep=0pt, itemsep=0pt]
	\item \textbf{Trainable Proportions.} The trainable proportion $P$ is evaluated on the following grid
	\[ P\in \mathcal{G}_{\mathrm{exp}}
	=\frac{1}{32}\Big\{0,~1,~2,~\cdots,7,~8,~10,~12,~14,~16,~20,~24,~28,~32\Big\}, \]
	with finer spacing at small $P$ to examine the phase transition. Every tested width is a multiple of $32$, so $PW$ is an integer at each grid point, without rounding.
	\item \textbf{Depth and Width Variation.} The depth experiments use $W=64$ and $L\in\{2, 3, 4, 5, 7, 9,\allowbreak 11, 13, 15\}$. The width experiments use $L=6$ and $W\in\{32,64,128,256\}$. 
	\item \textbf{Shared Initialization.} Configurations with the same network architecture and repetition index share the same initial values of all parameters across $P$. The repetition index starts at $0$. The initialization seeds are
	\[ \operatorname{Seed_{init}}(W) = 42 + (\operatorname{Repeat~Index}) + 100W, \]
	for width variation, and 
	\[ \operatorname{Seed_{init}}(L) = 42 + 100(\operatorname{Repeat~Index}) + (L-1), \]
	for depth variation. Both formulas are independent of $P$. A dedicated \texttt{torch.\allowbreak Generator} generates initial weights and biases in a fixed order. For all activation functions, weights use Xavier-uniform initialization \cite{glorot2010understanding} with gain $5/3$, and biases are sampled uniformly from the same interval as the corresponding layer's weights.
	\item \textbf{Repetitions and Training Budget.} Each configuration is trained for 12,000 steps and repeated 21 times with different initialization seeds. Table \ref{tab:experimental-configuration} summarizes the settings.
\begin{table}[!t]
	\centering
	\caption{Experimental configuration with ReLU activation.}
	\label{tab:experimental-configuration}
    \small
	\setlength{\tabcolsep}{2pt}
	\renewcommand{\arraystretch}{1.2}
	\begin{tabularx}{\linewidth}{@{}
			>{\raggedright\arraybackslash}p{0.21\linewidth}
			>{\raggedright\arraybackslash}X
			@{\hspace{1.2em}}
			>{\raggedright\arraybackslash}p{0.21\linewidth}
			>{\raggedright\arraybackslash}X
			@{}}
		\toprule
		\textbf{Item} & \textbf{Value} & \textbf{Item} & \textbf{Value} \\
		\midrule
	Batch size        & Full batch         & Optimizer   & Adam \\
	Training steps    & 12,000             & Base LR     & $10^{-3}$ \\
	Gradient clipping & $10.0$             & LR schedule & \texttt{ReduceLR\allowbreak{}OnPlateau} \\
	Early stopping    & Disabled           & LR factor   & $0.5$ \\
	Precision         & \texttt{float32}   & LR patience & $400$ \\
	Hardware          & Single NVIDIA A100 & LR threshold & $10^{-6}$ \\
	Repeats           & $21$               & Minimum LR  & $10^{-6}$ \\
		\bottomrule
	\end{tabularx}
\end{table}
\end{itemize}
\mypara{Evaluation Metrics.} Let $\mathcal{L}_{t,r}(P)$ denote the training mean-square error at step $t$ in repetition $r$. For each run, the best training loss within the budget $T$ is $\mathcal{E}_{r}(P)=\min_{0\le t\le T} \mathcal{L}_{t,r}(P)$. The losses are aggregated by averaging their logarithms across $R=21$ repetitions and transforming back to the original scale. The corresponding geometric mean loss and normalized training gain are defined as
\[
\mathcal E(P)
=\exp\Big(\frac1R\sum_{r=1}^{R}\log\mathcal E_r(P)\Big)=\Big(\prod_{r=1}^{R}\mathcal E_r(P)\Big)^{1/R},~\text{and}~G(P)=\frac{\mathcal{E}(0)-\mathcal{E}(P)}{\mathcal{E}(0)-\mathcal{E}(1)}.
\]
Thus $G(P)$ is computed from the aggregated losses and measures the fraction of the observed reduction between
$P=0$ and $P=1$ attained by partial training. 

\subsection{Experimental Results}
\label{subsubsec:phase-transition-results}
\begin{figure}[t]
	\centering
	\subfloat[Depth variation loss.]{
		\includegraphics[width=0.47\linewidth]
		{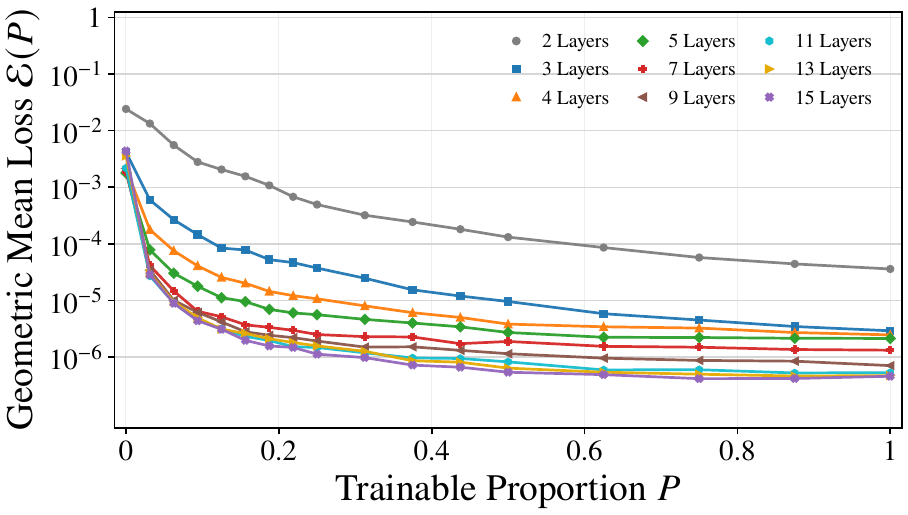}
		\label{fig:F1-depth-loss}
	}\hfill
	\subfloat[Depth variation gain.]{
		\includegraphics[width=0.47\linewidth]
		{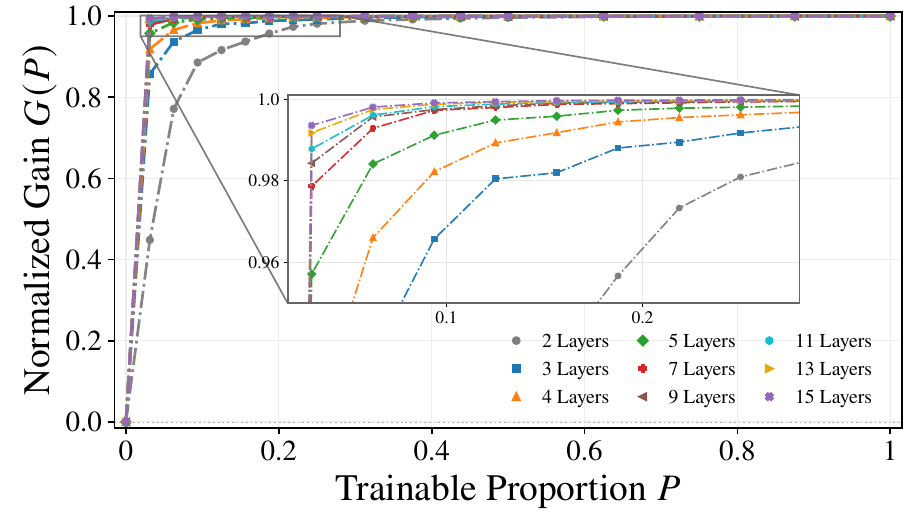}
		\label{fig:F1-depth-gain}
	}
	
	\subfloat[Width variation loss.]{
		\includegraphics[width=0.47\linewidth]
		{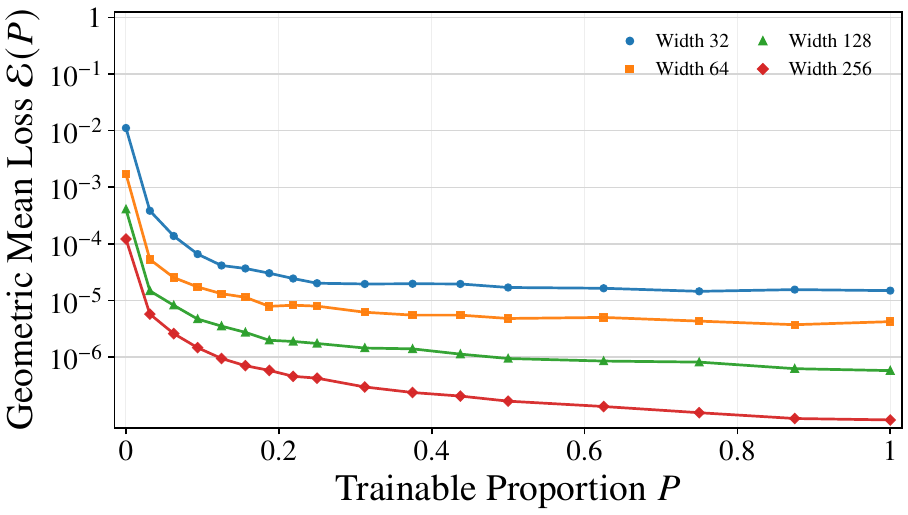}
		\label{fig:F1-width-loss}
	}\hfill
	\subfloat[Width variation gain.]{
		\includegraphics[width=0.47\linewidth]
		{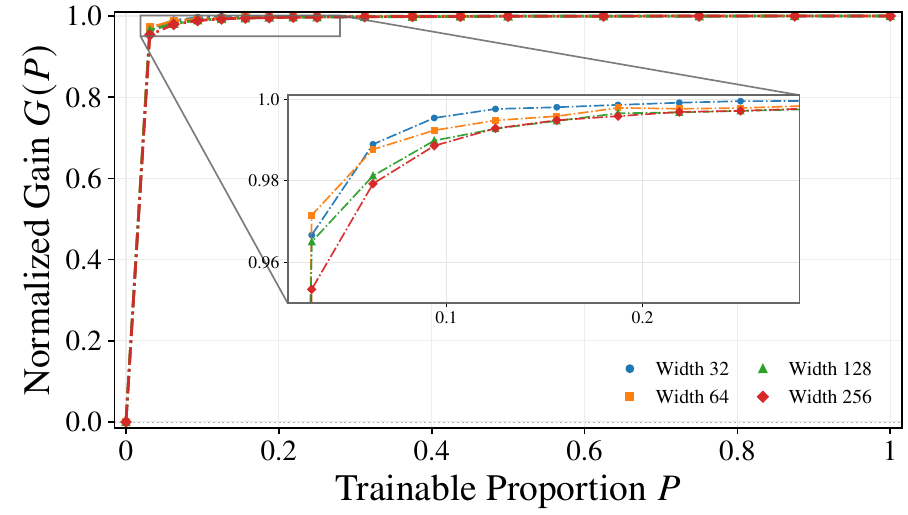}
		\label{fig:F1-width-gain}
	}
	\caption{Training results for $f_1^*$ with ReLU activation. The top and bottom rows show depth and width variation, respectively. The left column reports the geometric mean of the best training loss across 21 repetitions, and the right column shows the normalized gain. }
	\label{fig:numerical-target-1}
\end{figure}

\begin{figure}[p]
	\centering
	\subfloat[Depth variation loss.]{%
		\includegraphics[width=0.47\linewidth]
		{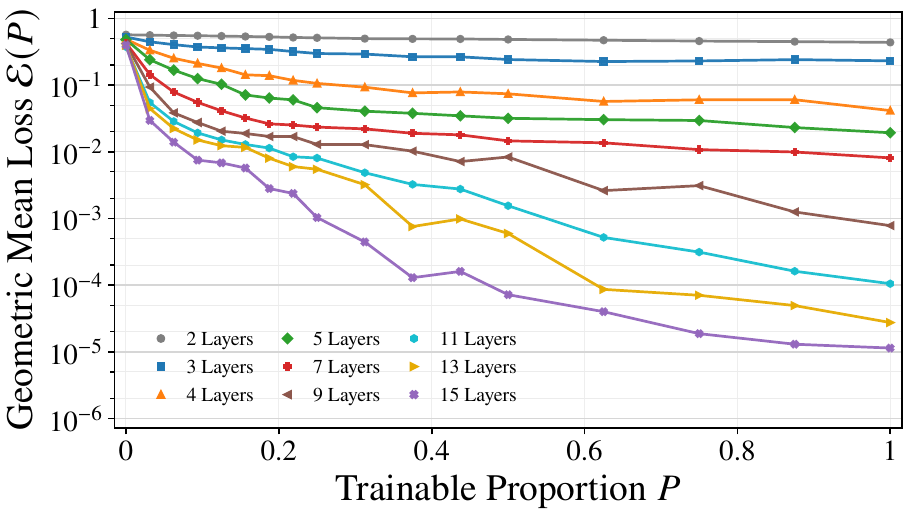}%
		\label{fig:F2-depth-loss}%
	}\hfill
	\subfloat[Depth variation gain.]{%
		\includegraphics[width=0.47\linewidth]
		{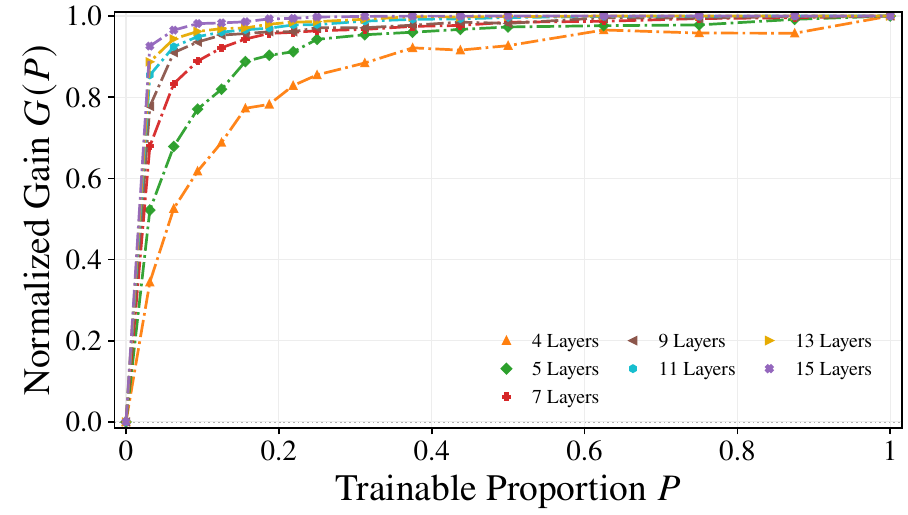}%
		\label{fig:F2-depth-gain}%
	}
	
	\subfloat[Width variation loss.]{%
		\includegraphics[width=0.47\linewidth]
		{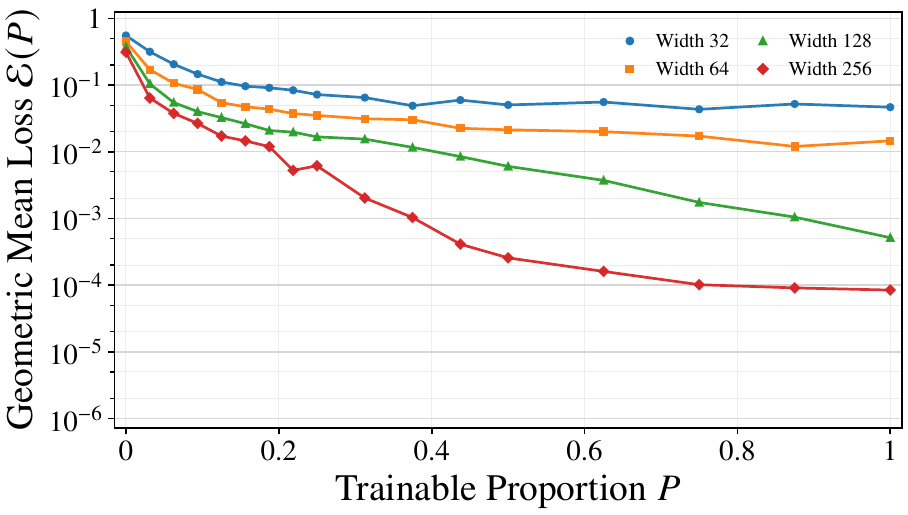}%
		\label{fig:F2-width-loss}%
	}\hfill
	\subfloat[Width variation gain.]{%
		\includegraphics[width=0.47\linewidth]
		{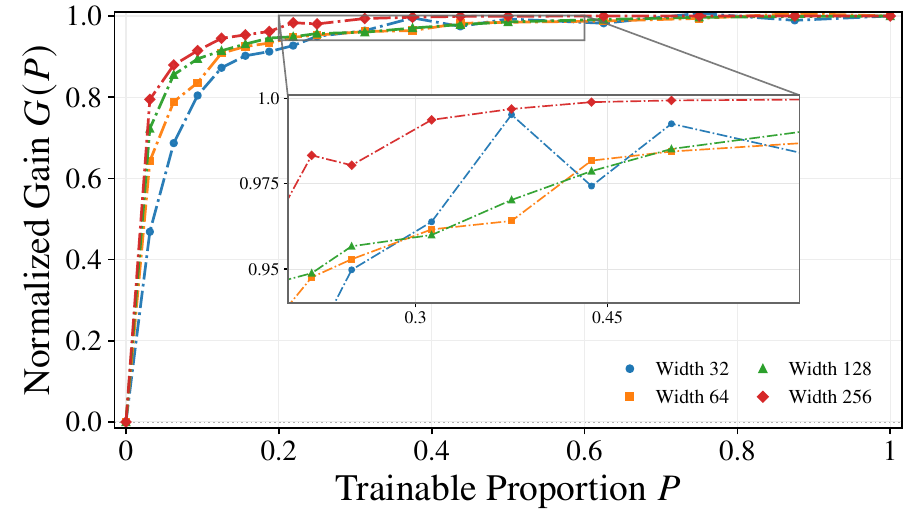}%
		\label{fig:F2-width-gain}%
	}
	\caption{Training loss and normalized gain for $f_2^*$.
		The layout and axis scales follow Figure \ref{fig:numerical-target-1}.}
	\label{fig:numerical-target-2}
	\setcounter{subfigure}{0}
	\subfloat[Depth variation loss.]{%
		\includegraphics[width=0.47\linewidth]
		{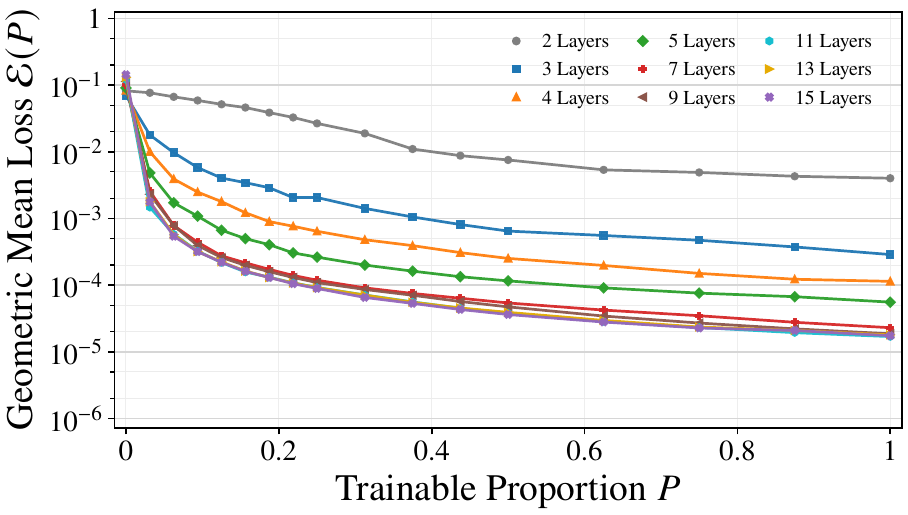}%
		\label{fig:F3-depth-loss}%
	}\hfill
	\subfloat[Depth variation gain.]{%
		\includegraphics[width=0.47\linewidth]
		{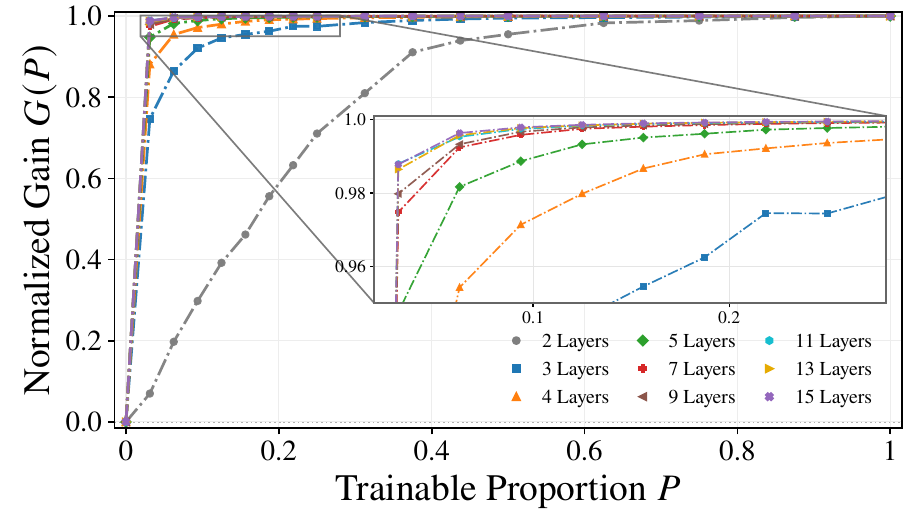}%
		\label{fig:F3-depth-gain}%
	}
	
	\subfloat[Width variation loss.]{%
		\includegraphics[width=0.47\linewidth]
		{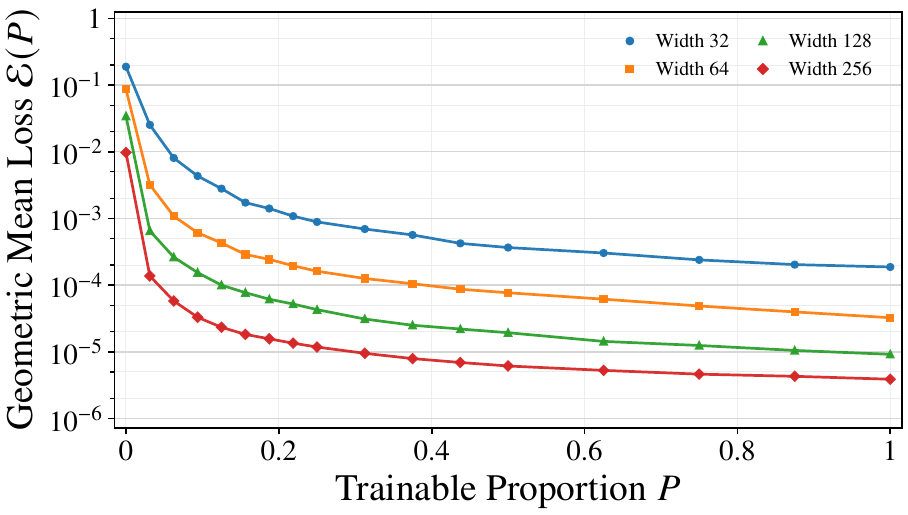}%
		\label{fig:F3-width-loss}%
	}\hfill
	\subfloat[Width variation gain.]{%
		\includegraphics[width=0.47\linewidth]
		{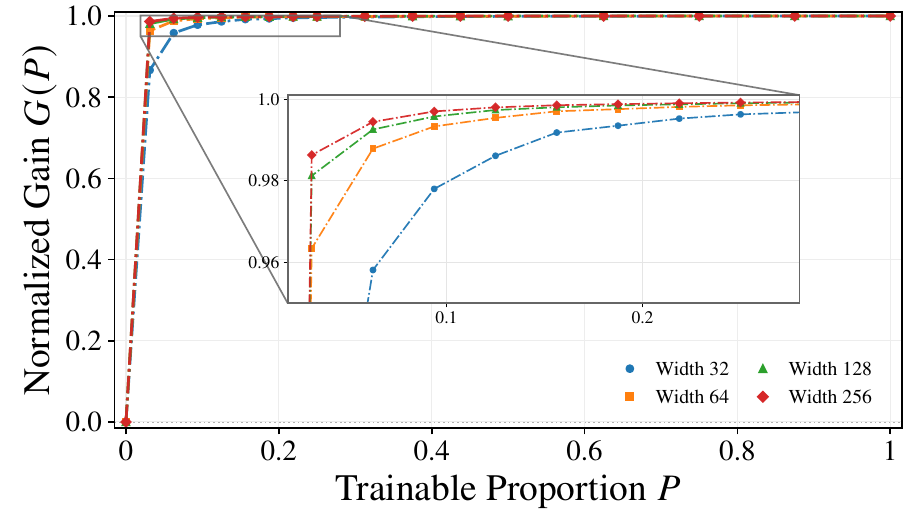}%
		\label{fig:F3-width-gain}%
	}
	\caption{Training loss and normalized gain for $f_3^*$.
		The layout and axis scales follow Figure \ref{fig:numerical-target-1}.}
	\label{fig:numerical-target-3}
\end{figure}

Figures \ref{fig:numerical-target-1}, \ref{fig:numerical-target-2}, and \ref{fig:numerical-target-3} report the results for $f_1^*$, $f_2^*$, and $f_3^*$, respectively. For each target, the top and bottom rows show depth and width variation, while the left and right columns show $\mathcal E(P)$ and $G(P)$. The horizontal axes and the gain axes use linear scales, while the loss axes use logarithmic scales. Insets enlarge the regions where the gain is close to $1$.

The two metrics provide complementary information. The absolute loss $\mathcal E(P)$ directly measures the attained errors, allowing comparisons across trainable proportions and architectures. The normalized gain $G(P)$ expresses the loss reduction on a common relative scale across configurations with different endpoint losses. A gain close to $1$ need not imply a small absolute error. The two metrics should therefore be interpreted together.

\mypara{Relation to Approximation Theory.} Before interpreting the numerical results, we note that the attained training losses depend on optimization, the finite training grids, and the training budget. They provide empirical indicators of approximation performance on these grids, so the comparison
with theory concerns qualitative trends as $P$ varies. 

\mypara{Results and Discussion.} The figures lead to the following observations.
\begin{itemize}[labelindent=0pt, leftmargin=*, topsep=0pt, itemsep=0pt]
	\item \textbf{Observed Phase Transition.} Most configurations exhibit substantial loss reductions at small $P$, followed by diminishing gains as the normalized gain curves approach a plateau. This pattern is particularly clear for $f_1^*$ and for the deeper networks tested on $f_2^*$ and $f_3^*$. These observations are qualitatively consistent with the small-$P$ phase transition in the theoretical approximation bounds.
	
	\item \textbf{Benefits of Partial Training.} Partial training can achieve substantial loss reductions with far fewer trainable hidden parameters. For example, in the width experiments for $f_1^*$ in Figure \ref{fig:numerical-target-1}\subref{fig:F1-width-gain}, all tested widths recover most of their endpoint improvement at $P=1/32$. Thus, obtaining most of the observed improvement does not require training all hidden parameters. This provides a useful trade-off between the number of trainable hidden parameters and  the attained loss reduction.
	\item \textbf{Task-dependent Deviation.} For the high-frequency target $f_2^*$, the two- and three-layer networks retain losses above $10^{-1}$ throughout the tested range of $P$ in Figure \ref{fig:numerical-target-2}\subref{fig:F2-depth-loss}, indicating poor fitting even at full training. Under the current settings,
	the three-layer network's loss fluctuates with $P$. Since neither architecture fits the target well under the current settings, their normalized gain curves are omitted from Figure~\ref{fig:numerical-target-2}\subref{fig:F2-depth-gain}.
\end{itemize}

\subsection{Results with Other Activation Functions}
The width and depth experiments are also conducted with GELU and $\tanh$ to examine whether similar phase transition patterns occur. Table \ref{tab:other-activations} summarizes representative training losses and normalized gains at $L=5$ and $W=64$. Both activations recover over $95\%$ of the endpoint loss reduction at $P=1/8$ for all three targets, consistent with the small-$P$ phase transition behavior observed with ReLU. The complete plots are reported in Appendix \ref{sec:addition-experiments}.
\begin{table}[!t]
	\centering
	\caption{Training losses and normalized gains for GELU and $\tanh$
		with depth $L=5$ and width $W=64$.
		Losses are geometric means of the best training losses across 21 repetitions.}
	\label{tab:other-activations}
    \small
	\setlength{\tabcolsep}{2pt}
	\renewcommand{\arraystretch}{1.2}
	\begin{tabular*}{\linewidth}{
			@{\extracolsep{\fill}}ccccccc@{}
		}
		\toprule
		\textbf{Activation} & \textbf{Target}
		& $\bm{\mathcal E(0)}$
		& $\bm{\mathcal E(1/32)}$
		& $\bm{\mathcal E(1)}$
		& $\bm{G(1/32)}$
		& $\bm{G(1/8)}$ \\
		\midrule
		
		\multirow{3}{*}{GELU} & $f_1^*$
		& $9.70\times10^{-5}$
		& $2.13\times10^{-6}$
		& $3.33\times10^{-7}$
		& $0.9814$ & $0.9927$ \\
		
		& $f_2^*$
		& $6.80\times10^{-1}$
		& $1.78\times10^{-1}$
		& $7.69\times10^{-4}$
		& $0.7390$ & $0.9723$ \\
		
		& $f_3^*$
		& $1.72\times10^{-1}$
		& $1.59\times10^{-3}$
		& $4.24\times10^{-6}$
		& $0.9908$ & $0.9996$ \\
		\midrule
		
		\multirow{3}{*}{$\tanh$} & $f_1^*$
		& $3.63\times10^{-3}$
		& $1.96\times10^{-5}$
		& $4.69\times10^{-7}$
		& $0.9947$ & $0.9997$ \\
		
		& $f_2^*$
		& $6.18\times10^{-1}$
		& $3.17\times10^{-1}$
		& $3.30\times10^{-4}$
		& $0.4867$ & $0.9552$ \\
		
		& $f_3^*$
		& $6.35\times10^{-2}$
		& $4.18\times10^{-3}$
		& $1.11\times10^{-5}$
		& $0.9343$ & $0.9973$ \\
		
		\bottomrule
	\end{tabular*}
	
	\caption{Training losses for $f_1^*$ with ReLU activation at depth $6$, keeping 32 trainable neurons per hidden layer while varying the number of fixed neurons. Losses are geometric means of the best training losses across 21 repetitions. Percentage reductions are relative to $W=32$.}
	\label{tab:randomly-fixed-part}
    \small
	\setlength{\tabcolsep}{2pt}
	\renewcommand{\arraystretch}{1.2}
	\begin{tabular*}{\linewidth}{
			@{\extracolsep{\fill}}cccccc@{}
		}
		\toprule
		\textbf{Width} & $\bm{P}$
		& \textbf{Trainable Per Layer}
		& \textbf{Fixed Per Layer}
		& $\bm{\mathcal E(P)}$
		& \textbf{Loss reduction (\%)} \\
		\midrule
		32 & $1$ & 32 & 0
		& $1.50\times10^{-5}$ & $0.00$ \\
		64 & $1/2$ & 32 & 32
		& $4.82\times10^{-6}$ & $67.76$ \\
		128 & $1/4$ & 32 & 96
		& $1.75\times10^{-6}$ & $88.33$ \\
		256 & $1/8$ & 32 & 224
		& $\bm{9.49\times10^{-7}}$ & $\bm{93.66}$ \\
		\bottomrule
	\end{tabular*}
\end{table}
\subsection{Expressive Role of Randomly Fixed Neurons}
The width experiments include configurations with the same number of trainable hidden neurons but  different numbers of randomly fixed neurons. Table \ref{tab:randomly-fixed-part} compares such configurations for $f_1^*$ at $L=6$, with $32$ trainable neurons per hidden layer.

\mypara{Effect of Fixed Neuron Count.} As the number of randomly fixed neurons increases, the observed loss decreases progressively, with a reduction of approximately $94\%$ from $W=32$ to $W=256$. The number of trainable neurons per hidden layer remains fixed, although the total number of trainable parameters varies with width. These results support the expressive role of the randomly fixed neurons.
	\section{Discussion and Conclusions}
\label{sec:discuss-conclusion}
This work studied the approximation capacity of partially trainable networks. For two-layer networks, the finite-width bounds quantify how sampling mismatch affects approximation as the trainable proportion varies. Under polynomial spectral decay and leverage compatibility, the rates show that a small trainable group can substantially reduce the bound when its mismatch is much smaller than that of the fixed group.
The spherical biased-ReLU analysis establishes the spectral decay
and two-sided bounds on the initialization mismatch, while a structured three-layer example illustrates a similar phase transition and shows how nonlinear composition allows trainable neurons to use randomly fixed downstream features to obtain a smaller explicit approximation bound. Numerical experiments show similar phase transition patterns across targets and architectures and further support the expressive role of the randomly fixed neurons.

The approximation bounds rely on a distributional construction of the trainable parameters. The leverage-score distribution provides an effective choice for deriving the bounds, but does not characterize the parameters produced by gradient-based training. Connecting this construction to the evolution of trainable parameters during optimization is an important direction for understanding when the predicted gains can be attained in training.

Besides, a general phase transition theory for deeper $P$-trainable networks remains to be developed. Changes in upstream trainable parameters also change the inputs to later neurons, including those with fixed parameters. This dependence between layers complicates a direct extension of current two-layer kernel techniques. Further analysis is needed to quantify how such interactions affect approximation as $P$ varies. Deterministic approximation techniques may help characterize more directly how a small trainable group uses the representation provided by the randomly fixed group.
	\appendix
	\section*{Overview of the Appendices}
\begin{itemize}[labelindent=0pt, leftmargin=*, topsep=0pt, itemsep=0pt]
	\item {Appendix \ref{sec:appendix-additional-t-notes}} summarizes the notation and collects supplementary technical remarks.
	\item {Appendix \ref{sec:appendix-proof-general-results}} proves the general approximation bounds and polynomial rates.
	\item {Appendix \ref{subsec:proo-of-specific}} proves the spectral and mismatch bounds for spherical ReLU features.
	\item {Appendix \ref{subsec:proof-of-Example}} derives error bounds for the structured two- and three-layer examples.
	\item {Appendix \ref{sec:addition-experiments}} presents additional experiments that complement the main results.
\end{itemize}

\begin{table*}[!t]
	\centering
	\caption{Notation used throughout the paper.}
	\renewcommand{\arraystretch}{1.2}
	{\small \begin{tabular}{c l}
			\toprule
			\makebox[1cm][c]{\textbf{Symbols}}&\makebox[8.5cm][l]{\textbf{Descriptions}}\\
			\midrule
			$\mathcal{X}~\&~\Theta$&The input domain and the parameter space.\\
			$\phi(x;\theta)$&Feature map and its reweighted form $\phi_{P}(x;\theta)$. \\
			\midrule
			$\mathcal{F}(W,L)$& The neural network class with $W$ width and depth $L$ in \eqref{eq:Full-FNN}.\\
			$\mathcal{F}_{P}(W,L)$&The partially trainable counterpart of $\mathcal{F}(W,L)$ in \eqref{eq:p-trainable-neural-network}.\\
			$k_{\rho}(x,x^{\prime})~\&~T_{\rho}$&Kernel induced by $\phi(x;\theta)$ \eqref{eq:induced-kernel} and the corresponding integral operator \eqref{eq:kernel-integral-operator}. \\
			$\mathcal{H}_{\rho}$& Reproducing Kernel Hilbert Space (RKHS) induced by $\rho$ in \eqref{eq:target-function-statement}. \\
			\midrule
			$\rho$& Reference measure defining the target RKHS $\mathcal H_\rho$.\\
			$\rho_0$ & Prescribed initialization distribution for fixed neurons in Definition \ref{def:sampling-scheme}.\\
			$\rho_{1}$ & Adjustable sampling distribution for trainable neurons in Definition \ref{def:sampling-scheme}.\\
			$\rho_{P}~\&~\gamma_{P}$& Mixture measure and its density relative to $\rho$ in \eqref{eq:distributional-constr.}\\
			$s_{\lambda,P}(\theta)$ & The ridge leverage score defined in \eqref{eq:leverage-score}.\\
			$d_{\max}(\gamma,\lambda)$ &The maximal ridge leverage defined in \eqref{eq:definition-of-d-max}.\\
			$\kappa_{\gamma_{P}}(\lambda),\kappa_{0}^{*},\kappa_{1}^{*}$& Mismatch factor of $\rho_{P}$ \eqref{eq:definition-of-kappa-first}; Admissible constants $\kappa_{0}^{*}$ for $\rho_{0}$, and $\kappa_{1}^{*}$ for $\rho_{1}$.\\
			$P~\&~G(P)$& Trainable proportion $P$ and normalized gain $G(P)$ in \eqref{eq:the-normalized-gain}. \\
			$Y_{k,j}~\&~\YY_{k}$& The real orthonormal basis $\{Y_{k,j}\}$ of degree $k$ spherical harmonics space $\YY_{k}$.\\
			$\eta$& Eigenvalues ordered by magnitude (with multiplicity) $\eta$ in \eqref{eq:spectral-decomposition}.\\
			$\xi$&Eigenvalues $\xi$ indexed by spherical harmonic degree in \eqref{eq:relu-degree-eigenvalue}.\\ 
			\midrule
			$\lesssim~\&~\gtrsim$& $A\lesssim B$ and $B\gtrsim A$ denote $A\le cB$ for some positive constant $c>0$. \\
			$\asymp$& $A\asymp B$ means $c_{1}A\le B\le c_{2}A$ for some positive constants $c_{1},c_{2}>0$.\\
			$\mathcal{O}(\cdot)~\&~\tilde{\mathcal{O}}(\cdot)$& $\mathcal{O}(\cdot)$ is used to hide constant factors, while $\tilde{\mathcal{O}}(\cdot)$ ignores logarithmic factors.\\
			$\norm{\cdot}_{L_{2}(d\mu)}$&$L_2$ norms under input measure and $\norm{\cdot}_{L_{2}(d\rho)}$ under parameter measure.\\
			$a\otimes_{L_2(d\mu)}b$& $(a\otimes_{L_2(d\mu)} b)g=\langle b,g\rangle_{L_{2}(d\mu)} a$ for any $g\in L_{2}(d\mu)$.\\
			$A\preccurlyeq B, B\succcurlyeq A$&Positive semidefinite order, meaning that $B-A$ is positive semidefinite.\\
			\bottomrule
	\end{tabular}}
	\label{tab:summary-of-notations}
\end{table*}

\section{Additional Technical Notes}
\label{sec:appendix-additional-t-notes}
\subsection{The Summary of Notations}
\label{sec:notation-summary}
Table \ref{tab:summary-of-notations} summarizes the main notations used in the paper.

\subsection{Clarifications of the Assumptions}
\label{subsec:assumption-discussions}
For convenience, we restate all assumptions used in the main text before discussing their roles and necessity.

\assumpoperatorRegularity*
\mypara{Clarification.} Assumption \ref{assumption:operator} provides the operator regularity needed for the general analysis. Joint measurability and square
integrability of $\phi$ make $S_\rho$ a Hilbert--Schmidt operator. Consequently, $T_\rho=S_\rho S_\rho^*$ is self-adjoint, positive
semi-definite, trace-class, and compact. These properties justify the spectral decomposition and ensure that $d(\lambda)$ is finite for every
$\lambda>0$. The assumption is general and directly verifiable, for example, when the feature map is bounded on probability spaces.

\assumQuality*
\mypara{Clarification.} Assumption \ref{assumption:mismatch-used} is a uniform compatibility condition between each candidate sampling distribution density $\gamma_{{i}}$ and the optimal ridge leverage density $\gamma_{\lambda}^{*}$. Indeed, 
\[ \kappa_{\gamma_i}(\lambda) = \frac{d_{\max}(\gamma_{{i}},\lambda)}{d(\lambda)}= \operatorname*{ess\,sup}_{\theta\in\Theta} \frac{\gamma_\lambda^*(\theta)}{\gamma_i(\theta)}. \]
Hence, the assumption requires $\gamma_i(\theta) \geq\gamma_\lambda^*(\theta)/\kappa_i^*$ for $\rho$-almost every $\theta$ at each $\lambda\in(0,\lambda_{0}]$. Its role is to bound the
sampling cost $d_{\max}(\gamma,\lambda)$ by a constant multiple of the effective dimension $d(\lambda)$ in the rate analysis. The
ordering $\kappa_1^*\leq\kappa_0^*$ further expresses the advantage of the adjustable distribution at the level of these bounds.

This condition is closely related to the compatibility condition of \cite[Assumption 2]{rudi2017generalization}, which, in our notation, controls the maximal ridge leverage score by $d_{\max}(\gamma,\lambda)\le C_{\gamma}\lambda^{-C^{*}}$ for some constants $C_{\gamma}>0$, and $C^{*}\in[0,1]$. Our assumption likewise controls $d_{\max}(\gamma_i,\lambda)$, but requires a stronger constant-factor bound
$d_{\max}(\gamma_i,\lambda)\leq\kappa_i^*d(\lambda)$. Under Assumption \ref{assumption:eigendecay-used}, we have
$d(\lambda)\asymp\lambda^{-1/\alpha}$. Hence our condition implies a bound of the former form with $C^*=1/\alpha$ for sufficiently small $\lambda$.

\assumpPoly*

\mypara{Clarification.} Assumption \ref{assumption:eigendecay-used} specifies the spectral scale needed to turn the general approximation bound into a polynomial
rate. Its two-sided bounds imply $d(\lambda)\asymp\lambda^{-1/\alpha}$ (Derivation in \eqref{eq:bound-of-d(lamda)-lambda}). The upper spectral bound controls the approximation rate, while the lower bound ensures that $\alpha$ describes the actual polynomial scale rather than only an upper envelope. The trace-class property of  $T_{\rho}$ alone does not establish this two-sided decay, but it can be verified for the particular kernel, as in the spherical ReLU example in Proposition \ref{prop:spectrum-biased-relu-kernel}.

\section{Proofs of Main Results in Section \ref{sec:main-results} -- General Bounds}
\label{sec:appendix-proof-general-results}
\subsection*{Outline}
\begin{itemize}[labelindent=0pt, leftmargin=*, topsep=0pt, itemsep=0pt]
	\item Subsection \ref{subsec:auxiliary-t-pre} proves an operator Bernstein inequality with a variance proxy and uses it to establish the regularized operator concentration for the two independently sampled feature groups.
	
	\item Subsection \ref{subsec:proof-main-genral-results} first verifies kernel and operator invariance, then combines the concentration result with the ridge construction to prove the general approximation theorem. It subsequently proves the mixture leverage bound and combines this bound with sampling compatibility and polynomial spectral decay to derive polynomial approximation rates.
\end{itemize}
\subsection{Auxiliary Theoretical Preliminaries}
\label{subsec:auxiliary-t-pre}
To prove the general approximation result in Theorem \ref{thm:general-approx.-bounds}, we utilize the following operator Bernstein inequality and the partial order bound on the deviation $\widehat{T}_{\rho}-T_{\rho}$ between empirical and integral operators.
\begin{lemma}[Operator Bernstein Inequality With Variance Proxy]
	\label{lemma:modified-bernstein-inequality}
	Let $X_1,\ldots,X_m$ be independent, centered, self-adjoint Hilbert--Schmidt random operators on a separable Hilbert space. Suppose that $\norm{X_j}_{\operatorname{op}}\le U$ almost surely for $j\in[m]$ and some positive $U$. Then for $\sigma^{2}>0$ with $\sigma^{2}\ge\norm{\sum_{j=1}^{m}\EE[(X_{j})^{2}]}_{\operatorname{op}}$ and any $t>0$ satisfying $t^{2}\ge \sigma^{2}+{Ut}/{3}$, we have
	\[
	\mathbb{P}\Big(
	\Big\Vert{\sum_{j=1}^m X_j}\Big\Vert_{\operatorname{op}}>t
	\Big)
	\le
	14\frac{\tr (\sum_{j=1}^{m}\EE[(X_{j})^{2}] )}{\sigma^2}
	\exp \Big(
	-\frac{t^2/2}{\sigma^2+Ut/3}
	 \Big).
	\]
\end{lemma}
\begin{proof}[\textbf{Proof of Lemma \ref{lemma:modified-bernstein-inequality}}]
	The result follows from the intermediate estimate (3.9) in the proof of Theorem 3.1 of \cite{minsker2017some}, which gives
	\[ \mathbb{P}\Big(
	\eta_{\max}\Big({\sum_{j=1}^m X_j}\Big)>t
	\Big)
	\le
	\frac{\tr (\sum_{j=1}^{m}\EE[(X_{j})^{2}])}{\sigma^2}
	\exp\Big(
	-\frac{t^2/2}{\sigma^2+Ut/3}
	\Big)\Big(1+\frac{6(\sigma^{2}+Ut/3)^{2}}{t^{4}}\Big). \]
	The condition $t^{2}\ge \sigma^{2}+{Ut}/{3}$ implies the last factor is at most $7$. Apply the same estimate to $-\sum_{j}X_{j}$ and use
	a union bound to obtain the factor $14$. Further with the operator extension of Minsker's Bernstein inequality \cite[Section 3.2]{minsker2017some}, we complete the proof.
\end{proof}
\begin{proposition}[Regularized Operator Concentration]
	\label{prop:partial-order-deviation-empirical-continu} Let $\lambda>0$, $t\in(0,1)$ and assume that $0\le d_{max}(\gamma_{P},\lambda)<\infty$. Suppose that $\gamma_P(\theta)>0$ for $\rho$-almost every
	$\theta$ with $\norm{\phi(\cdot;\theta)}_{L_{2}(d\mu)}>0$. Let $P={m_{\mathrm{free}}}/{m}, 1-P={m_{\mathrm{fix}}}/{m}, m=m_{\mathrm{free}}+m_{\mathrm{fix}}\ge 1$. Let the hidden parameters be sampled according to
	Definition \ref{def:sampling-scheme}. Define  
	\[\widehat T_{\rho}=\frac{1}{m}\Big(\sum_{j=1}^{m_{\mathrm{free}}}\phi_P(\cdot;\theta_j)\otimes\phi_P(\cdot;\theta_j)+\sum_{j=1}^{m_{\mathrm{fix}}}\phi_P(\cdot;\bar\theta_j)\otimes\phi_P(\cdot;\bar\theta_j)\Big).
	\]
	Then for any $\delta\in(0,1)$, if
	\[ m\ge\max \Bigg\{d_{\max}(\gamma_{P},\lambda) \Big(\frac{1}{t^2}+\frac{1}{3t} \Big), 2d_{\max}(\gamma_{P},\lambda)
	 \Big(\frac{1}{t^2}+\frac{1}{3t} \Big)
	\log\frac{14d(\lambda)}{\delta}
	 \Bigg\}, \] 
	with probability at least $1-\delta$, we have 
	\[ ({T}_{\rho}+\lambda I)^{-1/2}(\widehat{T}_{\rho}-{T}_{\rho})({T}_{\rho}+\lambda I)^{-1/2}\succcurlyeq -tI,~\text{equivalently,}~\widehat T_\rho+\lambda I
	\succcurlyeq
	(1-t)(T_\rho+\lambda I). \]
\end{proposition}
\begin{proof}[\textbf{Proof of Proposition \ref{prop:partial-order-deviation-empirical-continu}}]
	We note that if $d_{\max}(\gamma_P,\lambda)=0$, then $\phi_P(\cdot;\theta)=0$ for $\rho_P$-almost every $\theta$, and $T_\rho=\widehat T_\rho=0$ almost surely. The conclusion holds for every $m\ge1$. Thus, below assume $d_{\max}(\gamma_P,\lambda)>0$. We define the rank-one operator as
	\[ A_{P}(\theta)=({T}_{\rho}+\lambda I)^{-1/2} [\phi_{P}(\cdot;\theta)\otimes_{L_2(d\mu)}\phi_{P}(\cdot;\theta) ]({T}_{\rho}+\lambda I)^{-1/2}. \]
	And the two related self-adjoint operators $X_{j}^{\mathrm{free}}$ and $X_{j}^{\mathrm{fix}}$ are defined as
	\begin{equation}
		\nonumber
		\begin{aligned}
			&X_{j}^{\mathrm{free}}=\frac{1}{m} \bigl(A_{P}(\theta_{j})-\EE_{\theta\sim {\rho_{1}}}[A_{P}(\theta)] \bigl),~j\in[m_{\mathrm{free}}],\\
			&X_{j}^{\mathrm{fix}}=\frac{1}{m} \bigl(A_{P}(\bar\theta_{j})-\EE_{\bar\theta\sim {\rho_{0}}}[A_{P}(\bar\theta)] \bigl),~j\in[m_{\mathrm{fix}}].
		\end{aligned}
	\end{equation}
	By construction and the fixed group sizes, we have
	\[ \begin{aligned}
		&\sum_{j=1}^{m_{\mathrm{free}}}X_{j}^{\mathrm{free}}+\sum_{j=1}^{m_{\mathrm{fix}}}X_{j}^{\mathrm{fix}}\\
		&=\frac{1}{m}\Big(\sum_{j=1}^{m_{\mathrm{free}}}A_{P}(\theta_{j})+\sum_{j=1}^{m_{\mathrm{fix}}}A_{P}(\bar\theta_{j})\Big)-P\EE_{\theta\sim{\rho_{1}}}[A_{P}(\theta)]-(1-P)\EE_{\theta\sim{\rho_{0}}}[A_{P}(\theta)]\\
		&=({T}_{\rho}+\lambda I)^{-1/2}\widehat{T}_{\rho}({T}_{\rho}+\lambda I)^{-1/2}-\EE_{\theta\sim\rho_{P}}[A_{P}(\theta)]=({T}_{\rho}+\lambda I)^{-1/2}(\widehat{T}_{\rho}-{T}_{\rho})({T}_{\rho}+\lambda I)^{-1/2}.
	\end{aligned} \]
	Therefore, we bound the operator norm of  $\sum_{j=1}^{m_{\mathrm{free}}}X_{j}^{\mathrm{free}}+\sum_{j=1}^{m_{\mathrm{fix}}}X_{j}^{\mathrm{fix}}$ with high probability to get the final result. 
	
	First, recall the following notation in \eqref{eq:effective-dimension} and \eqref{eq:definition-of-d-max},
	\begin{equation}
        \nonumber
		\begin{aligned}
			d(\lambda)&=\operatorname{tr}(T_{\rho}(T_{\rho}+\lambda I)^{-1})\\
			&=\int_{\Theta}\operatorname{tr}((T_{\rho}+\lambda I)^{-1/2}[\phi_{P}(\cdot;\theta)\otimes\phi_{P}(\cdot;\theta)](T_{\rho}+\lambda I)^{-1/2})d\rho_{P}(\theta)\\
			&=\int_{\Theta}\frac{\langle \phi(\cdot;\theta),(T_{\rho}+\lambda I)^{-1}\phi(\cdot;\theta)\rangle}{\gamma_{P}(\theta)}\gamma_{P}(\theta)d\rho(\theta)\le d_{\max}(\gamma_{P},\lambda).
		\end{aligned}
	\end{equation}
	For $\rho_{P}$-almost every $\theta$, we have
	\begin{equation}
		\label{eq:bound-of-AP-theta}
		\begin{aligned}
			\norm{A_{P}(\theta)}_{\operatorname{op}}&=\norm{\frac{1}{\gamma_{P}(\theta)}({T}_{\rho}+\lambda I)^{-1/2}[\phi(\cdot;\theta)\otimes_{L_2(d\mu)}\phi(\cdot;\theta)]({T}_{\rho}+\lambda I)^{-1/2}}_{\operatorname{op}}\\
			&=\frac{1}{\gamma_{P}(\theta)}\norm{[({T}_{\rho}+\lambda I)^{-1/2}\phi(\cdot;\theta)]\otimes_{L_2(d\mu)}[({T}_{\rho}+\lambda I)^{-1/2}\phi(\cdot;\theta)]}_{\operatorname{op}}\\
			&=\frac{\norm{({T}_{\rho}+\lambda I)^{-1/2}\phi(\cdot;\theta)}^{2}_{L_{2}(d\mu)}}{\gamma_{P}(\theta)}=\frac{\langle \phi(\cdot;\theta),(T_{\rho}+\lambda I)^{-1}\phi(\cdot;\theta)\rangle}{\gamma_{P}(\theta)}\\
			&\le d_{\max}(\gamma_{P},\lambda),
		\end{aligned}
	\end{equation}
	where $\norm{\cdot}_{\operatorname{op}}$ denotes the operator norm,
which here equals the largest singular value. We use the identity
$\norm{u\otimes_{L_2(d\mu)}u}_{\operatorname{op}}
=\norm{u}_{L_2(d\mu)}^{2}$ for every $u\in L_2(d\mu)$. By the construction ($\phi_{P}\otimes\phi_{P}$ is positive) and \eqref{eq:bound-of-AP-theta}, we have 
	\begin{equation}
		\label{eq:bound-of-AP-theta-2222}
		0 \preccurlyeq A_{P}(\theta) \preccurlyeq d_{\max}(\gamma_{P},\lambda) I.
	\end{equation}
	
	Second, by construction, we have $\EE_{\theta\sim \rho_{1}}[X_{j}^{\mathrm{free}}]=0$ and $\EE_{\bar\theta\sim\rho_{0}}[X_{j}^{\mathrm{fix}}]=0$. Thus these operators are mutually independent, centered and identically distributed within each group, although they are not identically distributed across the two groups.
	
	And for $X_{j}^{\mathrm{free}}$ and $X_{j}^{\mathrm{fix}}$, by \eqref{eq:bound-of-AP-theta-2222}, we have 
	\[ \begin{aligned}
		-\frac{1}{m}d_{\max}(\gamma_{P},\lambda) I \preccurlyeq X_{j}^{\mathrm{free}},X_{j}^{\mathrm{fix}} \preccurlyeq \frac{1}{m}d_{\max}(\gamma_{P},\lambda) I,
	\end{aligned} \]
	which means $\norm{X_{j}^{\mathrm{free}}}_{\operatorname{op}}\le \frac{d_{\max}(\gamma_{P},\lambda)}{m},j\in[m_{\mathrm{free}}]$ and $\norm{X_{j}^{\mathrm{fix}}}_{\operatorname{op}}\le \frac{d_{\max}(\gamma_{P},\lambda)}{m},j\in[m_{\mathrm{fix}}]$.
	
	Third, we need to estimate the operators' variance $\EE[(X_{j}^{\mathrm{free}})^{2}]$ and $\EE[(X_{j}^{\mathrm{fix}})^{2}]$. For $\EE[(X_{j}^{\mathrm{free}})^{2}]$, we have
	\begin{equation}
		\nonumber
		\begin{aligned}
			\EE_{\rho_{1}}[(X_{j}^{\mathrm{free}})^{2}]&=\frac{1}{m^{2}} (\EE_{\rho_{1}}[A_{P}(\theta)^{2}]-\EE_{\rho_{1}}[A_{P}(\theta)]^{2} )\\
			&\preccurlyeq \frac{1}{m^{2}}\EE_{\rho_{1}}[A_{P}(\theta)^{2}]\preccurlyeq  \frac{d_{\max}(\gamma_{P},\lambda)}{m^{2}}\EE_{\rho_{1}}[A_{P}(\theta)].
		\end{aligned}
	\end{equation}
	Similarly, the relation $\EE_{\rho_{0}}[(X_{j}^{\mathrm{fix}})^{2}]\preccurlyeq \frac{d_{\max}(\gamma_{P},\lambda)}{m^{2}}\EE_{\rho_{0}}[A_{P}(\theta)]$ holds for $X_{j}^{\mathrm{fix}}$. Thus,
	\begin{equation}
		\nonumber
		\begin{aligned}
			&\sum_{j=1}^{m_{\mathrm{free}}}\EE_{\rho_{1}}[(X_{j}^{\mathrm{free}})^{2}]+\sum_{j=1}^{m_{\mathrm{fix}}}\EE_{\rho_{0}}[(X_{j}^{\mathrm{fix}})^{2}]\preccurlyeq \frac{d_{\max}(\gamma_{P},\lambda)}{m} \Big(\frac{m_{\mathrm{free}}}{m}\EE_{\rho_{1}}[A_{P}(\theta)]+\frac{m_{\mathrm{fix}}}{m}\EE_{\rho_{0}}[A_{P}(\theta)] \Big)\\
			&= \frac{d_{\max}(\gamma_{P},\lambda)}{m}\EE_{{\rho_{P}}}[A_{P}(\theta)]=\frac{d_{\max}(\gamma_{P},\lambda)}{m}{T}_{\rho}({T}_{\rho}+\lambda I)^{-1}.
		\end{aligned}
	\end{equation}
	For notational convenience, denote $\sum_{j=1}^{m_{\mathrm{free}}}\EE_{\rho_{1}}[(X_{j}^{\mathrm{free}})^{2}]+\sum_{j=1}^{m_{\mathrm{fix}}}\EE_{\rho_{0}}[(X_{j}^{\mathrm{fix}})^{2}]$ as $\sum_{j=1}^{m}\EE[(X_{j})^{2}]$. Therefore, we have 
	\begin{equation}
		\nonumber
		\begin{aligned}
			\Big\Vert{\sum_{j=1}^{m}\EE[(X_{j})^{2}]}\Big\Vert_{\operatorname{op}}&\le \frac{d_{\max}(\gamma_{P},\lambda)}{m},~\text{and}~\\
			\operatorname{tr}\Big(\sum_{j=1}^{m}\EE[(X_{j})^{2}]\Big)&\le \frac{d_{\max}(\gamma_{P},\lambda)}{m} \operatorname{tr} ({T}_{\rho}({T}_{\rho}+\lambda I)^{-1} )=\frac{d(\lambda)d_{\max}(\gamma_{P},\lambda)}{m}.
		\end{aligned}
	\end{equation}
	Finally, we apply the operator Bernstein inequality. By Lemma \ref{lemma:modified-bernstein-inequality}, for $t$ satisfying \eqref{eq:mid-m-ref} below, we have
	\begin{equation}
		\label{eq:the-bound-mid}
		\PP\Bigg(\Big\Vert{\sum_{j=1}^{m_{\mathrm{free}}}X_{j}^{\mathrm{free}}+\sum_{j=1}^{m_{\mathrm{fix}}}X_{j}^{\mathrm{fix}}}\Big\Vert_{\operatorname{op}}>t\Bigg)\le 14d(\lambda)\exp\Big(-\frac{t^{2}/2}{(d_{\max}(\gamma_{P},\lambda)/m)(1+t/3)}\Big).
	\end{equation}
	Lemma \ref{lemma:modified-bernstein-inequality} requires that, with $\sigma^{2}=U=d_{\max}(\gamma_{P},\lambda)/m$,
	\begin{equation}
		\label{eq:mid-m-ref}
		t^{2}\ge \sigma^{2}+\frac{Ut}{3} \Longrightarrow m\ge d_{\max}(\gamma_{P},\lambda) \Big(\frac{1}{t^2}+\frac{1}{3t} \Big).
	\end{equation}
	At the same time, for any $\delta\in(0,1)$, with $d(\lambda)\le d_{\max}(\gamma_{P},\lambda)$, to make the upper bound in \eqref{eq:the-bound-mid} at most $\delta$,  it suffices that
	\begin{equation}
		\label{eq:mid-m-ref-2}
		m\ge 2d_{\max}(\gamma_{P},\lambda)
		 \Big(\frac{1}{t^2}+\frac{1}{3t} \Big)
		\log\frac{14d(\lambda)}{\delta}.
	\end{equation}
	Thus when \eqref{eq:mid-m-ref} and \eqref{eq:mid-m-ref-2} are valid, we have with probability at least $1-\delta$,
	\[ \Big\Vert{\sum_{j=1}^{m_{\mathrm{free}}}X_{j}^{\mathrm{free}}+\sum_{j=1}^{m_{\mathrm{fix}}}X_{j}^{\mathrm{fix}}}\Big\Vert_{\operatorname{op}}\le t,~\text{and}~({T}_{\rho}+\lambda I)^{-1/2}(\widehat{T}_{\rho}-{T}_{\rho})({T}_{\rho}+\lambda I)^{-1/2}\succcurlyeq -tI. \qedhere\]
\end{proof}

\subsection{Proofs of the Main Results in Section \ref{sec:main-results}}
\label{subsec:proof-main-genral-results}
\begin{proof}[\textbf{Proof of Lemma \ref{lemma:kernel-operator-invariance} (\nameref*{lemma:kernel-operator-invariance})}]~
	
\textbf{For item 1}, the changes of measure below use $\gamma_P>0$ for $\rho$-almost every parameter in the nonzero-feature region, with the reweighted feature set to zero elsewhere. By Fubini’s theorem, this zero-feature convention preserves the kernel for $(\mu\otimes\mu)$-almost every input pair. All kernel identities below are understood in this sense. Since $d\rho_P=\gamma_Pd\rho$ in \eqref{eq:distributional-constr.} and $\phi_P=\gamma_P^{-1/2}\phi$ on the nonzero-feature region, we have
	\[ \begin{aligned}
		k_{\rho}(x,x^{\prime})&=\int_{\Theta}\phi(x;\theta)\gamma_{P}^{-1}(\theta)\phi(x^{\prime};\theta)\gamma_{P}(\theta)d\rho(\theta)=\int_{\Theta}\phi_{P}(x;\theta)\phi_{P}(x^{\prime};\theta)d\rho_{P}(\theta).
	\end{aligned} \]
	For $0<P<1$, linearity of expectation and the fixed group sizes yield
	\begin{align*}
		\EE[\hat k_\rho^{(P)}(x,x')]
		&=\frac{m_{\mathrm{free}}}{m}\EE_{\theta\sim\rho_{1}}[\phi_P(x;\theta)\phi_P(x';\theta)]+\frac{m_{\mathrm{fix}}}{m}\EE_{\theta\sim\rho_{0}}[\phi_P(x;\theta)\phi_P(x';\theta)]
		\\
		&=\int_\Theta\phi_P(x;\theta)\phi_P(x';\theta)\,
		d\bigl(P\rho_{1}+(1-P)\rho_{0}\bigr)(\theta)=k_\rho(x,x').
	\end{align*}
	For $P=0$ or $P=1$, the same identity follows directly from the single nonempty group.
	
	\textbf{For item 2}, Assumption \ref{assumption:operator} gives
	\[
	\int_\Theta
	\|\phi_P(\cdot;\theta)\|_{L_2(d\mu)}^2d\rho_P(\theta)=\int_\Theta\|\phi(\cdot;\theta)\|_{L_2(d\mu)}^2d\rho(\theta)<\infty.
	\]
	Thus the integral of the rank-one operators is well defined. By the change of measure,
	\begin{align*}
		\int_\Theta
		\phi_P(\cdot;\theta)\otimes\phi_P(\cdot;\theta)
		\,d\rho_P(\theta)
		&=\int_\Theta\gamma_P(\theta)^{-1}\bigl(\phi(\cdot;\theta)\otimes\phi(\cdot;\theta)\bigr)\gamma_P(\theta)\,d\rho(\theta)=T_\rho.
	\end{align*}
	Thus the original spectral decomposition $T_\rho=\sum_{k\ge1}\eta_k(e_k\otimes e_k)$	and its eigenpairs remain unchanged.
\end{proof}

\begin{proof}[\textbf{Proof of Theorem \ref{thm:general-approx.-bounds} (\nameref*{thm:general-approx.-bounds})}]
	The following proof builds on the general strategy of \cite[Proposition~1]{bach2017equivalence}, but requires a nontrivial modification. The original proof considers i.i.d. features drawn from a single distribution, while our sampling scheme consists of two independent groups drawn separately from $\rho_{1}$ and $\rho_{0}$. We therefore extend the proof to the two-group sampling setting and use the variance-proxy form of Bernstein's inequality.
	
	\mypara{Proof Overview.} The ridge solution reduces control of both the approximation error and the coefficient norm to a common quadratic form involving the regularized empirical operator in \eqref{eq:finial-target-mid-ref}. The variance-proxy Bernstein inequality provides the operator comparison needed to bound this quadratic form by the target's squared RKHS norm, yielding \eqref{eq:mid-estimate-111}. Since the concentration event does not depend on the target function, the resulting bounds hold uniformly over the RKHS unit ball. Finally, choosing $\lambda=\varepsilon/4$ in \eqref{eq:finial-result-thm1} gives the prescribed accuracy.
	
	The proof details are demonstrated below.
	
	\mypara{Ridge Construction.} First, by Definition \ref{def:sampling-scheme}, we recall necessary notations:
	\[ d\rho_{1}(\theta)=\gamma_{1}(\theta)d\rho(\theta),~d\rho_{0}(\theta)=\gamma_{0}(\theta)d\rho(\theta).
	\] 
	Assuming $Pm\in\NN$, we have $m_{\mathrm{free}}=Pm$, $m_{\mathrm{fix}}=(1-P)m$, $\gamma_{P}(\theta)=P{\gamma_{1}}(\theta)+(1-P)\gamma_{0}(\theta)$ and $d\rho_{P}(\theta)=\gamma_{P}(\theta)d\rho(\theta)$. By \eqref{eq:target-function-statement}, there exists $a_{f}^{*}\in L_{2}(d\rho)$ such that
	\[ f=\int_{\Theta} a_{f}^{*}(\theta)\phi(\cdot;\theta)d\rho=\int_{\Theta}\gamma_{P}^{-1/2}a_{f}^{*}(\theta)\phi_{P}(\cdot;\theta)d\rho_{P}(\theta), \]
	with $\norm{a_{f}^{*}}_{L_{2}(d\rho)}=\norm{\gamma_{P}^{-1/2}a_{f}^{*}}_{L_{2}(d\rho_{P})}=\norm{f}_{\mathcal{H}_{\rho}}$. Therefore, it suffices to consider $\norm{f}_{\mathcal{H}_{\rho}}\le 1$ and set $a_f=\gamma_P^{-1/2}a_f^*$ when $\gamma_{P}>0$ and $a_f=0$ when $\gamma_{P}=0$. With the regularization parameter $\lambda>0$, now consider the following minimization problem
	\begin{equation}
		\label{eq:target-error-term-in-proof}
		\underset{\beta}{\min}\Bigg\Vert{f-\sum_{j=1}^{m_{\mathrm{free}}}\beta_{j}\phi_{P}(\cdot;\theta_{j})-\sum_{j=1}^{m_{\mathrm{fix}}}\beta_{m_{\mathrm{free}}+j}\phi_{P}(\cdot;\bar\theta_{j})}\Bigg\Vert_{L_{2}(d\mu)}^{2}+\lambda m\norm{\beta}_{2}^{2},
	\end{equation}
	where $\phi_{P}(\cdot;\theta)=\gamma_{P}(\theta)^{-1/2}\phi(\cdot;\theta)$, ${\theta}_{j}\stackrel{\mathrm{i.i.d.}}{\sim}\rho_{1}, j\in[m_{\mathrm{free}}]$, and $\bar{\theta}_{j}\stackrel{\mathrm{i.i.d.}}{\sim}\rho_{0},j\in[m_{\mathrm{fix}}]$.
	
	We consider the operators $\Phi_{\mathrm{free}}:\RR^{m_{\mathrm{free}}}\to L_{2}(\mathcal{X},d\mu)$ and $\Phi_{\mathrm{fix}}:\RR^{m_{\mathrm{fix}}}\to L_{2}(\mathcal{X},d\mu)$ such that
	\[ \Phi_{\mathrm{free}}\beta_{\mathrm{free}}=\sum_{j=1}^{m_{\mathrm{free}}}\beta_{j}\phi_{P}(\cdot;\theta_{j}),~\Phi_{\mathrm{fix}}\beta_{\mathrm{fix}}=\sum_{j=1}^{m_{\mathrm{fix}}}\beta_{m_{\mathrm{free}}+j}\phi_{P}(\cdot;\bar\theta_{j}). \]
	Further define $\Phi_{P}:\RR^{m_{\mathrm{free}}}\times \RR^{m_{\mathrm{fix}}}\to L_{2}(\mathcal{X},d\mu)$, with $\Phi_{P}=[\Phi_{\mathrm{free}} ~\Phi_{\mathrm{fix}}]$ and  $\Phi_{P}\beta=\Phi_{\mathrm{free}}\beta_{\mathrm{free}}+\Phi_{\mathrm{fix}}\beta_{\mathrm{fix}}$. Therefore, we solve the following regularized least-squares problem
	\[ \underset{\beta}{\min}\norm{f-\Phi_{P}\beta}_{L_{2}(d\mu)}^{2}+\lambda m(\norm{\beta_{\mathrm{free}}}_{2}^{2}+\norm{\beta_{\mathrm{fix}}}_{2}^{2}), \]
	whose solution is 
	\begin{equation}
		\label{eq:the-solution-alpha-beta-proof}
		(\beta_{\mathrm{free}},\beta_{\mathrm{fix}})=(\Phi_{P}^{*}\Phi_{P}+\lambda mI)^{-1}\Phi_{P}^{*}f=\frac{1}{m}\Phi_{P}^{*}(\frac{1}{m}\Phi_{P}\Phi_{P}^{*}+\lambda I)^{-1}f\in\RR^{m}.
	\end{equation}
	
	We thus denote the empirical integral operator $\widehat{T}_{\rho}: L_{2}(\mathcal{X},d\mu)\to L_{2}(\mathcal{X},d\mu)$ as
	\[ \widehat{T}_{\rho}=\frac{1}{m}\Phi_{P}\Phi_{P}^{*}=\frac{1}{m} \Big[\sum_{j=1}^{m_{\mathrm{free}}}\phi_{P}(\cdot;\theta_{j})\otimes\phi_{P}(\cdot;\theta_{j})+\sum_{j=1}^{m_{\mathrm{fix}}}\phi_{P}(\cdot;\bar\theta_{j})\otimes\phi_{P}(\cdot;\bar\theta_{j}) \Big]. \]
	By construction, for $\widehat{T}_{\rho}$, we have $\EE [\widehat{T}_{\rho} ]=T_{\rho}$ from
	\[ \begin{aligned}
		\EE [\widehat{T}_{\rho} ]&=\frac{m_{\mathrm{free}}}{m}\int_{\Theta}\frac{\gamma_{1}(\theta)}{\gamma_{P}(\theta)}\phi(\cdot;\theta)\otimes\phi(\cdot;\theta)d\rho(\theta)+\frac{m_{\mathrm{fix}}}{m}\int_{\Theta}\frac{\gamma_{0}(\theta)}{\gamma_{P}(\theta)}\phi(\cdot;\theta)\otimes\phi(\cdot;\theta)d\rho(\theta)\\
		&=\int_{\Theta}\frac{P\gamma_{1}(\theta)+(1-P)\gamma_{0}(\theta)}{\gamma_{P}(\theta)}\phi(\cdot;\theta)\otimes\phi(\cdot;\theta)d\rho(\theta)\\
		&=\int_{\Theta}\phi(\cdot;\theta)\otimes\phi(\cdot;\theta)d\rho(\theta)=T_{\rho},
	\end{aligned} \]
	where we utilize the $m_{\mathrm{free}}/m=P$, $m_{\mathrm{fix}}/m=1-P$ and definition of $\gamma_{P}(\theta)$.
	
	\mypara{Error and Coefficient Bounds.} Now we consider $\norm{f-\Phi_{P}\beta}_{L_{2}(d\mu)}^{2}$, 
	\begin{equation}
		\label{eq:error-part-1-general-rates}
		\begin{aligned}
			\norm{f-\Phi_{P}\beta}_{L_{2}(d\mu)}^{2}&=\norm{f-\widehat{T}_{\rho}(\widehat{T}_{\rho}+\lambda I)^{-1}f}_{L_{2}(d\mu)}^{2}=\norm{\lambda(\widehat{T}_{\rho}+\lambda I)^{-1}f}_{L_{2}(d\mu)}^{2}\\
			&=\lambda^{2} \langle f, (\widehat{T}_{\rho}+\lambda I)^{-2}f \rangle_{L_{2}(d\mu)}\le\lambda \langle f,(\widehat{T}_{\rho}+\lambda I)^{-1}f  \rangle_{L_{2}(d\mu)},
		\end{aligned}
	\end{equation}
	where the classical partial order between self-adjoint operators indicates $(\widehat{T}_{\rho}+\lambda I)^{-2}\preccurlyeq \lambda^{-1}(\widehat{T}_{\rho}+\lambda I)^{-1}$. 
	
	Then consider $\lambda m(\norm{\beta_{\mathrm{free}}}_{2}^{2}+\norm{\beta_{\mathrm{fix}}}_{2}^{2})$, with \eqref{eq:the-solution-alpha-beta-proof}, we have 
	\begin{equation}
		\label{eq:error-part-2-general-rates}
		\begin{aligned}
			&\lambda m(\norm{\beta_{\mathrm{free}}}_{2}^{2}+\norm{\beta_{\mathrm{fix}}}_{2}^{2})\\
			&= \lambda m \langle  \frac{1}{m}\Phi_{P}^{*}(\frac{1}{m}\Phi_{P}\Phi_{P}^{*}+\lambda I)^{-1}f, \frac{1}{m}\Phi_{P}^{*}(\frac{1}{m}\Phi_{P}\Phi_{P}^{*}+\lambda I)^{-1}f \rangle_{\RR^{m}}\\
			&=\lambda \langle  (\widehat{T}_{\rho}+\lambda I)^{-1}f, \widehat{T}_{\rho}(\widehat{T}_{\rho}+\lambda I)^{-1}f  \rangle_{L_{2}(d\mu)}\le \lambda \langle  f, (\widehat{T}_{\rho}+\lambda I)^{-1}f  \rangle_{L_{2}(d\mu)},
		\end{aligned}
	\end{equation}
	where we use $(\widehat{T}_{\rho}+\lambda I)^{-2}\widehat{T}_{\rho}\preccurlyeq (\widehat{T}_{\rho}+\lambda I)^{-1}$.
	
	With \eqref{eq:target-error-term-in-proof}, \eqref{eq:error-part-1-general-rates} and \eqref{eq:error-part-2-general-rates}, both quantities are bounded by
	\begin{equation}
		\label{eq:finial-target-mid-ref}
		\begin{aligned}
			\max \{\norm{f-\Phi_{P}\beta}_{L_{2}(d\mu)}^{2}, \lambda m(\norm{\beta_{\mathrm{free}}}_{2}^{2}+\norm{\beta_{\mathrm{fix}}}_{2}^{2}) \}&\le \lambda \langle  f, (\widehat{T}_{\rho}+\lambda I)^{-1}f  \rangle_{L_{2}(d\mu)}
		\end{aligned}
	\end{equation}
	\mypara{Uniform Control over the RKHS Unit Ball.} To bound $\langle  f, (\widehat{T}_{\rho}+\lambda I)^{-1}f \rangle_{L_{2}(d\mu)}$, use the operator factorization in \cite{alaoui2015fast},
	\[ 
	\begin{aligned}
		& \langle  f, (\widehat{T}_{\rho}+\lambda I)^{-1}f  \rangle_{L_{2}(d\mu)}= \langle  f, ({T}_{\rho}+\lambda I+\widehat{T}_{\rho}-{T}_{\rho})^{-1}f  \rangle_{L_{2}(d\mu)}\\
		&= \langle  f, ({T}_{\rho}+\lambda I+{({T}_{\rho}+\lambda I)^{1/2}({T}_{\rho}+\lambda I)^{-1/2}}(\widehat{T}_{\rho}-{T}_{\rho}){({T}_{\rho}+\lambda I)^{-1/2}({T}_{\rho}+\lambda I)^{1/2}})^{-1}f  \rangle_{L_{2}(d\mu)}\\
		&= \langle ({T}_{\rho}+\lambda I)^{-1/2}f, (I+{({T}_{\rho}+\lambda I)^{-1/2}(\widehat{T}_{\rho}-{T}_{\rho})({T}_{\rho}+\lambda I)^{-1/2}})^{-1}({T}_{\rho}+\lambda I)^{-1/2}f  \rangle_{L_{2}(d\mu)}.
	\end{aligned}
	\]
	Based on Proposition \ref{prop:partial-order-deviation-empirical-continu}, for any $\delta\in(0,1)$, we can prove that the deviation part satisfies, for any $t\in(0,1)$, with probability at least $1-\delta$,
	\[ {({T}_{\rho}+\lambda I)^{-1/2}(\widehat{T}_{\rho}-{T}_{\rho})({T}_{\rho}+\lambda I)^{-1/2}}\succcurlyeq -tI. \]
	Then we consider $t=\frac{3}{4}$. Note that the choice $t=3/4$ is made only for notational convenience. Any fixed $t\in(0,1)$ leads to the same
	approximation rate, only with different universal constants. We require that
	\begin{equation}
		\label{eq:m-required-main-thm-proof}
		m\ge\max \Bigg\{\frac{20}{9}d_{\max}(\gamma_{P},\lambda), \frac{40}{9}d_{\max}(\gamma_{P},\lambda)\log\frac{14d_{\max}(\gamma_{P},\lambda)}{\delta} \Bigg\}
	\end{equation}
	Since $m\in\NN_{+}$, we can turn \eqref{eq:m-required-main-thm-proof} into
	\begin{equation}
		\label{eq:m-required-main-thm-proof-final}
		m\ge\frac{40}{9}d_{\max}(\gamma_{P},\lambda)\log\frac{14d_{\max}(\gamma_{P},\lambda)}{\delta},
	\end{equation}
	which is because when $d_{\max}\le \frac{9}{20}$, $m\ge \frac{20}{9}d_{\max}$ naturally holds while when $d_{\max}> \frac{9}{20}$, $\log\frac{14d_{\max}}{\delta}>\log(14\cdot\frac{9}{20})>\frac{1}{2}$. Thus when \eqref{eq:m-required-main-thm-proof-final} is valid, with probability at least $1-\delta$,
	\[ ({T}_{\rho}+\lambda I)^{-1/2}(\widehat{T}_{\rho}-{T}_{\rho})({T}_{\rho}+\lambda I)^{-1/2}\succcurlyeq -\frac{3}{4}I. \]
	Therefore,
	\begin{equation}
		\label{eq:mid-estimate-111}
		\begin{aligned}
			\langle f, (\widehat{T}_{\rho}+\lambda I)^{-1}f  \rangle_{L_{2}(d\mu)}&\le  \langle ({T}_{\rho}+\lambda I)^{-1/2}f, (1-t)^{-1}({T}_{\rho}+\lambda I)^{-1/2}f  \rangle_{L_{2}(d\mu)}\\
			&=(1-t)^{-1} \langle f, ({T}_{\rho}+\lambda I)^{-1}f  \rangle_{L_{2}(d\mu)}\\
			&= 4\langle f, ({T}_{\rho}+\lambda I)^{-1}f \rangle_{L_{2}(d\mu)}.
		\end{aligned}
	\end{equation}
	We next bound $\langle f, ({T}_{\rho}+\lambda I)^{-1}f \rangle_{L_{2}(d\mu)}$. By Cauchy-Schwarz inequality, we have for any $b\in L_{2}(\mathcal{X}, d\mu)$,
	\[ \begin{aligned}
		&\langle b,(f\otimes_{L_2(d\mu)}f)b \rangle_{L_{2}(d\mu)}= \Big(\int_{\mathcal{X}}b(x)f(x)d\mu(x) \Big)^{2}\\
        &= \Big(\int_{\mathcal{X}}\int_{\Theta}b(x)a_{f}(\theta)\phi_{P}(x;\theta)d\mu(x)d\rho_{P}(\theta) \Big)^{2}\\
		&\le \int_{\Theta}a_{f}(\theta)^{2}d\rho_{P}(\theta)\cdot\int_{\Theta} \Big(\int_{\mathcal{X}}b(x)\phi_{P}(x;\theta)d\mu(x) \Big)^{2}d\rho_{P}(\theta)\\
		&\le \norm{a_{f}}_{L_{2}(d\rho_{P})}^{2}\cdot \int_{\Theta} \langle b,(\phi_{P}\otimes_{L_2(d\mu)}\phi_{P})b \rangle_{L_{2}(d\mu)}d\rho_{P}(\theta)\le  \langle b,T_{\rho}b \rangle_{L_{2}(d\mu)},
	\end{aligned} \]
	where $T_{\rho}=\int_{\Theta}\phi_{P}(\cdot,\theta)\otimes_{L_2(d\mu)}\phi_{P}(\cdot,\theta)d\rho_{P}(\theta)$. This means $f\otimes_{L_2(d\mu)} f \preccurlyeq T_{\rho}$. Using this inequality with  $(T_{\rho}+\lambda I)^{-1/2}$ gives
	\[ [(T_{\rho}+\lambda I)^{-1/2}f]\otimes [(T_{\rho}+\lambda I)^{-1/2}f]\preccurlyeq [(T_{\rho}+\lambda I)^{-1/2}]T_{\rho}[(T_{\rho}+\lambda I)^{-1/2}]\preccurlyeq I. \]
	This yields $\langle f, ({T}_{\rho}+\lambda I)^{-1}f \rangle_{L_{2}(d\mu)}\le 1$. Consequently, $\langle  f, (\widehat{T}_{\rho}+\lambda I)^{-1}f\rangle_{L_{2}(d\mu)} \le 4$. Further with \eqref{eq:finial-target-mid-ref} and \eqref{eq:mid-estimate-111},  we have with probability at least $1-\delta$
	\begin{equation}
		\label{eq:finial-result-thm1}
		\begin{aligned}
			\norm{f-\Phi_{P}\beta}_{L_{2}(d\mu)}^{2}&\le \lambda \langle  f, (\widehat{T}_{\rho}+\lambda I)^{-1}f  \rangle_{L_{2}(d\mu)}\le \lambda\cdot\frac{1}{1-t}= 4\lambda.
		\end{aligned}
	\end{equation}
	The coefficients also satisfy
	\[ \lambda m(\norm{\beta_{\mathrm{free}}}_{2}^{2}+\norm{\beta_{\mathrm{fix}}}_{2}^{2})\le 4\lambda\Longrightarrow \norm{\beta_{\mathrm{free}}}_{2}^{2}+\norm{\beta_{\mathrm{fix}}}_{2}^{2}=\norm{\beta}_{2}^{2}\le\frac{4}{m}, \]
	when \eqref{eq:m-required-main-thm-proof-final} is satisfied. Let $\lambda=\varepsilon/4$ for any $\varepsilon>0$, and we prove the desired result.
\end{proof}

\begin{proof}[\textbf{Proof of Lemma \ref{lemma:maximal-leverage-bound} (\nameref*{lemma:maximal-leverage-bound})}]
	By the definition of $d_{\max}$, the following inequalities hold $\rho$-almost everywhere.
	\[
	 \langle\phi(\cdot;\theta),(T_\rho+\lambda I)^{-1}\phi(\cdot;\theta) \rangle_{L_2(d\mu)}\le d_{\max}(\gamma_i,\lambda)\gamma_i(\theta),~\text{for}~ i\in\{0,1\}.
	\]
	With $\gamma_{P}=P\gamma_{1}+(1-P)\gamma_{0}$, it follows that
	\[
	\gamma_P(\theta)
	\ge
	 \Big(
	\frac{P}{d_{\max}(\gamma_1,\lambda)}
	+
	\frac{1-P}{d_{\max}(\gamma_0,\lambda)}
	 \Big)
	 \langle
	\phi(\cdot;\theta),(T_\rho+\lambda I)^{-1}\phi(\cdot;\theta)
	 \rangle_{L_2(d\mu)}.
	\]
	Dividing by $\gamma_P(\theta)$ and taking the essential supremum yields
	\[ d_{\max}(\gamma_P,\lambda)\le
	 \Big(
	\frac{P}{d_{\max}(\gamma_1,\lambda)}
	+
	\frac{1-P}{d_{\max}(\gamma_0,\lambda)}
	 \Big)^{-1}= \Big(\frac{P}{\kappa_{\gamma_{1}}(\lambda)}+\frac{1-P}{\kappa_{\gamma_{0}}(\lambda)} \Big)^{-1}d(\lambda), \]
	where we use $\kappa_{\gamma_{i}}(\lambda)={d_{\max}(\gamma_{i},\lambda)}/{d(\lambda)}$.
\end{proof}

\begin{proof}[\textbf{Proof of Theorem \ref{thm:polynomial-rates} (\nameref*{thm:polynomial-rates})}] 
Replacing $\lambda_0$ by $\min\{\lambda_0,\eta_1\}$ preserves Assumption \ref{assumption:mismatch-used}. Thus, assume without loss of generality that $\lambda_0\le\eta_1$. 

For $0<\lambda\le\lambda_0$, since the $\{\eta_j\}$ are in the descending order, let
$J_\lambda=\max\{j\ge1:\eta_j\ge\lambda\}$.	Under Assumption \ref{assumption:eigendecay-used}, we have $\eta_{J_{\lambda}}\asymp \lambda\asymp J_{\lambda}^{-\alpha}$, and hence
	\begin{equation}
		\label{eq:bound-of-d(lamda)-lambda}
		\begin{aligned}
			d(\lambda)&=\sum_{j\ge1}\frac{\eta_j}{\eta_j+\lambda}=\sum_{1\le j\le J_{\lambda}}\frac{\eta_j}{\eta_j+\lambda}+\sum_{j>J_{\lambda}}\frac{\eta_j}{\eta_j+\lambda}\lesssim 1\cdot J_{\lambda}+\frac{1}{\lambda}\sum_{j> J_{\lambda}}\eta_j\\
			&\lesssim J_\lambda+ \frac{1}{\lambda}\sum_{j> J_{\lambda}}j^{-\alpha}\lesssim J_\lambda+\frac{1}{\lambda}J_{\lambda}^{1-\alpha}\asymp\lambda^{-1/\alpha}.
		\end{aligned} 
	\end{equation}
	The matching lower bound follows from $d(\lambda)\ge J_{\lambda}/2 \asymp\lambda^{-1/\alpha}$. Thus $d(\lambda)\asymp\lambda^{-1/\alpha}$. For each such $\lambda$, Assumption \ref{assumption:mismatch-used}
	guarantees a choice of $\gamma_1$ for which
	Lemma \ref{lemma:maximal-leverage-bound} gives
	\[ d_{\max}(\gamma_{P},\lambda)\le \Big(\frac{P}{\kappa_{\gamma_{1}}(\lambda)}+\frac{1-P}{\kappa_{\gamma_{0}}(\lambda)} \Big)^{-1}d(\lambda)\le \Big(\frac{P}{\kappa_{1}^{*}}+\frac{1-P}{\kappa_{0}^{*}} \Big)^{-1}d(\lambda).   \] 
	Combining this with
	\eqref{eq:bound-of-d(lamda)-lambda} yields
	\[ d_{\max}(\gamma_P,\lambda)\lesssim \Big(\frac{P}{\kappa_1^*}+\frac{1-P}{\kappa_0^*} \Big)^{-1}\lambda^{-1/\alpha}, \]
	where the implicit constant is independent of $P$. For the given network width $m$, a sufficient condition for \eqref{eq:sample-condition} is
	\begin{equation}
        \nonumber
		d_{\max}(\gamma_{P},\lambda)\le\frac{9m}{40\log(63m/(20\delta))}.
	\end{equation}
	Indeed, $\log(63m/(20\delta))>1$, so this condition implies $d_{\max}(\gamma_{P},\lambda)\le 9m/40$ and hence $\log(14d_{\max}(\gamma_{P},\lambda)/\delta)\le\log(63m/(20\delta))$. The width condition \eqref{eq:sample-condition} holds. It therefore suffices to choose $\lambda$ satisfying
	\begin{equation}
		\label{eq:relation-lambda-lower-bound-choice}
		\lambda\gtrsim
		 \Big(
		 \Big(\frac{P}{\kappa_1^*}
		+\frac{1-P}{\kappa_0^*} \Big)
		\frac{9m}{40\log(63m/(20\delta))}
		 \Big)^{-\alpha},
	\end{equation}
	where the implicit constant is chosen sufficiently large. Choose $\lambda$ to attain this lower bound,
	with the implicit constant fixed. If $\lambda\le \lambda_{0}$, then \eqref{eq:main-result-GAB} with $\varepsilon=4\lambda$ gives,
	with probability at least $1-\delta$,
	\[
	\sup_{\norm{f}_{\mathcal H_\rho}\le1}\inf_{f_2^{(P)}\in\mathcal F_P(W,2)}\norm{f-f_2^{(P)}}_{L_2(d\mu)}^2
	\le4\lambda.
	\]
	
	If $\lambda>\lambda_0$, the zero approximant instead gives
	\[
	\inf_{f_2^{(P)}\in\mathcal F_P(W,2)}
	\norm{f-f_2^{(P)}}_{L_2(d\mu)}^2
	\le\norm{T_\rho}_{\mathrm{op}}\le\frac{\norm{T_\rho}_{\mathrm{op}}}{\lambda_0}\lambda,
	\]
	where we use 
	\[
	\norm{f}_{L_2(d\mu)}^2=\norm{T_{\rho}^{1/2}q_{f}}_{L_2(d\mu)}^2=\langle q_{f},T_{\rho}q_{f}\rangle_{L_2(d\mu)}\le \norm{T_\rho}_{\mathrm{op}}\norm{q_{f}}_{L_2(d\mu)}^2\le\norm{T_\rho}_{\mathrm{op}}.
	\]
	Combining the two cases with
	\eqref{eq:relation-lambda-lower-bound-choice} yields
	\[
	\sup_{\norm{f}_{\mathcal H_\rho}\le1}\inf_{f_2^{(P)}\in\mathcal F_P(W,2)}
	\norm{f-f_2^{(P)}}_{L_2(d\mu)}^2\lesssim \Big( \Big(\frac{P}{\kappa_1^*}+\frac{1-P}{\kappa_0^*} \Big)\frac{9m}{40\log(63m/(20\delta))} \Big)^{-\alpha}.
	\]
	The implicit constant is independent of $m$, $P$ and $\delta$.
\end{proof}

\section{Proofs of Main Results in Section \ref{sec:specific-activation-functions} -- Specific Spherical Setting}
\label{subsec:proo-of-specific}
\subsection*{Outline}
\begin{itemize}[labelindent=0pt, leftmargin=*, topsep=0pt, itemsep=0pt]
	\item Subsection \ref{subsec:auxilary-spherical} establishes the Funk–Hecke formula and harmonic representations of the ridge leverage score and ReLU harmonic coefficients. The Jacobi representation then yields sharp estimates for the degree-wise eigenvalues.
	\item  Subsection \ref{subsec:proofs-spherical} uses the above estimates to prove the kernel spectrum and eigenvalue decay. It then bounds the initialization mismatch, first for the uniform spherical distribution and then for the scaled spherical family. Combining the spectral decay and mismatch bounds gives final explicit spherical approximation rates.
\end{itemize}
\subsection{Auxiliary Theoretical Preliminaries}
\label{subsec:auxilary-spherical}
\begin{lemma}[The Funk--Hecke Formula]
	\label{lemma:normalized-funk-hecke}
	Let $\mu=\operatorname{Unif}(\mathbb S^{d-1})$ be the normalized
	surface measure on $\mathbb S^{d-1}$, where $d\geq2$. Let
	$\psi:[-1,1]\to\RR$ be measurable and satisfy
	\[
	\int_{-1}^{1}
	\abs{\psi(t)}(1-t^2)^{(d-3)/2}\,dt<+\infty.
	\]
	For every spherical harmonic $Y_{k,j}\in\YY_k$ and every
	$x\in\mathbb S^{d-1}$,
	\begin{equation}
		\nonumber
		\label{eq:normalized-funk-hecke}
		\int_{\mathbb S^{d-1}}
		\psi(x^\top y)Y_{k,j}(y)\,d\mu(y)
		=
		\widehat{\psi}_kY_{k,j}(x),
	\end{equation}
	where
	\begin{equation}
		\nonumber
		\label{eq:normalized-funk-hecke-coefficient}
		\widehat{\psi}_k
		=
		\frac{\Gamma(d/2)}
		{\sqrt{\pi}\Gamma((d-1)/2)}
		\int_{-1}^{1}
		\psi(t)P_k(t)(1-t^2)^{(d-3)/2}\,dt,
	\end{equation}
	and $P_k$ is the degree-$k$ Legendre polynomial in dimension
	$d$, normalized by $P_k(1)=1$.
\end{lemma}
\begin{proof}[\textbf{Proof of Lemma \ref{lemma:normalized-funk-hecke}}]
	This is the Funk--Hecke formula
	\cite[Theorem~2.22]{atkinson2012spherical}, rewritten for the normalized surface measure
	$d\mu=dS_{d-1}/\abs{\mathbb S^{d-1}}$, using $\frac{\abs{\mathbb S^{d-2}}}{\abs{\mathbb S^{d-1}}}=\frac{\Gamma(d/2)}{\sqrt{\pi}\Gamma((d-1)/2)}$.
\end{proof}

\begin{lemma}[Harmonic Representation of the Ridge Leverage Score]
	\label{lemma:harmonic-representation-leverage-score}
	Let $\theta=(\omega,b)=(\sqrt{1-b^2}u,b)\in\mathbb S^d$, where $u\sim\operatorname{Unif}(\mathbb S^{d-1})$ is independent
	of $b$, and $b$ has density
	\begin{equation}
		\label{eq:definition-pb-density}
		p(b)= \frac{\Gamma((d+1)/2)}{\sqrt{\pi}\Gamma(d/2)}(1-b^2)^{(d-2)/2},~\text{with}~b\in(-1,1).
	\end{equation}
	Define
	\[
	\chi_k(b)=\frac{\Gamma(d/2)}{\sqrt{\pi}\Gamma((d-1)/2)}\int_{-1}^{1}\bigl(\sqrt{1-b^2}\,t+b\bigr)_+P_k(t)(1-t^2)^{(d-3)/2}dt.
	\]
	For $\lambda>0$, write $s_{\lambda}(\theta)=\langle \phi(\cdot;\theta),(T_{\rho}+\lambda I)^{-1}\phi(\cdot;\theta)\rangle_{L_2(d\mu)}$ for  the unweighted score. Then, for every $Y_{k,j}\in\YY_k$, we have $\langle\phi(\cdot;\theta),Y_{k,j}\rangle_{L_2(d\mu)}=\chi_k(b)Y_{k,j}(u)$. Moreover,
	\begin{equation}
		\label{eq:harmonic-repre-of-essential-terms}
		\xi_k=\int_{-1}^{1}\chi_k(b)^2p(b)db,~s_\lambda(\theta)=\sum_{k=0}^{\infty}\frac{N(d,k)\chi_k(b)^2}{\xi_k+\lambda},~\text{and}~d(\lambda)=\sum_{k=0}^{\infty}\frac{N(d,k)\xi_k}{\xi_k+\lambda}.
	\end{equation}
\end{lemma}
\begin{proof}[\textbf{Proof of Lemma \ref{lemma:harmonic-representation-leverage-score}}]
	Rewrite every $\theta=(\omega,b)\in\mathbb S^{d}$, except on a zero measure set (when $b=\pm 1$), as 
	\[ \theta=(\omega,b)=(\norm{\omega}_{2}\frac{\omega}{\norm{\omega}_{2}},b)\coloneq(\sqrt{1-b^{2}}u,b),  \]
	where $u=\omega/\norm{\omega}_{2}\in\mathbb S^{d-1}$ is independent of $b$ and $u\sim\operatorname{Unif}(\mathbb S^{d-1})$. By the spherical integral, $b$ has density
	\begin{equation}
		\nonumber
		p(b)= \frac{\Gamma((d+1)/2)}{\sqrt{\pi}\Gamma(d/2)}(1-b^2)^{(d-2)/2},~\text{with}~b\in(-1,1).
	\end{equation} 
	And we can write $\phi(x;\theta)=\sigma_{R}(\omega^{\top}x+b)=\sigma_{R}(\sqrt{1-b^{2}}u^{\top}x+b)$. Applying the Funk-Hecke formula of Lemma \ref{lemma:normalized-funk-hecke} to the feature map $\phi$, we have for any $Y_{k,j}\in \YY_{k}$, 
	\begin{equation}
		\label{eq:F-Hformula-phi-ref1}
		\langle\phi(\cdot;\theta),Y_{k,j}(\cdot)\rangle_{L_2(d\mu)}=\int_{\mathcal{X}}\sigma_{R}(\sqrt{1-b^{2}}u^{\top}x+b)Y_{k,j}(x)d\mu(x)=\chi_{k}(b)Y_{k,j}(u),
	\end{equation}
	with 
	\[ \chi_{k}(b)= \frac{\Gamma(d/2)}{\sqrt{\pi}\Gamma((d-1)/2)}\int_{-1}^{1}\sigma_{R}(\sqrt{1-b^2}\,t+b)P_k(t)(1-t^2)^{(d-3)/2}dt. \]
	
	The independence of $u$ and $b$ and orthonormality give
	\[ 
	\begin{aligned}
		\langle Y_{\ell,i}, T_{\rho}Y_{k,j}\rangle&=\int_\Theta
		\langle Y_{\ell,i},\phi(\cdot;\theta)\rangle
		\langle\phi(\cdot;\theta),Y_{k,j}\rangle d\rho(\theta)\\
		&=\int_{-1}^{1}\chi_{\ell}(b)\chi_{k}(b)p(b)db\int_{\mathbb S^{d-1}}Y_{\ell,i}(u)Y_{k,j}(u)d\mu(u).
	\end{aligned}
	 \]
	The second integral is zero unless $(\ell,i)=(k,j)$, when it equals one.
	Completeness therefore gives $T_\rho Y_{k,j}=\xi_kY_{k,j}$ with
	$\xi_k=\int_{-1}^{1}\chi_k(b)^2p(b)\,db$.
	
	Further, by the addition formula $\sum_{j=1}^{N(d,k)}Y_{k,j}(u)^{2}=N(d,k)$,  the corresponding spectral decomposition of $(T_{\rho}+\lambda I)^{-1}$ similar to item 2 of Proposition \ref{prop:spectrum-biased-relu-kernel}, and \eqref{eq:F-Hformula-phi-ref1}, we obtain
	\begin{equation}
		\label{eq:score-proof-the-aim-to-bound}
		\begin{aligned}
			s_{\lambda}(\theta)
			&=\langle\phi(\cdot;\theta), (T_{\rho}+\lambda I)^{-1}\phi(\cdot;\theta)\rangle_{L_2(d\mu)}\\
			&=\langle\phi(\cdot;\theta),\sum_{k=0}^{\infty}\sum_{j=1}^{N(d,k)}\frac{1}{\xi_{k}+\lambda}Y_{k,j}\otimes Y_{k,j} \phi(\cdot;\theta) \rangle_{L_2(d\mu)}\\
			&=\sum_{k=0}^{\infty}\sum_{j=1}^{N(d,k)}\frac{\abs{\langle\phi(\cdot;\theta),Y_{k,j}(\cdot)\rangle_{L_2(d\mu)}}^{2}}{\xi_{k}+\lambda}\\
			&=\sum_{k=0}^{\infty}\sum_{j=1}^{N(d,k)}\frac{(\chi_{k}(b)Y_{k,j}(u))^{2}}{\xi_{k}+\lambda}=\sum_{k=0}^{\infty}\frac{N(d,k)\chi_{k}(b)^{2}}{\xi_k+\lambda}.
		\end{aligned}
	\end{equation}
	With $d(\lambda)=\EE_{\theta\sim\rho}[s_{\lambda}(\theta)]$, we also obtain $d(\lambda)=\sum_{k=0}^{\infty}\frac{N(d,k)\xi_{k}}{\xi_{k}+\lambda}$.
\end{proof}

\begin{lemma}[Jacobi Representation of the ReLU Harmonic Coefficients]
	\label{lemma:relu-harmonic-coefficient-asymptotics}
	Let $c_d=\frac{d-3}{2}$, $G_d(b)=(1-2b^2)_+^{d/2}(1-b^2)^{(2-d)/2}$ for $\abs{b}<1$, set $G_{d}(\pm 1)=0$ and let $C_{1,d}=\frac{2^{(d-2)/2}\Gamma(d/2)}{\pi}$. Then the following statements hold for $k\geq2$.
	\begin{enumerate}[labelindent=0pt, leftmargin=*,topsep=0pt,partopsep=0pt]
		\item If $\abs{b}\geq1/\sqrt{2}$, then $\chi_k(b)=0$. While if $\abs{b}<1/\sqrt{2}$,
		\begin{equation}
			\nonumber\chi_k(b)=\frac{\Gamma(d/2)\sqrt{1-b^2}}{4\sqrt{\pi}\Gamma((d-1)/2)k(k-1)J_k^{(c_d,c_d)}(1)} \Big(\frac{1-2b^2}{1-b^2} \Big)^{c_d+2}J_{k-2}^{(c_d+2,c_d+2)} \Big(-\frac{b}{\sqrt{1-b^2}} \Big).
		\end{equation}
		\item There exists a constant $C_d>0$ such that
		\begin{equation}
            \nonumber
			\chi_k(b)^2
			\leq
			C_dG_d(b)k^{-(d+2)},
			\qquad b\in[-1,1].
		\end{equation}
		\item For every compact interval
		$I\subset(-1/\sqrt{2},1/\sqrt{2})$,
		\begin{equation}
			\label{eq:compact-asymptotic-chi}
			\chi_k(b)=C_{1,d}\sqrt{G_d(b)}\,k^{-(d+2)/2} [\cos\Psi_k(b)+\mathcal O(k^{-1}) ]
		\end{equation}
		uniformly for $b\in I$, where
		\[
		\Psi_k(b)= \Big(k+\frac{d-2}{2} \Big)\arccos \Big(-\frac{b}{\sqrt{1-b^2}} \Big)-\frac{\pi(d+2)}{4}.
		\]
	\end{enumerate}
\end{lemma}
\begin{proof}[\textbf{Proof of Lemma \ref{lemma:relu-harmonic-coefficient-asymptotics}}] We prove the statements item by item.
	\mypara{For item 1:} Let $J_k^{(\alpha,\beta)}$ denote the standard Jacobi polynomial. Since the Legendre polynomial $P_k$ is normalized by $P_k(1)=1$ and, by \cite[Section 2.6.1]{atkinson2012spherical}, $P_{k}$ is proportional to the Jacobi polynomial $J_{k}^{((d-3)/2,(d-3)/2)}(t)$, $P_{k}$ can be written as 
	\[ P_{k}(t)=\frac{J_{k}^{((d-3)/2,(d-3)/2)}(t)}{J_{k}^{((d-3)/2,(d-3)/2)}(1)}. \]
	
	If $b\le -\frac{1}{\sqrt{2}}$, the feature $\sigma_{R}(\sqrt{1-b^{2}}u^\top x+b)$ vanishes. And if $b\ge \frac{1}{\sqrt{2}}$,
	\[ \sigma_{R}(\sqrt{1-b^{2}}u^\top x+b)= b+\sqrt{1-b^{2}}\,u^\top x
	\in\YY_0\oplus\YY_1. \]
	By the orthogonality of spherical harmonics of different degrees and \eqref{eq:F-Hformula-phi-ref1}, $\chi_k(b)=0$ for every $k\ge2$. Thus when $k\ge 2$ and $\abs{b}\ge\frac{1}{\sqrt{2}}$, the $\chi_{k}(b)=0$.
	
	For $\abs{b}<\frac{1}{\sqrt{2}}$, by the Rodrigues formula for Jacobi Polynomials, let {$c_{d}=(d-3)/2$},
	\[ (1-t^{2})^{c_{d}}J_{k}^{(c_{d},c_{d})}(t)=\frac{(-1)^{k}}{2^{k}k!}\frac{d^{k}}{dt^{k}}[(1-t^{2})^{c_{d}+k}] \]
	after two integrations by parts, we have for $k\ge 2$, 
	\begin{equation}
		\label{eq:bound-of-chi-former}
		\begin{aligned}
			&\chi_{k}(b)\\
			&=C_\Gamma\cdot\frac{\sqrt{1-b^{2}}}{J_{k}^{(c_{d},c_{d})}(1)}\cdot\frac{(-1)^{k}}{2^{k}k!}\int_{-\frac{b}{\sqrt{1-b^{2}}}}^{1}\Big(t+\frac{b}{\sqrt{1-b^{2}}}\Big)\frac{d^{k}}{dt^{k}}[(1-t^{2})^{c_{d}+k}]dt\\
			&=\frac{C_\Gamma\sqrt{1-b^{2}}}{4k(k-1)J_{k}^{(c_{d},c_{d})}(1)}\cdot\frac{(-1)^{k-2}}{2^{k-2}(k-2)!}\cdot\frac{d^{k-2}}{dt^{k-2}}[(1-t^{2})^{c_{d}+k}] \Bigg|_{t=-\frac{b}{\sqrt{1-b^{2}}}}\\
			&=\frac{C_\Gamma\sqrt{1-b^{2}}}{4k(k-1)J_{k}^{(c_{d},c_{d})}(1)} \Big(\frac{1-2b^{2}}{1-b^{2}} \Big)^{c_{d}+2}J_{k-2}^{(c_{d}+2,c_{d}+2)}\Big(-\frac{b}{\sqrt{1-b^{2}}}\Big),
		\end{aligned}
	\end{equation}
    where $C_\Gamma=\frac{\Gamma(d/2)}{\sqrt{\pi}\Gamma((d-1)/2)}$.
	\mypara{For item 2:} Based on \cite[Eq.(2.3.41)]{mastroianni2008interpolation}, we have for $t\in (-1,1)$,
	\[ \Bigg\vert{\frac{J_{k}^{(c_{d},c_{d})}(t)}{\norm{J_{k}^{(c_{d},c_{d})}}}}\Bigg\vert\le C_{d} \Big(\sqrt{1-t}+\frac{1}{k} \Big)^{-c_{d}-1/2} \Big(\sqrt{1+t}+\frac{1}{k} \Big)^{-c_{d}-1/2}\le C_{d}(1-t^{2})^{-c_{d}/2-1/4}. \]
	Since $\norm{J_{k}^{(c_{d},c_{d})}}=\mathcal{O}(k^{-1/2})$ with the norm taken in $L^2([-1,1],(1-t^2)^{c_d}dt)$, we have 
	\[ \abs{J_{k}^{(c_{d},c_{d})}(t)}\le C_{d}k^{-1/2}(1-t^{2})^{-c_{d}/2-1/4}~\text{for}~ t\in(-1,1). \]
	Therefore, we use 
	\begin{equation}
		\label{eq:bound-of-jacobi-abs}
		|J_{k-2}^{(c_d+2,c_d+2)}(t)|
		\le C_d(k-2)^{-1/2}
		(1-t^2)^{-(c_d+2)/2-1/4},~k\ge3.
	\end{equation}
	
	 Moreover, for $J_{k}^{(c_{d},c_{d})}(1)$, by the definition, we derive 
	\begin{equation}
		\label{eq:estimate-jacobi(1)}
		J_{k}^{(c_{d},c_{d})}(1)=\frac{\Gamma(c_{d}+k+1)}{\Gamma(k+1)\Gamma(c_{d}+1)}\sim\frac{1}{\Gamma(c_{d}+1)}k^{c_{d}}\asymp k^{(d-3)/2}.
	\end{equation}
	Combining \eqref{eq:bound-of-chi-former}, \eqref{eq:bound-of-jacobi-abs} and \eqref{eq:estimate-jacobi(1)}, we obtain, for $k\ge 3$,
	\[ 
	\begin{aligned}
		&\abs{\chi_k(b)}\\
		&\le \frac{\Gamma(d/2)\sqrt{1-b^{2}}}{\sqrt{\pi}\Gamma((d-1)/2)4k(k-1)} \Big(\frac{1-2b^{2}}{1-b^{2}} \Big)^{c_{d}+2}\cdot C_{d} k^{-c_{d}} (k-2)^{-1/2} \Big(\frac{1-2b^{2}}{1-b^{2}} \Big)^{-(c_{d}+2)/2-1/4}\\
		&\le C_{d}(1-b^{2})^{(2-d)/4}(1-2b^{2})^{d/4}k^{-(d+2)/2}.
	\end{aligned}
	\]
	When $k=2$, $J_{k-2}^{(c_{d}+2,c_{d}+2)}$ is constant and can be included by enlarging the constant $C_d$.
	
	\textbf{Consequently}, for $\abs{b}< \frac{1}{\sqrt{2}}$, $\abs{\chi_k(b)}^{2}\le C_{d}(1-b^{2})^{(2-d)/2}(1-2b^{2})^{d/2}k^{-(d+2)}$. And for $\abs{b}\ge \frac{1}{\sqrt{2}}$, $\abs{\chi_k(b)}=0$. We can write (Recall that the final aim is to bound $\kappa_{\gamma_{0}}(\lambda)$ via the score $s_{\lambda}(\theta)$ in \eqref{eq:score-proof-the-aim-to-bound})
	\begin{equation}
		\nonumber
		\abs{\chi_k(b)}^{2}\le C_dG_d(b)k^{-(d+2)}~\text{for}~b\in[-1,1].
	\end{equation}
	This global bound will be used to control endpoint regions where the Darboux expansion below is not uniform.
	\mypara{For item 3:} For $\abs{b}<\frac{1}{\sqrt{2}}$, let $z_{b}=\arccos(-\frac{b}{\sqrt{1-b^{2}}})$. Applying the classical Darboux formula for Jacobi polynomials \cite[Theorem 8.21.8]{szeg1939orthogonal} to $J_{k-2}^{(c_d+2,c_d+2)}(\cos z_b)$ with $c_{d}=(d-3)/2$, we have 
	\begin{equation}
		\nonumber
		\begin{aligned}
			&J_{k-2}^{((d+1)/2,(d+1)/2)}(\cos z_b)\\
			&=\frac{1}{{\pi}^{1/2} (k-2)^{1/2}}  \Big(\sin\frac{z_{b}}{2}\cos\frac{z_{b}}{2} \Big)^{-(d+2)/2}\Big[\cos \Big(\frac{d-2+2k}{2}z_{b}-\frac{\pi(d+2)}{4} \Big)+\mathcal{O}(k^{-1}) \Big]\\
			&=\frac{2^{(d+2)/2}}{{\pi}^{1/2} (k-2)^{1/2}}  \Big(\frac{1-2b^2}{1-b^2} \Big)^{-(d+2)/4} \Big[\cos \Big(\frac{d-2+2k}{2}z_{b}-\frac{\pi(d+2)}{4} \Big)+\mathcal{O}(k^{-1}) \Big],
		\end{aligned}
	\end{equation}
	where we have used \[ \sin z_b = \sqrt{\frac{1-2b^2}{1-b^2}},~ \sin\frac{z_b}{2}\cos\frac{z_b}{2} = \frac{1}{2}\sin z_{b}=\frac{1}{2}\sqrt{\frac{1-2b^2}{1-b^2}}. \] The remainder $\mathcal{O}(k^{-1})$ is uniform for $b$ in every compact subset of $(-1/\sqrt{2},1/\sqrt{2})$. Further combining above estimate of $J_{k-2}^{((d+1)/2,(d+1)/2)}(\cos z_b)$, \eqref{eq:bound-of-chi-former} and \[ J_k^{(c_d,c_d)}(1) = \frac{\Gamma(k+c_d+1)} {\Gamma(k+1)\Gamma(c_d+1)} = \frac{k^{c_d}}{\Gamma(c_d+1)}  (1+\mathcal{O}(k^{-1}) ), \] we obtain \begin{equation} 
        \nonumber
		\begin{aligned} 
			\chi_k(b)&=\frac{\Gamma(d/2)\sqrt{1-b^{2}}}{\sqrt{\pi}\Gamma((d-1)/2)4k(k-1)J_{k}^{(c_{d},c_{d})}(1)} \Big(\frac{1-2b^{2}}{1-b^{2}} \Big)^{c_{d}+2}J_{k-2}^{(c_{d}+2,c_{d}+2)}\Big(-\frac{b}{\sqrt{1-b^{2}}}\Big)\\
			&=\frac{\Gamma(d/2)\sqrt{1-b^{2}}\Gamma(c_{d}+1)(1+\mathcal{O}(k^{-1}))}{\sqrt{\pi}\Gamma((d-1)/2)4k(k-1)k^{c_{d}}} \Big(\frac{1-2b^{2}}{1-b^{2}} \Big)^{c_{d}+2}J_{k-2}^{(c_{d}+2,c_{d}+2)}\Big(-\frac{b}{\sqrt{1-b^{2}}}\Big)\\
			&=C_{1,d}\sqrt{G_{d}(b)}k^{-(d+2)/2} \Big[\cos \Big(\frac{d-2+2k}{2}z_{b}-\frac{\pi(d+2)}{4} \Big)+\mathcal{O}(k^{-1}) \Big],
		\end{aligned} 
	\end{equation} 
	where $C_{1,d}=\frac{2^{(d-2)/2}\Gamma(d/2)}{\pi}$ and $G_d(b)=(1-2b^2)_{+}^{d/2}(1-b^2)^{(2-d)/2}$. The asymptotic relation in above formula is again uniform for $b$ in compact subsets of $(-1/\sqrt{2},1/\sqrt{2})$. Now we finish the proof.
\end{proof}

\begin{corollary}[Sharp Spectral Estimation of $\xi$]
	\label{cor:sharp-degree-spectrum-spheical} Under the spherical biased-ReLU setting,
	\begin{equation}
        \nonumber
		\xi_k=\frac{C_{1,d}^2d}
		{2\sqrt{2}(d+1)}k^{-(d+2)}(1+o(1)),
	\end{equation}
\end{corollary}
\begin{proof}[\textbf{Proof of Corollary \ref{cor:sharp-degree-spectrum-spheical}}]
	By Lemma \ref{lemma:harmonic-representation-leverage-score} and the fact that $\chi_k(b)=0$ for $|b|\ge 1/\sqrt2$, we have
	\[
	k^{d+2}\xi_k=\int_{-1/\sqrt2}^{1/\sqrt2}k^{d+2}\chi_k(b)^2p(b)\,db.
	\]
	To estimate this integral, let
	$z_b=\arccos(-b/\sqrt{1-b^2})$. Squaring \eqref{eq:compact-asymptotic-chi} gives
	\begin{equation}
		\label{eq:chi-square-darboux}
		k^{d+2}\chi_k(b)^2 =
		C_{1,d}^2G_d(b)
		 \Big[
		\cos^2 \Big(
		 (k+\frac{d-2}{2} )z_b
		-\frac{\pi(d+2)}4
		 \Big)
		+\mathcal O(k^{-1})
		 \Big],
	\end{equation}
	uniformly for $b$ in compact subsets of
	$(-1/\sqrt2,1/\sqrt2)$. For each fixed $|b|<1/\sqrt2$, the remainder $\mathcal{O}(k^{-1})$ in
	\eqref{eq:chi-square-darboux} tends to zero as $k\to\infty$. Using $\cos^2 t=(1+\cos(2t))/2$, we rewrite
	\eqref{eq:chi-square-darboux} as
	\[
	\begin{aligned}
		k^{d+2}\chi_k(b)^2&=\frac{C_{1,d}^2}{2}G_d(b)+
		\frac{C_{1,d}^2}{2}G_d(b)
		\cos \Big(
		(2k+d-2)z_b-\frac{\pi(d+2)}2
		 \Big)\\
		&+C_{1,d}^2G_d(b)\mathcal O(k^{-1}).
	\end{aligned}
	\]
	By the global bound in item~2 of
	Lemma~\ref{lemma:relu-harmonic-coefficient-asymptotics}
	and dominated convergence, the last term contributes
	$o(1)$ after integration. Hence,
	\[
	\begin{aligned}
		k^{d+2}\xi_k
		&=\frac{C_{1,d}^2}{2}
		\int_{-1/\sqrt2}^{1/\sqrt2}G_d(b)p(b)\,db\\
		&+
		\frac{C_{1,d}^2}{2}
		\int_{-1/\sqrt2}^{1/\sqrt2}
		G_d(b)p(b)
		\cos \Big(
		(2k+d-2)z_b-\frac{\pi(d+2)}2
		 \Big)db
		+o(1).
	\end{aligned}
	\]
	
	The first integral can be evaluated directly.
	Using the density in \eqref{eq:definition-pb-density},
	a beta-integral calculation gives
	\[
		\int_{-1/\sqrt2}^{1/\sqrt2}G_d(b)p(b)db=
		\frac{\Gamma((d+1)/2)}
		{\sqrt{\pi}\Gamma(d/2)}
		\int_{-1/\sqrt2}^{1/\sqrt2}
		(1-2b^2)^{d/2}db=\frac{d}{\sqrt2(d+1)}.
	\]
	
	For the second integral, the change of variables
	$z=z_b$, with $b=-\frac{\cos z}{\sqrt{1+\cos^2z}}$, gives
	\[
	\begin{aligned}
		&\int_{-1/\sqrt2}^{1/\sqrt2}
		G_d(b)p(b)
		\cos \Big(
		(2k+d-2)z_b-\frac{\pi(d+2)}2
		 \Big)\,db\\
		&=
		\int_0^\pi
		\widetilde G_d(z)
		\cos \Big(
		(2k+d-2)z-\frac{\pi(d+2)}2
		 \Big)dz,
	\end{aligned}
	\]
	where
	\[
	\widetilde G_d(z)
	=
	\frac{\Gamma((d+1)/2)}
	{\sqrt{\pi}\Gamma(d/2)}
	\frac{\sin^{d+1}z}
	{(1+\cos^2z)^{(d+3)/2}}
	\in L_1(0,\pi).
	\]
	The Riemann--Lebesgue lemma therefore implies that
	this integral tends to zero as $k\to\infty$.
	
	Combining these estimates yields
	\[
	\xi_k
	=
	\frac{C_{1,d}^2d}{2\sqrt2(d+1)}
	k^{-(d+2)}(1+o(1)),
	\]
	which proves the result.
	In particular, $\xi_k\asymp k^{-(d+2)}$ for large $k$.
\end{proof}

\subsection{Proofs of Main Results in Section \ref{sec:specific-activation-functions}}
\label{subsec:proofs-spherical}
\begin{proof}[\textbf{Proof of Proposition \ref{prop:spectrum-biased-relu-kernel} (\nameref*{prop:spectrum-biased-relu-kernel})}] ~
	
	\mypara{Proof Overview.} Rotation invariance and the Funk–Hecke formula identify the spherical harmonics as eigenfunctions of $T_\rho$, while the density of biased-ReLU networks establishes the strict positivity of the eigenvalues (\textbf{item 1}). Completeness of the spherical harmonics then gives the operator decomposition, and Mercer's theorem yields the corresponding kernel expansion (\textbf{item 2}). Finally, the degree-wise eigenvalue decay follows from Corollary~\ref{cor:sharp-degree-spectrum-spheical}. Combining this decay with the spherical harmonic multiplicities gives the decay rate of the eigenvalues ordered with multiplicity (\textbf{item 3}).
	
	\mypara{Proof Details.}	We prove the proposition item by item. 
	\mypara{For item 1:} The rotation invariance of $k_\rho$ implies that it depends on $x$ and $x'$ only through their inner product. Write $h:[-1,1]\to\RR$ such that
	\[
	k_\rho(x,x^{\prime})=h \Big(\frac{1+x^\top x^{\prime}}{2} \Big).
	\]
	For any $Y_{k,j}\in \YY_{k}$, the Funk-Hecke formula in Lemma \ref{lemma:normalized-funk-hecke} gives 
	\[ T_{\rho}Y_{k,j}(x)=\int_{\mathbb{S}^{d-1}}
	h \Big(\frac{1+x^{\top}y}{2} \Big)Y_{k,j}(y)\,d\mu(y)=\xi_{k}Y_{k,j}(x), \]
	where 
	\[ \xi_{k}=\frac{\Gamma(d/2)}{\sqrt{\pi}\Gamma((d-1)/2)}\int_{-1}^{1}h \Big(\frac{1+t}{2} \Big)P_{k}(t)(1-t^{2})^{(d-3)/2}dt. \] 
	Hence every element of $\{Y_{k,j}\}_{j=1}^{N(d,k)}$ is an eigenfunction of $T_{\rho}$ with eigenvalue $\xi_k$. 
	
	We next establish the strict positivity of $
	\xi_{k}$. Since $T_{\rho}Y_{k,j}=\xi_{k}Y_{k,j}, \norm{Y_{k,j}}_{L_{2}(d\mu)}=1$, and $T_{\rho}=S_{\rho}S_{\rho}^{*}$, we can also write
	\[ \xi_{k}=\langle Y_{k,j}, T_{\rho}Y_{k,j} \rangle_{L_{2}(d\mu)}=\langle S_{\rho}^{*}Y_{k,j}, S_{\rho}^{*}Y_{k,j} \rangle_{L_{2}(d\rho)}=\int_{\Theta}\abs{\langle (\theta^{\top}\tilde{\cdot})_{+},Y_{k,j}\rangle_{L_{2}(d\mu)}}^{2}d\rho(\theta)\ge 0. \]
	Suppose $\xi_{k}=0$, then we have 
	\begin{equation}
		\label{eq:contradiction-eq-ref}
		\langle (\theta^{\top}\tilde{\cdot})_{+},Y_{k,j}\rangle_{L_{2}(d\mu)}=0,~\text{for}~\rho\text{-almost every}~\theta\in\Theta.
	\end{equation}
	The map $\theta\mapsto\langle(\theta^\top\tilde{\cdot})_+,Y_{k,j}\rangle_{L_2(d\mu)}$ is continuous on $\mathbb S^d$. Since $\rho$ has full support on $\mathbb S^d$, this identity holds for every $\theta\in\Theta$.
	
	By the universal approximation theorem for non-polynomial activation function \cite{leshno1993multilayer}, with the homogeneity of ReLU, the two-layer network class $\mathcal{F}=\bigcup_{W\ge d}\mathcal{F}(W,2)$ with parameters restricted to $\mathbb S^{d}$ is dense in $C(\mathbb S^{d-1})$. Consequently, $\mathcal{F}$ is dense in $L_2(\mathbb S^{d-1},d\mu)$ since $C(\mathbb S^{d-1})$ is dense in $L_{2}(\mathbb S^{d-1},d\mu)$. Further with \eqref{eq:contradiction-eq-ref}, this means $Y_{k,j}$ is orthogonal to $\mathcal F$ which is dense in $L_2(\mathbb S^{d-1}, d\mu)$. And hence $Y_{k,j}=0$ in $L_2(\mathbb S^{d-1},d\mu)$, contradicting $\norm{Y_{k,j}}_{L_2(d\mu)}=1$. Therefore, $\xi_{k}>0$ for every $k\ge 0$.
	
	It remains to determine the explicit form of $h$ from the definition of $k_\rho$. Denote $u=\frac{\tilde{x}}{\sqrt{2}}$, $v=\frac{\tilde{x}^{\prime}}{\sqrt{2}}$ and $u^{\top}v=t=\cos\varphi_{1}$ where $\varphi_{1}\in[0,\pi]$. Since the $\rho$ is rotation-invariant, we can let
	\[ u=\mathbf e_{1},~\text{and}~ v=\cos\varphi_{1} \mathbf e_{1} +\sin\varphi_{1} \mathbf e_{2}.\]
	Projecting any $\theta\in\mathbb S^{d}$ into the $\operatorname{span}\{u,v\}$, we write $(\theta_{1},\theta_{2})=r(\cos\varphi_{2},\sin\varphi_{2})$ with $\varphi_{2}\in[0,2\pi]$. Thus
	\[
	\begin{aligned}
		&k_{\rho}(x,x^{\prime})\\
		&=\EE_{\theta\sim\rho}[\phi(x;\theta)\phi(x^{\prime};\theta)]=2\EE_{\theta\sim\rho}[(\theta^{\top}u)_{+}(\theta^{\top}v)_{+}]\\
		&=2\EE_{\theta\sim\rho}[r^{2}(\cos\varphi_{2})_{+}(\cos(\varphi_{2}-\varphi_{1}))_{+}]=\frac{2}{\pi(d+1)}\int_{\varphi_{1}-\frac{\pi}{2}}^{\frac{\pi}{2}}\cos\varphi_{2}\cos(\varphi_{2}-\varphi_{1})d\varphi_{2}\\
		&=\frac{1}{\pi(d+1)}[\sin\varphi_{1}+(\pi-\varphi_{1})\cos\varphi_{1}]
		=\frac{1}{\pi(d+1)}[\sqrt{1-t^{2}}+(\pi-\arccos t)t],
	\end{aligned}  \]
	where we use $\EE[r^{2}]=\frac{2}{d+1}$ and $t=\cos\varphi_{1}$. Further with $t=\frac{1+x^{\top}x^{\prime}}{2}$, we obtain the stated $h$.
	
	\mypara{For item 2:} Since $\{Y_{k,j}\}_{k,j}$ is a complete orthonormal basis in $L_2(\mathbb{S}^{d-1},d\mu)$ and the spherical harmonic spaces form an orthogonal decomposition \cite[Theorem 2.38]{atkinson2012spherical}, we have $L_2(\mathbb{S}^{d-1},d\mu) =\bigoplus_{k=0}^{\infty}\YY_{k}$, where the direct sum is in the $L_{2}$ space sense. For any $f\in L_2(\mathbb{S}^{d-1},d\mu)$,  the following series converge in $L_{2}(d\mu)$,
	\[ f(x)=\sum_{k=0}^{+\infty}\sum_{j=1}^{N(d,k)}\langle f, Y_{k,j} \rangle_{L_{2}(d\mu)}Y_{k,j}(x)~\text{and}~T_{\rho}f(x)=\sum_{k=0}^{+\infty}\xi_{k}\sum_{j=1}^{N(d,k)}\langle f, Y_{k,j} \rangle_{L_{2}(d\mu)}Y_{k,j}(x). \]
	This indicates that $T_{\rho}=\sum_{k=0}^{+\infty}\xi_{k}\sum_{j=1}^{N(d,k)}Y_{k,j}\otimes Y_{k,j}$. Moreover, since $k_{\rho}$ is continuous and positive semi-definite on the compact space
	$\mathbb{S}^{d-1}\times\mathbb{S}^{d-1}$, by Mercer's theorem, we also have the expansion of $k_{\rho}$.
	
	\mypara{For item 3:} The estimate $\xi_k\asymp k^{-(d+2)}$ follows directly from Corollary~\ref{cor:sharp-degree-spectrum-spheical}.
	It remains to account for multiplicities and determine the decay of the ordered eigenvalues. Now we rearrange the eigenvalues in non-increasing order with multiplicity. The dimension formula gives
	\[
	\begin{aligned}
		&N(d,k)\\
		&=\frac{2k+d-2}{k}\binom{k+d-3}{d-2}=\frac{2k+d-2}{k}\frac{(k+d-3)!}{(d-2)!(k-1)!}\sim \frac{2}{(d-2)!}k^{d-2}~\text{as}~k\to+\infty.
	\end{aligned}  \] 
	Thus, $N(d,k)\asymp k^{d-2}$. Consequently, the number of eigenvalues, counted with multiplicity, contributed by degrees up to $K$ satisfies $\sum_{k=0}^{K}N(d,k)\asymp K^{d-1}$.  
	
	Combining $\xi_k\asymp k^{-(d+2)}$ with
	$\sum_{k=0}^{K}N(d,k)\asymp K^{d-1}$ shows that
	an eigenvalue with ordered index $n$ corresponds to
	a degree $K\asymp n^{1/(d-1)}$. Hence,
	\[
	\eta_n=\xi_K
	\asymp K^{-(d+2)}
	\asymp n^{-\frac{d+2}{d-1}},~\text{as }n\to\infty. 
	\]
	This completes the proof.
\end{proof}

Before proving Proposition \ref{prop:anisotropic-spherical-mismatch}, we first prove the special case of uniform spherical distribution with $\tau=1$. 
\begin{proposition}[Mismatch of Uniform Spherical Initialization]
	\label{prop:uniform-spherical-mismatch}
	Under the spherical biased ReLU setting, let $\rho_0=\operatorname{Unif}(\mathbb S^d)$. For every $\varepsilon>0$, there exists
	$\lambda_\varepsilon>0$ such that the mismatch associated
	with the initialization distribution $\rho_0$ satisfies
	\begin{equation}
        \nonumber
		\frac{\sqrt{2}(d+1)}{d}-\varepsilon\le \kappa_{\gamma_0}(\lambda)\le\frac{2\sqrt{2}(d+1)}{d}+\varepsilon,~\text{for}~\lambda\in(0,\lambda_\varepsilon].
	\end{equation}
\end{proposition}

\begin{proof}[\textbf{Proof of Proposition \ref{prop:uniform-spherical-mismatch}}]
	For the upper bound, by the definition of $\kappa_{\gamma_0}(\lambda)$ and \eqref{eq:harmonic-repre-of-essential-terms},
	\[
	\kappa_{\gamma_0}(\lambda) = \operatorname*{ess\,sup}_{\theta\sim\rho_0} \frac{s_\lambda(\theta)}{d(\lambda)} = \operatorname*{ess\,sup}_{\theta\sim\rho_0} \Bigg( \sum_{k=0}^{\infty} \frac{N(d,k)\chi_k(b)^2}{\xi_k+\lambda} \Bigg) \Bigg/ \Bigg( \sum_{k=0}^{\infty} \frac{N(d,k)\xi_k}{\xi_k+\lambda} \Bigg).
	\]
	Fix $\varepsilon>0$. To bound this ratio, we compare $\chi_k(b)^2$ with $\xi_k$ uniformly over $b\in[-1,1]$ for all sufficiently large $k$.
	
	We first consider a compact interval inside $(-1/\sqrt2,1/\sqrt2)$. For $\delta\in(0,1/\sqrt2)$, let $I_\delta=[-1/\sqrt2+\delta, 1/\sqrt2-\delta]$. Squaring \eqref{eq:compact-asymptotic-chi}, using $\cos^2\Psi_k(b)\le1$, and comparing with Corollary \ref{cor:sharp-degree-spectrum-spheical}, we obtain
	\[
	\limsup_{k\to\infty} \sup_{b\in I_\delta} \frac{\chi_k(b)^2}{G_d(b)\xi_k} \le \frac{C_{1,d}^2}{C_{1,d}^2d/[2\sqrt2(d+1)]} = \frac{2\sqrt2(d+1)}d.
	\]
	Therefore, there exists $K\ge2$ such that for any $k\ge K, b\in I_\delta$,
	\[
	\chi_k(b)^2 \le  \Big(\frac{2\sqrt2(d+1)}d+\frac{\varepsilon}{2} \Big) G_d(b)\xi_k.
	\]
	
	We next extend this estimate to the whole interval $[-1,1]$. The global bound in item~2 of Lemma~\ref{lemma:relu-harmonic-coefficient-asymptotics}, together with $\xi_k\asymp k^{-(d+2)}$ and $\xi_k>0$, gives
	\[
	\chi_k(b)^2 \le C_d G_d(b)\xi_k, ~k\ge2,~ b\in[-1,1].
	\]
	Since $G_d(b)$ tends to zero as $|b|\to1/\sqrt2$ and vanishes for $|b|\ge1/\sqrt2$, choose $\delta$ sufficiently small that
	\[
	C_d G_d(b) \le \frac{\varepsilon}{2},~ b\in[-1,1]\setminus I_\delta.
	\]
	For this $\delta$, take the corresponding $K$ in the interior estimate. Since $\chi_k(b)^2\le(\varepsilon/2)\xi_k$ outside $I_\delta$ and $G_d(b)\le1$ inside it, combining the two regions gives
	\[
	\chi_k(b)^2 \le  \Big(\frac{2\sqrt2(d+1)}d+\frac{\varepsilon}{2} \Big)\xi_k, ~k\ge K, ~b\in[-1,1].
	\]
	
	For this fixed $K$, continuity of $\chi_k$ on $[-1,1]$ gives
	\[
	C_K := \sum_{k=0}^{K-1} \frac{N(d,k)}{\xi_k} \sup_{t\in[-1,1]}\chi_k(t)^2 < \infty.
	\]
	Splitting the sum at $K$ and using the uniform bound above, we obtain
	\[
	\begin{aligned} \frac{s_\lambda(\theta)}{d(\lambda)} &= \frac{1}{d(\lambda)}  \Big[ \sum_{k=0}^{K-1} \frac{N(d,k)\chi_k(b)^2}{\xi_k+\lambda} + \sum_{k=K}^{\infty} \frac{N(d,k)\chi_k(b)^2}{\xi_k+\lambda}  \Big] \\ &\le \frac{C_K}{d(\lambda)} +  \Big(\frac{2\sqrt2(d+1)}d+\frac{\varepsilon}{2} \Big) \frac{1}{d(\lambda)} \sum_{k=K}^{\infty} \frac{N(d,k)\xi_k}{\xi_k+\lambda} \le \frac{C_K}{d(\lambda)} + \frac{2\sqrt2(d+1)}d + \frac{\varepsilon}{2}. \end{aligned}
	\]
	Taking the essential supremum yields
	\[
	\kappa_{\gamma_0}(\lambda) \le \frac{2\sqrt2(d+1)}d + \frac{\varepsilon}{2} + \frac{C_K}{d(\lambda)}.
	\]
	
	Since $\xi_k>0$ for every $k$, monotone convergence gives $\lim_{\lambda\to0_+} d(\lambda) = \sum_{k=0}^{\infty} N(d,k) = \infty$. Hence there exists $\lambda_\varepsilon>0$ such that $C_K/d(\lambda)\le\varepsilon/2$ whenever $0<\lambda\le\lambda_\varepsilon$. Consequently,
	\[
	\kappa_{\gamma_0}(\lambda) \le \frac{2\sqrt2(d+1)}d+\varepsilon, 0<\lambda\le\lambda_\varepsilon.
	\]
	For the lower bound, setting $b=0$ in Lemma \ref{lemma:relu-harmonic-coefficient-asymptotics} gives $\Psi_{k}(0)=\frac{(k-2)\pi}{2}$. Combining this with the expansion in \eqref{eq:compact-asymptotic-chi} and Corollary \ref{cor:sharp-degree-spectrum-spheical}, we have
	\begin{equation}
		\label{eq:kappa-reffff-1}
		\frac{\chi_k(0)^2}{\xi_k}=
		\frac{2\sqrt2(d+1)}{d}
		\cos^2\Big(\frac{k\pi}{2}\Big)+o(1),~k\to\infty.
	\end{equation}
	We note that the cosine squared equals one for even $k$ and zero for odd $k$. To compare the contributions of even and odd degrees, the dimension formula \eqref{eq:degree-N-Formula} gives
	\[ \frac{N(d,2k+1)}{N(d,2k)}=\frac{4k+d}{4k+d-2}\frac{2k+d-2}{2k+1}\to 1,~k\to\infty, \]
	and Corollary \ref{cor:sharp-degree-spectrum-spheical} also gives
	\[ 
	\Big|
	\frac{\xi_{2k}+\lambda}{\xi_{2k+1}+\lambda}-1
	\Big|
	=\Big|\frac{\xi_{2k}-\xi_{2k+1}}{\xi_{2k+1}+\lambda}\Big|
	\le\Big|\frac{\xi_{2k}}{\xi_{2k+1}}-1\Big|\to0,~k\to\infty, \]
	uniformly over $\lambda>0$. To get the final lower bound, we need to estimate $\frac{1}{d(\lambda)} \sum_{k=1}^{\infty} \frac{N(d,2k)\xi_{2k}}{\xi_{2k}+\lambda}.$ For fixed $\varepsilon>0$, by the preceding estimates, there exists $K\ge2$, independent of $\lambda$, such that for all $k\ge K$ and $\lambda>0$,
	\[ \Bigg| \frac{N(d,2k)\xi_{2k}}{\xi_{2k}+\lambda} - \frac{N(d,2k+1)\xi_{2k+1}}{\xi_{2k+1}+\lambda} \Bigg| \le \varepsilon \frac{N(d,2k)\xi_{2k}}{\xi_{2k}+\lambda}. \]
	Since $d(\lambda)=\sum_{k=0}^{\infty}\frac{N(d,k)\xi_k}{\xi_k+\lambda}$, we have
	\[ \begin{aligned} &\Big| \frac{1}{d(\lambda)}\Big(2\sum_{k=1}^{\infty} \frac{N(d,2k)\xi_{2k}}{\xi_{2k}+\lambda} -d(\lambda)\Big) \Big| \\ &\le \frac{1}{d(\lambda)} \sum_{k=0}^{2K-1} \frac{N(d,k)\xi_k}{\xi_k+\lambda} + \frac{1}{d(\lambda)} \sum_{k=K}^{\infty} \Big| \frac{N(d,2k)\xi_{2k}}{\xi_{2k}+\lambda} - \frac{N(d,2k+1)\xi_{2k+1}}{\xi_{2k+1}+\lambda} \Big| \\ &\le \frac{\sum_{k=0}^{2K-1}N(d,k)}{d(\lambda)} + \frac{\varepsilon}{d(\lambda)} \sum_{k=K}^{\infty} \frac{N(d,2k)\xi_{2k}}{\xi_{2k}+\lambda}\le \frac{\sum_{k=0}^{2K-1}N(d,k)}{d(\lambda)} +\varepsilon. \end{aligned} \]
	Since $d(\lambda)\to\infty$ as $\lambda\to0_+$ and $\varepsilon>0$ is arbitrary, this proves
	\begin{equation}
		\label{eq:kappa-reffff-2}
		\lim_{\lambda\to0_+} \frac{1}{d(\lambda)} \sum_{k=1}^{\infty} \frac{N(d,2k)\xi_{2k}}{\xi_{2k}+\lambda} =\frac12.
	\end{equation}
	Since $s_\lambda$ is continuous on $\mathbb S^d$ and $\rho_0$ has full support, the essential supremum is at least $s_\lambda((\omega,0))$ for $\theta=(\omega,0)$. Consequently, by \eqref{eq:harmonic-repre-of-essential-terms}, \eqref{eq:kappa-reffff-1} and \eqref{eq:kappa-reffff-2}
	\begin{equation}
		\label{eq:b-0-limit-kappa}
		\begin{aligned} \liminf_{\lambda\to0_+}\kappa_{\gamma_0}(\lambda)&\ge\lim_{\lambda\to0_+}\frac{s_\lambda((\omega,0))}{d(\lambda)}\\ &= \frac{2\sqrt2(d+1)}d \lim_{\lambda\to0_+}\frac{1}{d(\lambda)}\sum_{k=0}^{\infty}\frac{N(d,k)\xi_k}{\xi_k+\lambda}\cos^2\Big(\frac{k\pi}{2}\Big) \\ &= \frac{2\sqrt2(d+1)}d \lim_{\lambda\to0_+}\frac{1}{d(\lambda)}\sum_{k=1}^{\infty}\frac{N(d,2k)\xi_{2k}}{\xi_{2k}+\lambda}= \frac{\sqrt2(d+1)}d. \end{aligned}
	\end{equation}
	After decreasing $\lambda_\varepsilon$ if necessary, the upper and lower bounds together give
	\[
	\frac{\sqrt2(d+1)}d-\varepsilon\le\kappa_{\gamma_0}(\lambda)
	\le\frac{2\sqrt2(d+1)}d+\varepsilon,~0<\lambda\le\lambda_\varepsilon. \qedhere
	\]
\end{proof}

Proposition \ref{prop:uniform-spherical-mismatch} now gives the corresponding result for scaled spherical initialization.
\begin{proof}[\textbf{Proof of Proposition \ref{prop:anisotropic-spherical-mismatch} (\nameref*{prop:anisotropic-spherical-mismatch})}]
	The leverage score associated with
	$\rho_0^{(\tau)}$ satisfies $s_{\lambda,0}^{(\tau)}(\theta)=\frac{s_\lambda(\theta)}{\gamma_0^{(\tau)}(\theta)}$. Moreover, the corresponding effective dimension is still $d(\lambda)$, since
	\[
	\int_{\mathbb S^d} s_{\lambda,0}^{(\tau)}(\theta)d\rho_0^{(\tau)}(\theta)=\int_{\mathbb S^d}\frac{s_\lambda(\theta)}{\gamma_0^{(\tau)}(\theta)}\gamma_0^{(\tau)}(\theta)d\rho(\theta)=d(\lambda).
	\]
	Since $\gamma_0^{(\tau)}(\theta)\ge1/\tau$ and $\gamma_0^{(\tau)}((\omega,0))=1/\tau$ for $\omega\in\mathbb S^{d-1}$, we obtain
	\[ \begin{aligned} \tau\frac{s_\lambda((\omega,0))}{d(\lambda)} &\le\kappa_{\gamma_0^{(\tau)}}(\lambda)=\operatorname*{ess\,sup}_{\theta\in\mathbb S^d}\frac{s_\lambda(\theta)}{\gamma_0^{(\tau)}(\theta)d(\lambda)}\le\tau\operatorname*{ess\,sup}_{\theta\in\mathbb S^d}\frac{s_\lambda(\theta)}{d(\lambda)}=\tau\kappa_{\gamma_0^{(1)}}(\lambda). \end{aligned} \]
	Applying Proposition~\ref{prop:uniform-spherical-mismatch} with $\varepsilon/\tau$ and using the limit at $b=0$ established in \eqref{eq:b-0-limit-kappa} gives
	\[ \frac{\sqrt2\tau(d+1)}d-\varepsilon\le\kappa_{\gamma_0^{(\tau)}}(\lambda)\le\frac{2\sqrt2\tau(d+1)}d+\varepsilon \]
	for all sufficiently small $\lambda>0$.
\end{proof}
\begin{proof}[\textbf{Proof of Theorem \ref{thm:sphericial-polynomial-rates} (\nameref*{thm:sphericial-polynomial-rates})}]
	By Proposition \ref{prop:spectrum-biased-relu-kernel}, the
	ordered eigenvalues of $T_\rho$ satisfy $\eta_k\asymp k^{-\frac{d+2}{d-1}}$. Hence, the polynomial eigenvalue decay assumption in
	Theorem \ref{thm:polynomial-rates} holds with exponent $\alpha=\frac{d+2}{d-1}$.
	
	For the trainable neurons, take the normalized leverage density $\gamma_1(\theta)=\frac{s_\lambda(\theta)}{d(\lambda)}$ which indicates the $\kappa_{\gamma_{1}}(\lambda)=\kappa_{1}^{*}=1$. For the randomly fixed neurons, apply
	Proposition~\ref{prop:anisotropic-spherical-mismatch} with
	$\varepsilon=\frac{2\sqrt{2}\tau(d+1)}{d}$, and there exists $\bar\lambda_{d,\tau}>0$ such that
	\[
	\kappa_{\gamma_0^{(\tau)}}(\lambda)\le\frac{4\sqrt{2}\tau(d+1)}{d},~\text{for}~ 0<\lambda\leq\bar\lambda_{d,\tau}.
	\]
	
	Set $W=\max\{d,m,1\}$. Following \eqref{eq:relation-lambda-lower-bound-choice} in the proof of Theorem \ref{thm:polynomial-rates}, choose
	\begin{equation}
        \nonumber
		\lambda_{P,m}=A_d \Big( \Big({P}+\frac{(1-P)d}{4\sqrt{2}\tau(d+1)} \Big)\frac{9m}{40\log(63m/(20\delta))} \Big)^{-\frac{d+2}{d-1}},
	\end{equation}
	where $A_d>0$ depends only on $d$ and is chosen to satisfy \eqref{eq:relation-lambda-lower-bound-choice}. If $\lambda_{P,m}\leq\bar\lambda_{d,\tau}$, Theorem \ref{thm:polynomial-rates} directly yields,
	with probability at least $1-\delta$,
	\[ \sup_{\norm{f}_{\mathcal H_\rho}\le 1}\inf_{f_{2}^{(P)}\in\mathcal{F}_P(W,2)}\norm{f-f_{2}^{(P)}}_{L_{2}(d\mu)}^{2}\lesssim \Big( \Big({P}+\frac{(1-P)d}{4\sqrt{2}\tau(d+1)} \Big)\frac{9m}{40\log(63m/(20\delta))} \Big)^{-\frac{d+2}{d-1}}. \]
	If $\lambda_{P,m}\ge\bar\lambda_{d,\tau}$, we use the zero function $f_{2}^{(P)}=0$ in $\mathcal F_P(W,2)$. Since
	$\norm{f}_{\mathcal H_\rho}=\norm{q_{f}}_{L_2(d\mu)}\leq1$ implies
	\[
	\norm{f}_{L_2(d\mu)}^2=\norm{T_{\rho}^{1/2}q_{f}}_{L_2(d\mu)}^2=\langle q_{f},T_{\rho}q_{f}\rangle_{L_2(d\mu)}\le \norm{T_\rho}_{\mathrm{op}}\norm{q_{f}}_{L_2(d\mu)}^2\le\norm{T_\rho}_{\mathrm{op}},
	\]
	then we have
	\[ \sup_{\norm{f}_{\mathcal H_\rho}\leq1}\inf_{f_2^{(P)}\in\mathcal F_P(W,2)}\norm{f-f_2^{(P)}}_{L_2(d\mu)}^2\le\norm{T_\rho}_{\mathrm{op}}\le\frac{\norm{T_\rho}_{\mathrm{op}}}{\bar{\lambda}_{d,\tau}}\lambda_{P,m}. \]
	Thus, the same rate holds in both cases, with a constant depending only on the fixed spherical setting and $\tau$.
	
	Besides, since ${P}+\frac{(1-P)d}{4\sqrt{2}\tau(d+1)}\ge \frac{d}{4\sqrt{2}\tau(d+1)}$, we have 
	\[ \lambda_{P,m}\le A_d \Big(\frac{d}{4\sqrt{2}\tau(d+1)}\frac{9m}{40\log(63m/(20\delta))} \Big)^{-\frac{d+2}{d-1}}. \]
	Therefore, $\lambda_{P,m}\leq\bar\lambda_{d,\tau}$ whenever $\frac{9m}{40\log(63m/(20\delta))}\ge \frac{4\sqrt{2}\tau(d+1)}{d}(A_d/\bar{\lambda}_{d,\tau})^{\frac{d-1}{d+2}}$, which means that, for fixed $d,\tau$ and $\delta$, sufficiently large $m$ leads to the same bound result.
\end{proof}

\section{Error Bounds for the Structured Three-Layer Example \ref{exm:structed-three-layer-compar}}
\label{subsec:proof-of-Example}
Let $r=m_{\mathrm{free}}\in\{0,\ldots,20\}$ be the number of trainable first-layer neurons in the three-layer network and the number of trainable hidden neurons in the two-layer network. The two network architectures have the expression
\[ \begin{aligned}
	f_{3}(x)&=\beta_0+
	\sum_{j=1}^{10}\sum_{\ell=1}^{20}\beta_{j,\ell}\omega_{j,\ell}
	\big(\phi(x;\theta_{2j-1})
	+\phi(x;\theta_{2j})-b_{j,\ell}\big)_{+},\\
	f_2(x)&=\beta_0^{(2)}+\sum_{i=1}^{220}\beta_i^{(2)}\phi(x;\theta_i^{(2)}).
\end{aligned} \]

\begin{proposition}[Approximation Error Upper Bounds]
	\label{prop:structured-depth-upper-bounds}
	For the construction in Example \ref{exm:structed-three-layer-compar},
	the approximation errors satisfy
	\[
	\mathbb E\inf_{\theta_{\mathrm{free}},\beta}
	\|f^*-f_d\|_{L_2(d\mu)}^2\leq\mathcal E^{(d)}(r),
	~ d\in\{2,3\},~ r\in\{0,2,4,20\},
	\]
	where the explicit upper bounds are
	\begin{equation}
		\nonumber
		\label{eq:structured-three-layer-upper-bounds}
		\begin{aligned}
			\mathcal E^{(3)}(0)&=\frac19,~\mathcal E^{(3)}(2)=\frac1{45}+\frac{144}{\prod_{j=21}^{25}j},~\mathcal E^{(3)}(4)=\frac{180}{\prod_{j=21}^{25}j},~
			\mathcal E^{(3)}(20)=
			\frac{580}{\prod_{j=101}^{105}j},
		\end{aligned}
	\end{equation}
	and
	\begin{equation}
		\nonumber
		\label{eq:structured-two-layer-upper-bounds}
		\mathcal E^{(2)}(0)=\frac19,~\mathcal E^{(2)}(2)=\frac{17}{729},~\mathcal E^{(2)}(4)=\frac1{729},~
		\mathcal E^{(2)}(20)=\frac1{9\cdot11^4}.
	\end{equation}
	Here $\mathcal E^{(d)}(r)$ denotes the stated upper bound.
\end{proposition}

\begin{proof}[\textbf{Proof of Proposition~\ref{prop:structured-depth-upper-bounds}}] \textbf{Overview.} The zero output gives a common baseline upper bound $\mathcal{E}^{(d)}(0)$. \textbf{For the three-layer network}, pairs of trainable first-layer neurons produce absolute coordinate values, allowing the fixed second-layer blocks to implement piecewise-linear interpolation with random nodes. The interpolation residuals and moments of the random spacings then yield the expected error bounds $\mathcal{E}^{(3)}(r)$. \textbf{For the two-layer network}, the trainable neurons directly implement symmetric piecewise-linear interpolation with equally spaced nodes. Centering the interpolation residuals and integrating their squares then yields the two-layer error bounds for the different numbers of trainable neurons $\mathcal{E}^{(2)}(r)$.
	
	\mypara{Proof Details.} Since the coordinates of $x$ are independent and $\int_{-1}^1\frac{t^2-1/3}{2}\,{dt}=0$ with  $\int_{-1}^1\frac{(t^2-1/3)^2}{2}\,{dt}=\frac4{45}$, the target satisfies $\|f^*\|_{L_2(d\mu)}^2
	=\frac4{45}(1+\frac14 )=\frac19$. Choosing all output coefficients and the output bias to be zero therefore gives the baseline upper bound $\mathcal{E}^{(d)}(0)$ for both architectures.
	
	\mypara{The Three-Layer Upper Bound.}
	Since $\omega_{j,\ell}>0$, the approximation residual is
	\[
	\begin{aligned}
		&f_3(x)-f^*(x)\\
		&=\beta_0+
		\sum_{j=1}^{10}\sum_{\ell=1}^{20}
		\beta_{j,\ell}\omega_{j,\ell}
		\bigl[
		\phi(x;\theta_{2j-1})+
		\phi(x;\theta_{2j})-b_{j,\ell}
		\bigr]_+-(x_2^2-1/3)-\frac12(x_3^2-1/3).
	\end{aligned}
	\]
	
	\textbf{(\romannumeral1) Construction with Four Trainable Neurons.} With four trainable first-layer neurons, choose them as 
	\[
	\theta_1=(\mathbf e_2,0),~
	\theta_2=(-\mathbf e_2,0),~
	\theta_3=(\mathbf e_3,0),~
	\theta_4=(-\mathbf e_3,0).
	\]
	Their pairwise sums are $|x_2|$ and $|x_3|$.
	Setting the output coefficients of all other blocks to zero
	therefore gives
	\[
	\begin{aligned}
		f_3(x)-f^*(x)&=\beta_0+
		\sum_{\ell=1}^{20}\beta_{1,\ell}\omega_{1,\ell}
		(|x_2|-b_{1,\ell})_++\sum_{\ell=1}^{20}\beta_{2,\ell}\omega_{2,\ell}
		(|x_3|-b_{2,\ell})_+\\
        &\quad -(x_2^2-1/3)-\frac12(x_3^2-1/3).
	\end{aligned}
	\]
	Although $x_k$ ranges over $[-1,1]$, the folded input $|x_k|$
	is uniform on $[0,1]$. The interpolation below is therefore
	performed on $[0,1]$ with $u=|x_k|$, for which $u^2=x_k^2$.
	
	For each selected block $j=1,2$, reorder and relabel its
	biases and the corresponding weights so that, almost surely,
	\[
	0<b_{j,1}<\cdots<b_{j,20}<1,~b_{j,0}=0,~ b_{j,21}=1.
	\]
	These biases serve as interpolation nodes.
	For the first block, approximate $u^2$ by a function that
	is constant with value $b_{1,1}^2$ on $[0,b_{1,1}]$ and
	linearly interpolates $u^2$ on each
	$[b_{1,\ell},b_{1,\ell+1}]$.
	The linear interpolant has slope
	$b_{1,\ell}+b_{1,\ell+1}$ on each such interval.
	Those coefficients $\beta_{1,\ell}\omega_{1,\ell}$ decide the successive slope differences,
	with zero slope on the initial interval.
	Consequently, choose
	\[
	\beta_{1,\ell}=
	\begin{cases}
		\frac{b_{1,1}+b_{1,2}}{\omega_{1,1}},
		&\ell=1,\\[5pt]
		\frac{b_{1,\ell+1}-b_{1,\ell-1}}{\omega_{1,\ell}},
		&2\leq\ell\leq20.
	\end{cases}
	\]
	Together with the constant term $b_{1,1}^2$, these coefficients
	give the prescribed interpolation, with residual
	\begin{equation}
		\label{eq:residual-blocks}
        \begin{aligned}
            &b_{1,1}^2+
		\sum_{\ell=1}^{20}\beta_{1,\ell}\omega_{1,\ell}
		(u-b_{1,\ell})_+-u^2\\
        &\qquad=
		\begin{cases}
			b_{1,1}^2-u^2,
			&0\leq u\leq b_{1,1},\\
			(u-b_{1,\ell})(b_{1,\ell+1}-u),
			&1\leq\ell\leq20,~b_{1,\ell}\leq u\leq b_{1,\ell+1}.
		\end{cases}
        \end{aligned}
	\end{equation}
	Evaluating this expression at $u=|x_2|$ gives the residual
	for the first target coordinate on $[-1,1]$.
	For the second block, apply the same construction to
	$\frac12u^2$ and evaluate it at $u=|x_3|$. This gives
	\[
	\beta_{2,\ell}=
	\begin{cases}
		\frac{b_{2,1}+b_{2,2}}{2\omega_{2,1}},
		&\ell=1,\\[5pt]
		\frac{b_{2,\ell+1}-b_{2,\ell-1}}{2\omega_{2,\ell}},
		&2\leq\ell\leq20.
	\end{cases}
	\]
	Its interpolation residual has the same form with its own
	biases and a factor of $1/2$.
	
	\textbf{(\romannumeral2) Centering and Error Estimation.} Since the target has zero spatial mean,
	$\int_{\mathcal X}f^*\,d\mu=0$, center the constructed approximation so that $\int_{\mathcal X}f_{3}d\mu=0$ by choosing
	\[
	\begin{aligned}
		\beta_0
		&=b_{1,1}^2+\frac12b_{2,1}^2-\Big[b_{1,1}^2+\frac12b_{2,1}^2+\sum_{j=1}^{2}\sum_{\ell=1}^{20}
		\beta_{j,\ell}\omega_{j,\ell}
		\frac12\int_{-1}^{1}(|u|-b_{j,\ell})_+\,du\Big]\\
		&=-\frac12\sum_{j=1}^{2}\sum_{\ell=1}^{20}
		\beta_{j,\ell}\omega_{j,\ell}(1-b_{j,\ell})^2,
	\end{aligned}
	\]
	where the first $b_{1,1}^{2} + b_{2,1}^{2}/2$ supplies the constant terms of the two interpolants, and the later bracketed expression gives the combined spatial mean. With all coefficients determined, the resulting residual is 
	\[
	\begin{aligned}
		f_3(x)-f^*(x)
		&=
		\sum_{\ell=1}^{20}\beta_{1,\ell}\omega_{1,\ell}
		\Big[
		(|x_2|-b_{1,\ell})_+
		-\frac{(1-b_{1,\ell})^2}{2}
		\Big]-(x_2^2-1/3)\\
		&\quad+
		\sum_{\ell=1}^{20}\beta_{2,\ell}\omega_{2,\ell}
		\Big[
		(|x_3|-b_{2,\ell})_+
		-\frac{(1-b_{2,\ell})^2}{2}
		\Big]-\frac12(x_3^2-1/3).
	\end{aligned}
	\]
	Each coordinate error has zero spatial integral, so the cross term vanishes when the squared norm is expanded. Taking the first block as an example, since $x_2$ is uniformly distributed on $[-1,1]$, we have 
	\[ 
	\begin{aligned}
		&\EE \int_{-1}^{1}
		\Big(		\sum_{\ell=1}^{20}\beta_{1,\ell}\omega_{1,\ell}
		\Big[
		(|x_2|-b_{1,\ell})_+
		-\frac{(1-b_{1,\ell})^2}{2}
		\Big]
		-(x_2^2-1/3)
		\Big)^2\frac{dx_2}{2}\\
		&=\EE \int_{0}^{1}\Big(b_{1,1}^{2}+\sum_{\ell=1}^{20}\beta_{1,\ell}\omega_{1,\ell}(u-b_{1,\ell})_+-u^{2}\Big)^2du \\ &\qquad\qquad\qquad-\EE \Big(\int_{0}^{1}\Big[b_{1,1}^{2}+\sum_{\ell=1}^{20}\beta_{1,\ell}\omega_{1,\ell}(v-b_{1,\ell})_+-v^{2}\Big]dv\Big)^2,
	\end{aligned}
	\]
	where we use $(1-b_{1,\ell})^{2}/2=\int_{0}^{1}(v-b_{1,\ell})_{+}dv$, $\int_{0}^{1}v^{2}dv=1/3$ and replace $x_{2}$ with $u$. Use the inequality for any square-integrable function $g$ on $[0,1]$,
	\[
	\int_0^1\Big[g(u)-\int_0^1g(v)\,dv\Big]^2du
	=
	\int_0^1g(u)^2du
	-\Big[\int_0^1g(u)\,du\Big]^2
	\leq\int_0^1g(u)^2du,
	\] 
	we get 
	\[ 
	\begin{aligned}
		&\EE \int_{-1}^{1}
		\Big(		\sum_{\ell=1}^{20}\beta_{1,\ell}\omega_{1,\ell}
		\Big[
		(|x_2|-b_{1,\ell})_+
		-\frac{(1-b_{1,\ell})^2}{2}
		\Big]
		-(x_2^2-1/3)
		\Big)^2\frac{dx_2}{2}\\
		&\le \EE \int_{0}^{1}
		\Big(	b_{1,1}^{2}+\sum_{\ell=1}^{20}\beta_{1,\ell}\omega_{1,\ell}
		(u-b_{1,\ell})_+-u^{2}	
		\Big)^2du
	\end{aligned}
	\]
	Therefore, further using the residual results in \eqref{eq:residual-blocks}, direct integration yields
	\[
	\begin{aligned}
		\mathbb E\|f_3-f^*\|_{L_2(d\mu)}^2
		&\leq
		\mathbb E\int_0^1
		\Big(
		b_{1,1}^2+
		\sum_{\ell=1}^{20}\beta_{1,\ell}\omega_{1,\ell}
		(u-b_{1,\ell})_+-u^2
		\Big)^2du\\
		&\qquad\qquad+\mathbb E\int_0^1
		\Big(
		\frac12b_{2,1}^2+
		\sum_{\ell=1}^{20}\beta_{2,\ell}\omega_{2,\ell}
		(u-b_{2,\ell})_+-\frac12u^2
		\Big)^2du\\
		&=\mathbb E\Big(
		\frac8{15}(b_{1,1}-b_{1,0})^5+
		\frac1{30}\sum_{\ell=1}^{20}
		(b_{1,\ell+1}-b_{1,\ell})^5
		\Big)\\
		&\qquad\qquad+\frac14\mathbb E\Big(
		\frac8{15}(b_{2,1}-b_{2,0})^5+
		\frac1{30}\sum_{\ell=1}^{20}
		(b_{2,\ell+1}-b_{2,\ell})^5
		\Big),
	\end{aligned}
	\]
	where $b_{j,0}=0, j\in\{1,2\}$.
	
	By the uniform-spacing identity
	\cite{pyke1965spacings},
	each spacing $b_{j,\ell+1}-b_{j,\ell},\ell\in\{0,1,\cdots,20\}$ has the density $p(x)=20(1-x)^{19}$. Therefore, $\mathbb E(b_{j,\ell+1}-b_{j,\ell})^5
	=20\int_0^1x^5(1-x)^{19}\,dx
	=\frac{120}{\prod_{\ell=21}^{25}\ell}$. Consequently,
	\[
	\begin{aligned}
		\mathbb E\|f_3-f^*\|_{L_2(d\mu)}^2
		&\leq
		 \Big(1+\frac14 \Big)
		 \Big(\frac8{15}+\frac{20}{30} \Big)
		\frac{120}{\prod_{\ell=21}^{25}\ell}=\frac{180}{\prod_{\ell=21}^{25}\ell}
		=\mathcal E^{(3)}(4).
	\end{aligned}
	\]
	Taking the infimum over the trainable hidden parameters and
	output coefficients can only decrease the error, thereby proving the upper bound with four trainable neurons.
	
	\textbf{(\romannumeral3) Two and Twenty Trainable Neurons.} With two trainable first-layer neurons, choose $\theta_1=(\mathbf e_2,0),\theta_2=(-\mathbf e_2,0)$. Their sum is $|x_2|$, so the first second-layer block provides 20 random hinge features for approximating $x_2^2-1/3$. Set the output coefficients of all other blocks to zero. Applying the single-block interpolation estimation derived above and leaving the $x_3$-component
	unapproximated gives
	\[
	\begin{aligned}
		\mathbb E\inf_{\theta_{\mathrm{free}},\beta}
		\|f_3-f^*\|_{L_2(d\mu)}^2
		&\leq
		\frac1{45}
		+ \Big(\frac8{15}+\frac{20}{30} \Big)
		\frac{120}{\prod_{j=21}^{25}j}=
		\frac1{45}
		+\frac{144}{\prod_{j=21}^{25}j}
		=\mathcal E^{(3)}(2).
	\end{aligned}
	\]
	
	With 20 trainable first-layer neurons, assign 5 pairs to
	$(x_2)_++(-x_2)_+$ and the other 5 pairs to
	$(x_3)_++(-x_3)_+$. Each coordinate then feeds 5 fixed second-layer blocks, providing 100 independent uniform
	biases. Pooling and ordering these biases gives the same
	interpolation construction with 100 nodes per coordinate.
	Repeating the calculation above yields
	\[
	\begin{aligned}
		\mathbb E\inf_{\theta_{\mathrm{free}},\beta}
		\|f_3-f^*\|_{L_2(d\mu)}^2
		&\leq
		 \Big(1+\frac14 \Big)
		 \Big(\frac8{15}+\frac{100}{30} \Big)
		\frac{120}{\prod_{\ell=101}^{105}\ell}=\frac{580}{\prod_{\ell=101}^{105}\ell}
		=\mathcal E^{(3)}(20).
	\end{aligned}
	\]
	All second-layer parameters remain fixed in this construction.
	
	\mypara{The Two-Layer Upper Bound.}
	We use the following symmetric interpolation construction for
	approximating $u^2$. Let $q\in\{1,2,\ldots,5\}$, set
	$m_{\mathrm{free}}=4q$, and consider the equally spaced nodes $u_i=-1+\frac{2i}{2q+1},~0\leq i\leq2q+1$. The piecewise-linear interpolant of $u^2$ at these nodes is
	\[
	I_q(u)
	=
	\frac1{(2q+1)^2}
	+\frac4{2q+1}\sum_{r=1}^{q}
	\Big[
	\Big(u-\frac{2r-1}{2q+1}\Big)_+
	+
	\Big(-u-\frac{2r-1}{2q+1}\Big)_+
	\Big].
	\]
	Thus, $I_q$ is represented by $q$ symmetric pairs of
	ReLU features, or $2q$ neurons in total. On every interval $[u_i,u_{i+1}]$, we have $I_q(u)-u^2=(u-u_i)(u_{i+1}-u)$. The spatial mean of the interpolation residual is
	\[
	\int_{-1}^{1}\bigl[I_q(u)-u^2\bigr]\frac{du}{2}=
	\frac{2q+1}{2}\int_0^{2/(2q+1)} u\Big(\frac{2}{2q+1}-u\Big)\,du=\frac{2}{3(2q+1)^{2}}.
	\]
	Therefore, $I_q(u)-\frac13-\frac{2}{3(2q+1)^{2}}$ has zero spatial mean and approximates $u^2-1/3$.
	Moreover,
	\[
	\int_{-1}^{1}\Big[I_q(u)-\frac13-\frac{2}{3(2q+1)^{2}}-(u^2-1/3)\Big]^2\frac{du}{2}
	=\frac{4}{45(2q+1)^4}.
	\]
	
	To implement these two coordinate-wise interpolants, assign
	$q$ symmetric pairs of trainable neurons to the $x_2$-coordinate
	and another $q$ symmetric pairs to the $x_3$-coordinate.
	For $1\leq r\leq q$, choose the corresponding hidden
	parameters as
	\[
	\begin{aligned}
		\theta_{2r-1}^{(2)}
		&=
		\Big(\mathbf e_2,-\frac{2r-1}{2q+1}\Big),
		&
		\theta_{2r}^{(2)}
		&=
		\Big(-\mathbf e_2,-\frac{2r-1}{2q+1}\Big),\\
		\theta_{2q+2r-1}^{(2)}
		&=
		\Big(\mathbf e_3,-\frac{2r-1}{2q+1}\Big),
		&
		\theta_{2q+2r}^{(2)}
		&=
		\Big(-\mathbf e_3,-\frac{2r-1}{2q+1}\Big).
	\end{aligned}
	\]
	Choose their output coefficients as $	\beta_{2r-1}^{(2)}=\beta_{2r}^{(2)}=\frac4{2q+1},~
	\beta_{2q+2r-1}^{(2)}=\beta_{2q+2r}^{(2)}=\frac2{2q+1}$, where the second pair of coefficients includes the factor
	$1/2$ multiplying the $x_3$-component of the target.
	Set the output coefficients of all fixed neurons to zero. Finally, choose the output bias as
	\[
	\beta_0^{(2)}= \Big(1+\frac12 \Big) \Big[\frac1{(2q+1)^2}-\frac13-\frac{2}{3(2q+1)^{2}} \Big]=\frac{1-(2q+1)^2}{2(2q+1)^2}.
	\]
	For these parameters, the network output is
	\[
	f_2(x)
	=
	I_q(x_{2})-\frac13-\frac{2}{3(2q+1)^{2}}
	+\frac12
	 \Big[
	I_q(x_{3})-\frac13-\frac{2}{3(2q+1)^{2}}
	 \Big].
	\]
	
	The two coordinate errors have zero spatial mean and depend
	on independent coordinates, so their cross term vanishes.
	Consequently,
	\[
	\|f_2-f^*\|_{L_2(d\mu)}^2= \Big(1+\frac14 \Big)
	\frac{4}{45(2q+1)^4}=\frac1{9(2q+1)^4}.
	\]
	The infimum over the trainable hidden parameters and output
	coefficients is no larger than the error of this explicit
	choice. Since the construction is available for every
	realization of the fixed random neurons, taking expectation
	over those parameters gives
	\[
	\EE\inf_{\theta_{\mathrm{free}},\beta}\norm{f^{*}-f_{2}}_{L_{2}(d\mu)}^{2}\le \mathcal E^{(2)}(m_{\mathrm{free}})=\frac1{9(m_{\mathrm{free}}/2+1)^4}.
	\]
	In particular, $ \mathcal E^{(2)}(4)=\frac1{9\cdot3^4}=\frac1{729},~\mathcal E^{(2)}(20)=\frac1{9\cdot11^4}$.
	
	For the case of two trainable hidden neurons, use one symmetric pair for the $x_2$-coordinate. Taking $q=1$ in the interpolation construction gives the centered approximation error $\frac{4}{45\cdot3^4}=\frac4{3645}$ for $x_2^2-1/3$. Leaving the $x_3$-component
	unapproximated therefore yields
	\[
	\begin{aligned}
		\mathbb E\inf_{\theta_{\mathrm{free}},\beta}
		\|f_2-f^*\|_{L_2(d\mu)}^2
		&\leq
		\frac4{3645}+\frac1{45}=\frac{17}{729}
		=\mathcal E^{(2)}(2).
	\end{aligned} \qedhere
	\]
\end{proof}
For $m_{\mathrm{free}}\in\{2,4\}$, substitution gives $G^{(3)}(2)\approx0.7998,
G^{(3)}(4)\approx0.9997,G^{(2)}(2)\approx0.7902,G^{(2)}(4)\approx0.9877$. 

\section{Complementary Experiments}
\label{sec:addition-experiments}
\subsection{Effect of the Initialization Scale}
\label{subsec:tau-experiments}

This experiment examines the dependence of the small-$P$ transition on $\tau$ for a fixed target in the spherical biased-ReLU RKHS in Section \ref{sec:specific-activation-functions}.

\mypara{Target Function.} Let $\mu=\operatorname{Unif}(\mathbb S^4)$ and $\rho=\operatorname{Unif}(\mathbb S^5)$. Define
\[
Y_2(x)=\sqrt{\frac{35}{92}}(4x_1^2+2x_2^2+x_3^2-3x_4^2-4x_5^2),~\text{and}~Y_4(x)=\sqrt{3465}\,x_1x_2x_3x_4.
\]
These are orthonormal spherical harmonics of degrees $2$ and $4$. We take the target function
\[
f^*(x)=C_{\xi}\big(\sqrt{0.9}Y_2(x)+\sqrt{0.1}Y_4(x)\big),~\text{with}~C_{\xi}= \Big(\frac{0.9}{\xi_2}+\frac{0.1}{\xi_4}\Big)^{-1/2},
\]
where $\xi_2$ and $\xi_4$ are the corresponding eigenvalues of $T_\rho$. By the spectral characterization of $\mathcal H_\rho$ in \eqref{eq:space-defition-of-RKHS-spericial} and the $L_2(d\mu)$ orthonormality of $Y_2$ and $Y_4$, we have
\[
\norm{f^*}_{\mathcal H_\rho}^{2}=\Big(\frac{0.9}{\xi_2}+\frac{0.1}{\xi_4}\Big)^{-1}\Big(\frac{0.9}{\xi_2}+\frac{0.1}{\xi_4}\Big)=1.
\]
Thus $\norm{f^*}_{\mathcal H_\rho}=1$ and $f^{*}\in\mathcal{H}_{\rho}$. The target and reference RKHS remain unchanged as $\tau$ varies. Training uses $f^*/C_{\xi}$, with both the target and network output rescaled by $C_{\xi}$ for reporting. This common scaling leaves the normalized gain unchanged.

\mypara{Experimental Set-up.}
\begin{itemize}[labelindent=0pt, leftmargin=*, topsep=0pt, itemsep=0pt]
	\item \textbf{Network and Parameter Grid.} We use a two-layer ReLU network of width $64$, with
	\[
	P\in\frac1{64}\Big\{0,~1,~2,~\ldots,~19,~20,~24,~32,~48,~64\Big\},
	\]
	and $\tau\in\{1,~1.2,~1.5,~2,~2.5,~3,~3.5,~4,~5,~7.5,~8,~9,~10\}$.
	\item \textbf{Initialization Distributions.} Fixed neurons are initialized from $\rho_0^{(\tau)}$, while trainable neurons use a numerical approximation to the ridge leverage distribution in \eqref{eq:leverage-score-distribution}. We use $\lambda\approx4.5244\times10^{-4}$ so that $d(\lambda)\approx64/\log64$ for all $P$ and $\tau$. In the spherical setting, $s_{\lambda}(\theta)$ depends only on the bias $b$, so the bias density is proportional to $p(b)s_{\lambda}(\theta)$ and then normalized. The leverage score $s_{\lambda}(\theta)$ is evaluated using harmonic degrees up to 32, with the omitted tail replaced by the remaining feature energy divided by $\lambda$ as an upper bound. The bias density is evaluated on 4097 grid points. We mix the approximate leverage distribution with the uniform reference distribution $\rho$, assigning weight $10^{-3}$ to $\rho$ to ensure positivity. Trainable hidden parameters are unconstrained during training, and all output parameters are also optimized.
	\item \textbf{Training and Repetitions.} Each configuration is trained on the same $12{,}000$ independent samples from $\mu$ for $12{,}000$ full-batch Adam steps and repeated $21$ times. The remaining optimizer settings follow Table \ref{tab:experimental-configuration}. Within each repetition, the two independent initialization pools are shared across $P$ and $\tau$. The full-training result is reused across $\tau$, since no hidden neurons remain fixed at $P=1$.
\end{itemize}

\mypara{Metrics.} The empirical gain uses the same geometric mean training loss and normalization as in Section \ref{sec:numerical-illustration}. The numerical threshold $P_{0.1}$ is the smallest tested $P$ satisfying $G(P)\ge0.9$. Pointwise $95\%$ intervals are obtained from $500$ paired bootstrap resamples across repetitions. For $d=5$, the theoretical curves use \eqref{eq:normalized-gain} with $\alpha=7/4$ and $r=24\sqrt2\tau/5$.

\mypara{Results and Interpretation.} Figure \ref{fig:intro-spherical-phase-transition} is reproduced below. The numerical threshold $P_{0.1}$ decreases overall as $\tau$ increases. The gain curve for $\tau=10$ also rises more rapidly at small $P$ than that for $\tau=1$. These observations are qualitatively consistent with the theoretical trend. The theoretical curves are derived from the approximation upper bound over the RKHS unit ball in Theorem \ref{thm:sphericial-polynomial-rates} with $d=5$, whereas the numerical curves describe training performance for one fixed target $f^{*}$ in that ball under finite sampling and a fixed training budget. Quantitative differences between the curves are therefore expected. Nevertheless, the numerical results exhibit the same overall dependence on $\tau$, providing empirical support for the theoretical prediction within the scaled spherical family.
\begin{figure}[H]
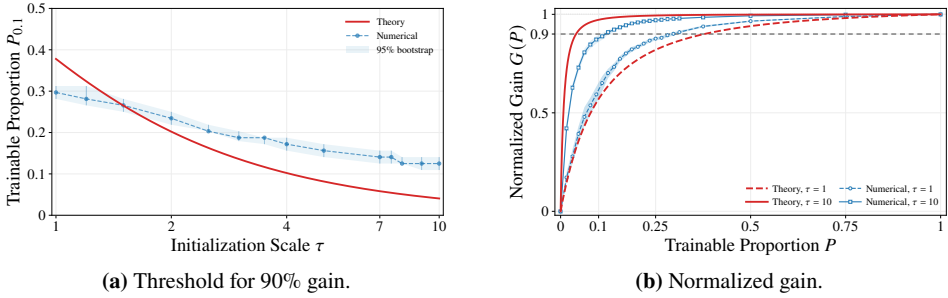

	\centering
	\subfloat[Threshold for $90\%$ gain.]{%
		\includegraphics[width=0.47\linewidth]
		{Figs/introduction/S3_d5_training_threshold_geometric_90.pdf}%
	}
	\hfill
	\subfloat[Normalized gain.]{%
		\includegraphics[width=0.47\linewidth]
		{Figs/introduction/S3_d5_training_gain_geometric_90.pdf}%
	}
	\caption*{ \textbf{Figure \ref{fig:intro-spherical-phase-transition} (reproduced).} Theoretical and numerical phase transitions for two-layer spherical biased-ReLU $P$-trainable networks with $d=5$. Red curves are derived from the approximation bound, while blue curves show illustrative training results for a fixed RKHS target. Threshold $P_{0.1}$ marks $90\%$ normalized gain. Shaded bands indicate $95\%$ bootstrap intervals.}
	\setcounter{subfigure}{0}
\end{figure}

\subsection{Additional Experiments with Other Activation Functions}
This section presents the width and depth experiments with GELU and $\tanh$, using the same targets, architectures, and training settings as in Section \ref{subsubsec:phase-transition-setting}. Figures \ref{fig:GeLU-Target1} to \ref{fig:tanh-target3} below show small-$P$ phase transition patterns similar to those observed with ReLU. For GELU on $f_2^*$, the two-layer network also shows normalized gains slightly above $1$, so we omit this gain curve. Partial training attains a slightly lower loss than full training at some $P$ under the fixed training budget. The corresponding differences in absolute training loss are only of order $10^{-5}$, representing small fluctuations that do not alter the overall phase transition pattern.
\begin{figure}[p]
	\centering
	\subfloat[Depth variation loss.]{
		\includegraphics[width=0.47\linewidth]
		{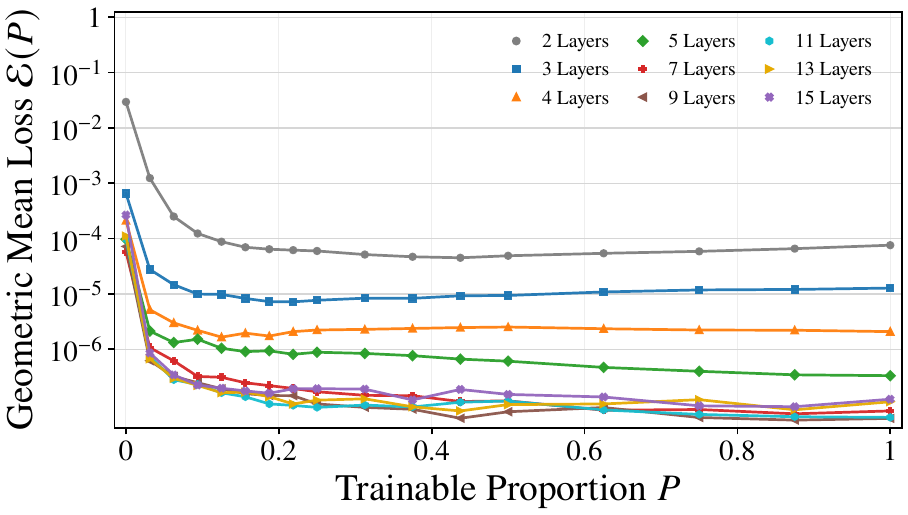}
	}\hfill
	\subfloat[Depth variation gain.]{
		\includegraphics[width=0.47\linewidth]
		{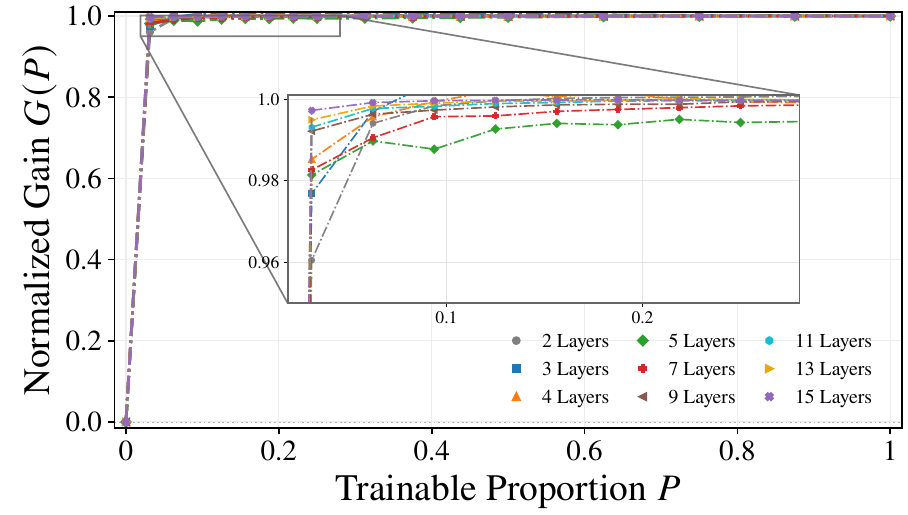}
	}
	
	\subfloat[Width variation loss.]{
		\includegraphics[width=0.47\linewidth]
		{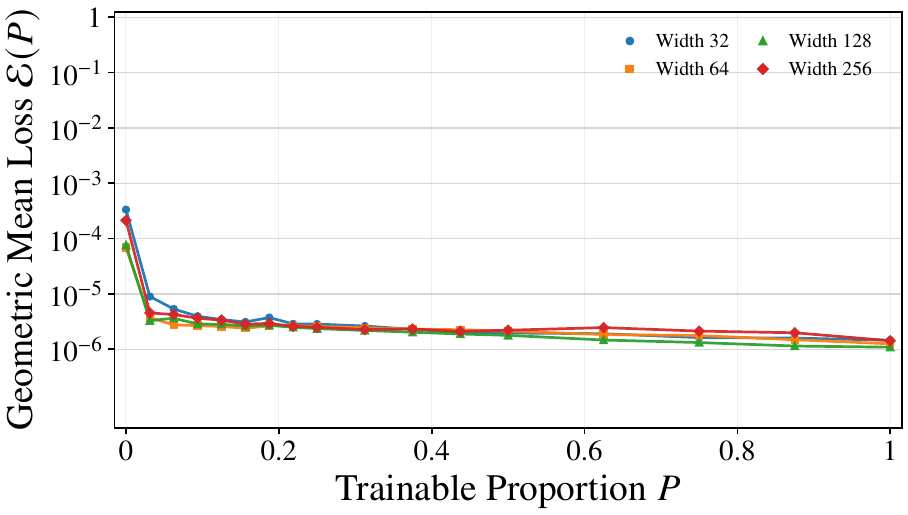}
	}\hfill
	\subfloat[Width variation gain.]{
		\includegraphics[width=0.47\linewidth]
		{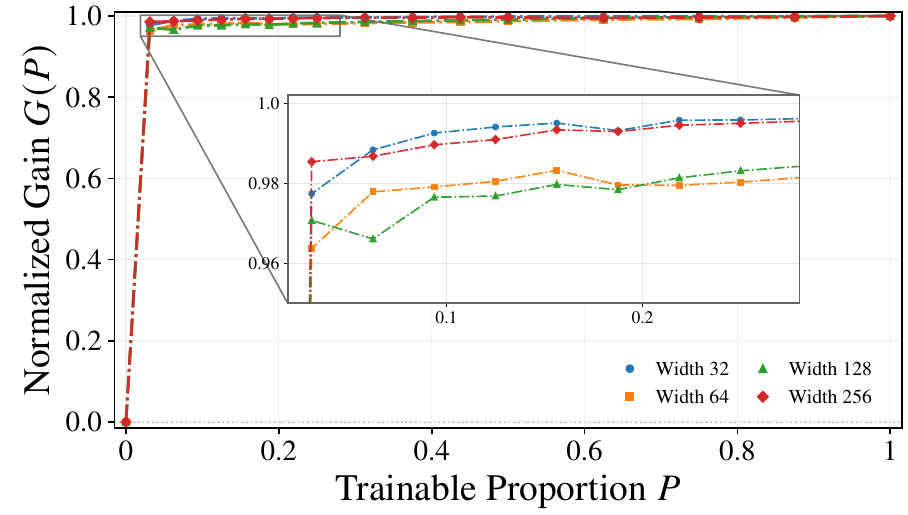}
	}
	\caption{Training results for $f_1^*$ with GELU activation.}
	\label{fig:GeLU-Target1}
	\setcounter{subfigure}{0}
	\subfloat[Depth variation loss.]{
		\includegraphics[width=0.47\linewidth]
		{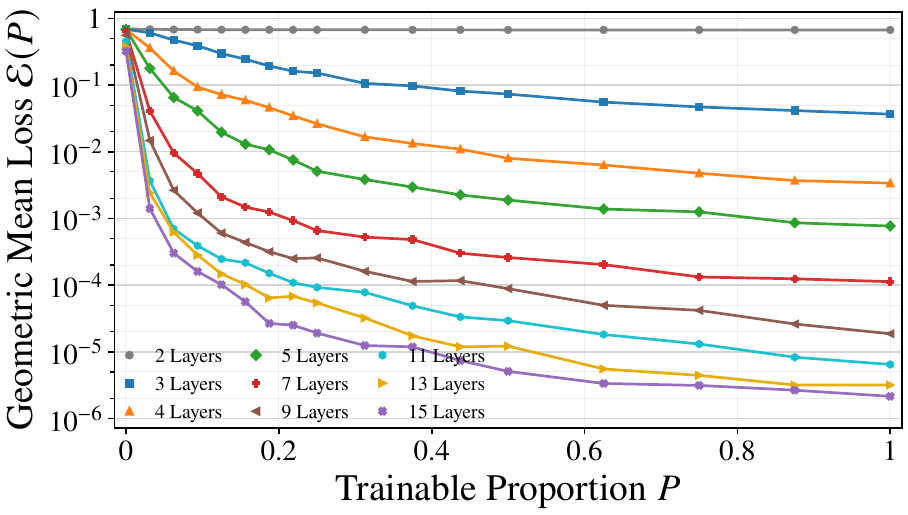}
	}\hfill
	\subfloat[Depth variation gain.]{
		\includegraphics[width=0.47\linewidth]
		{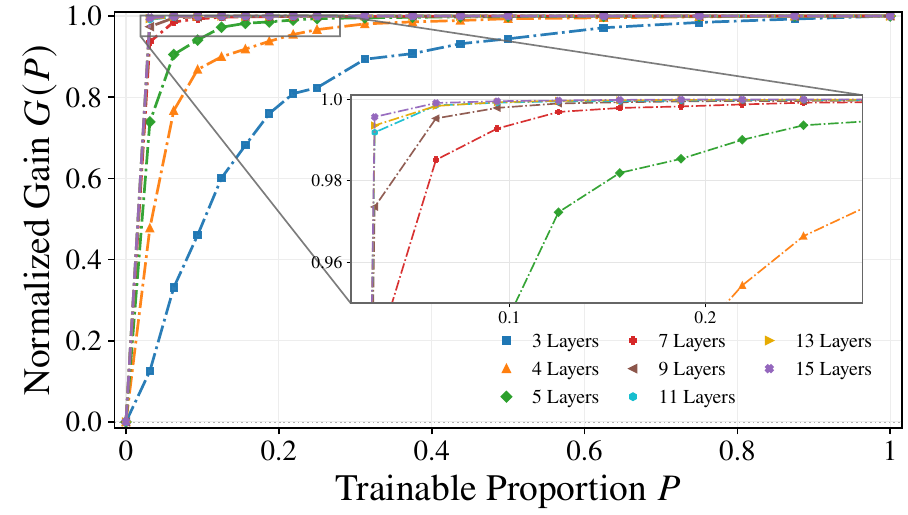}
	}
	
	\subfloat[Width variation loss.]{
		\includegraphics[width=0.47\linewidth]
		{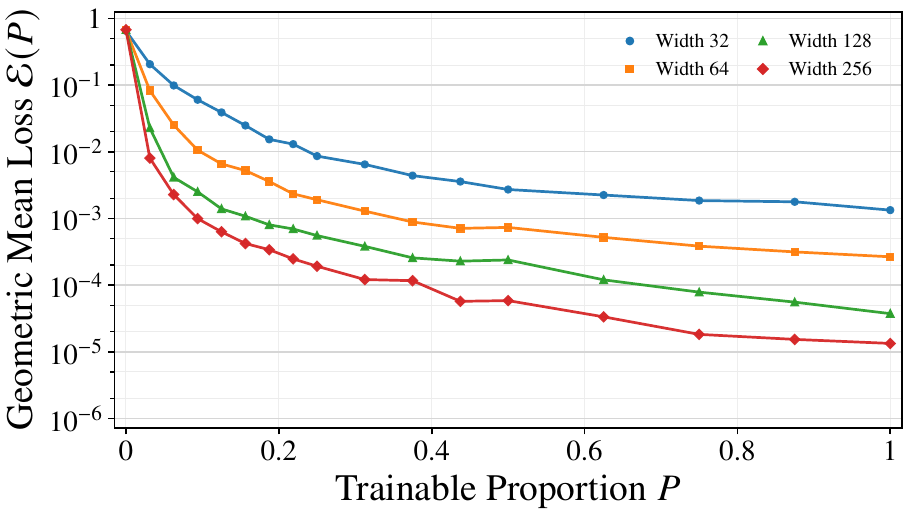}
	}\hfill
	\subfloat[Width variation gain.]{
		\includegraphics[width=0.47\linewidth]
		{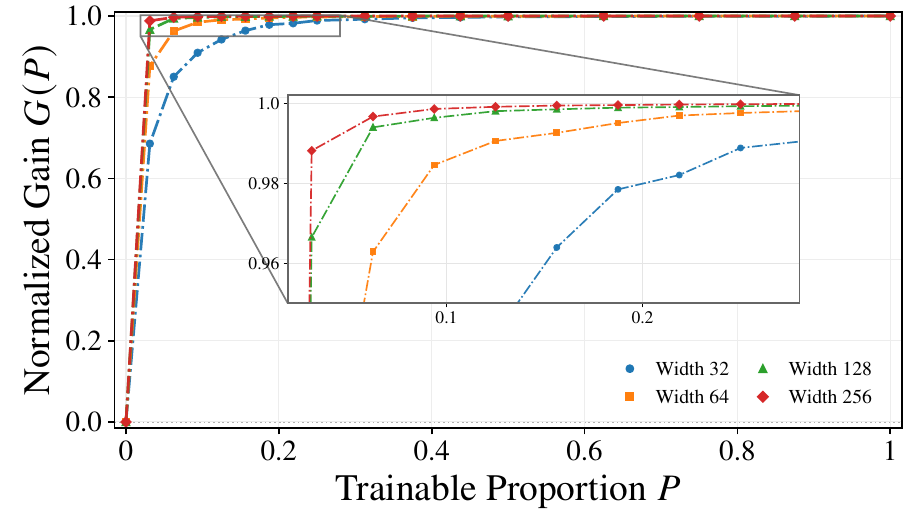}
	}
	\caption{Training results for $f_2^*$ with GELU activation.}
	\label{fig:GeLU-Target2}
\end{figure}

\begin{figure}[p]
	\centering
		\subfloat[Depth variation loss.]{
		\includegraphics[width=0.47\linewidth]
		{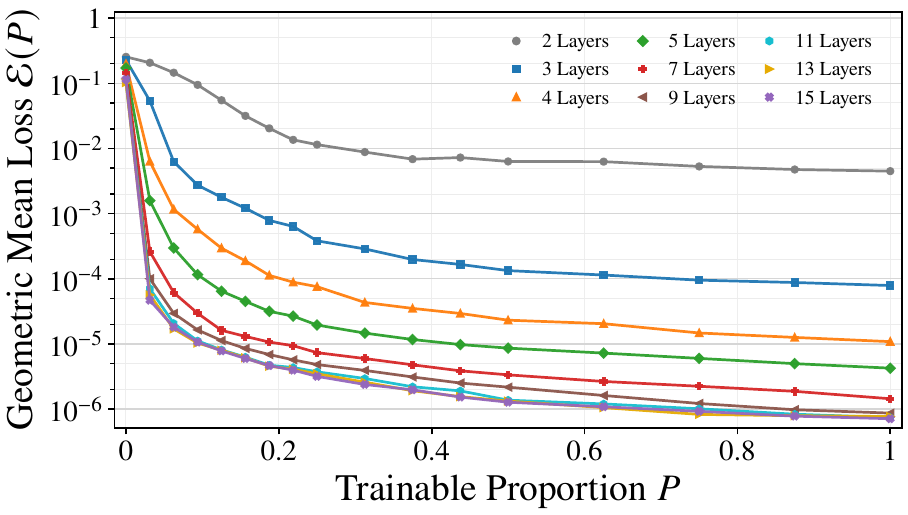}
	}\hfill
	\subfloat[Depth variation gain.]{
		\includegraphics[width=0.47\linewidth]
		{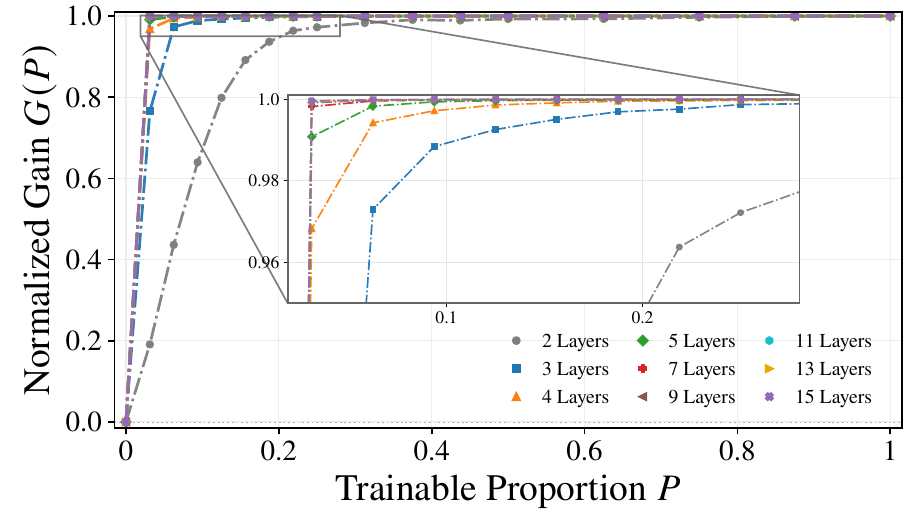}
	}
	
	\subfloat[Width variation loss.]{
		\includegraphics[width=0.47\linewidth]
		{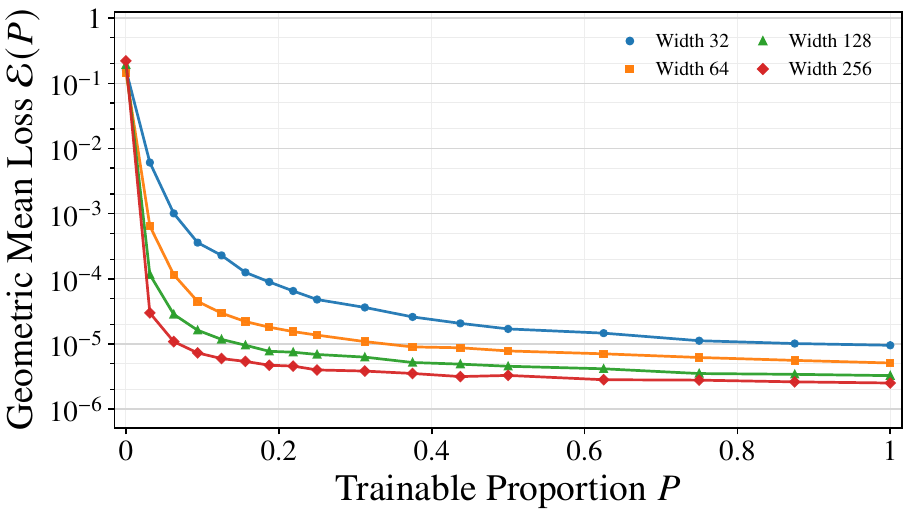}
	}\hfill
	\subfloat[Width variation gain.]{
		\includegraphics[width=0.47\linewidth]
		{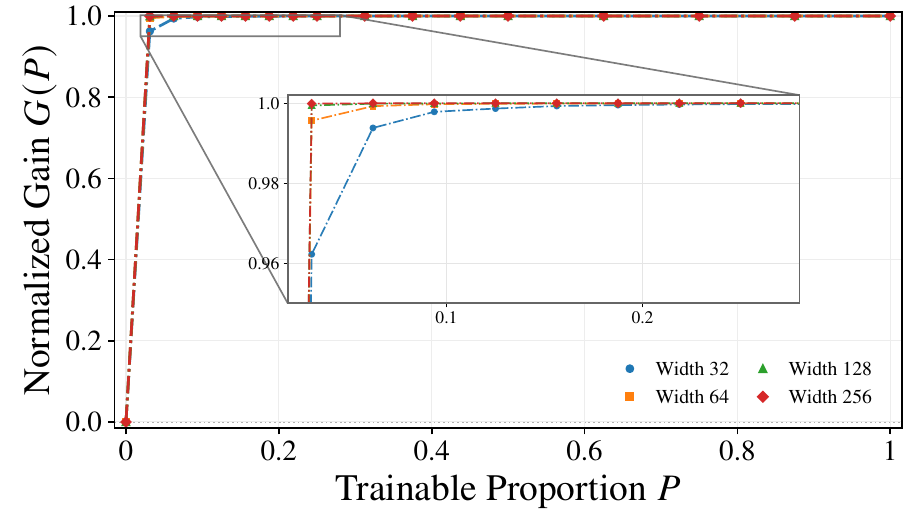}
	}
	\caption{Training results for $f_3^*$ with GELU activation.}
	\setcounter{subfigure}{0}
	\subfloat[Depth variation loss.]{
		\includegraphics[width=0.47\linewidth]
		{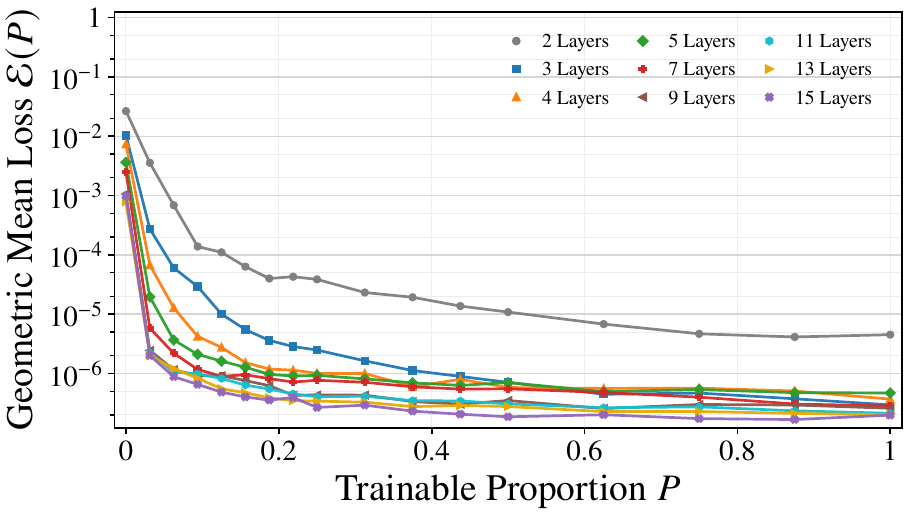}
	}\hfill
	\subfloat[Depth variation gain.]{
		\includegraphics[width=0.47\linewidth]
		{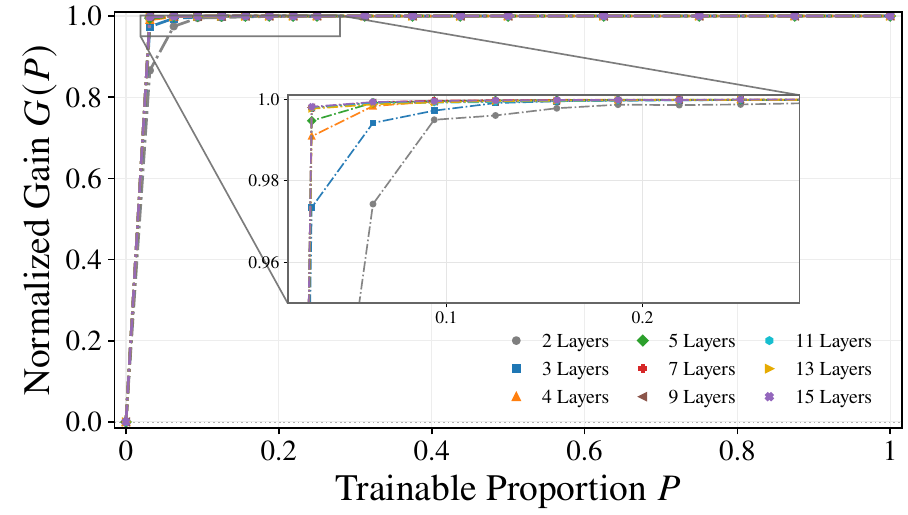}
	}
	
	\subfloat[Width variation loss.]{
		\includegraphics[width=0.47\linewidth]
		{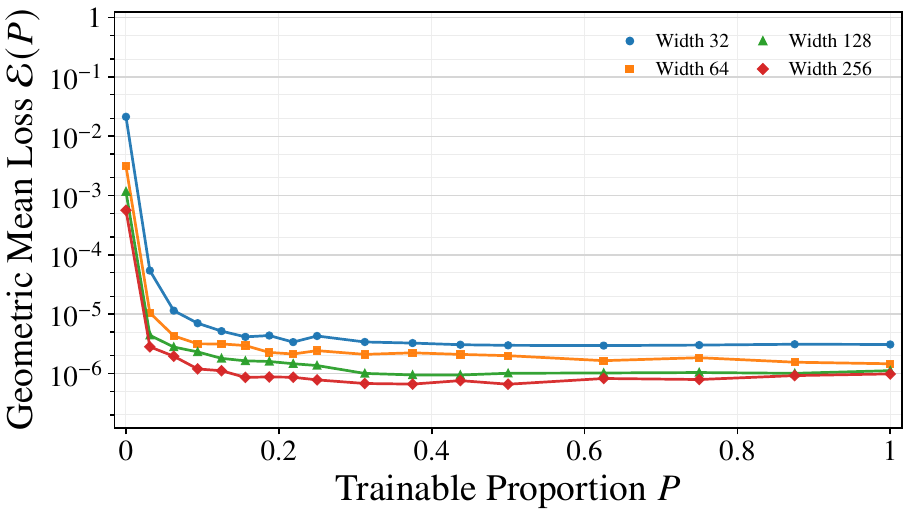}
	}\hfill
	\subfloat[Width variation gain.]{
		\includegraphics[width=0.47\linewidth]
		{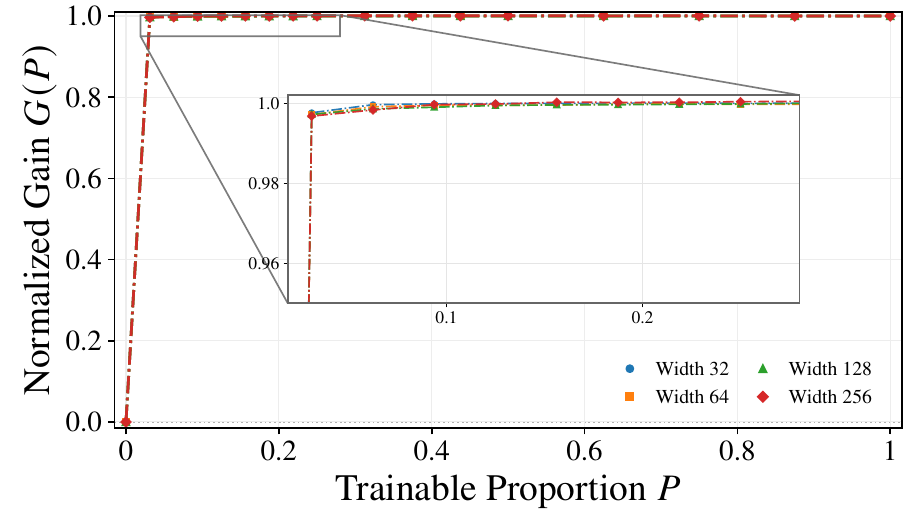}
	}
	\caption{Training results for $f_1^*$ with $\tanh$ activation.}
\end{figure}

\begin{figure}[p]
	\centering
	\subfloat[Depth variation loss.]{
		\includegraphics[width=0.47\linewidth]
		{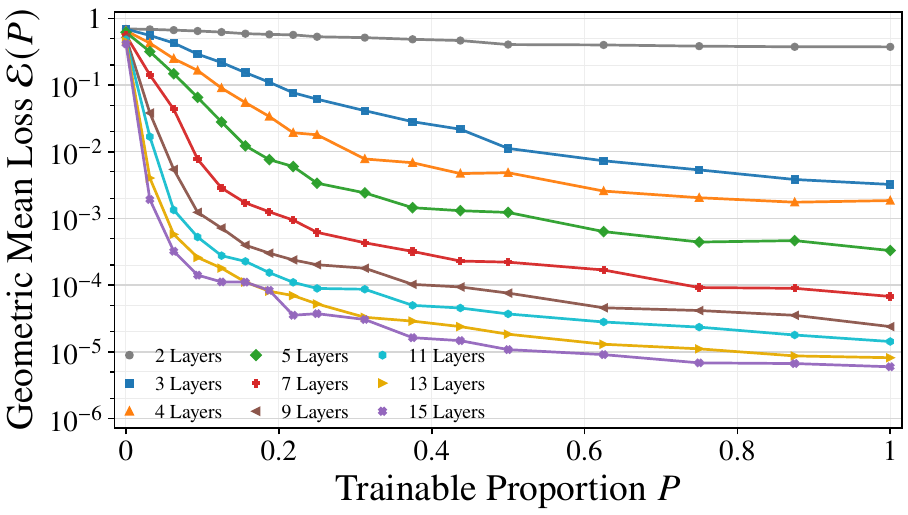}
	}\hfill
	\subfloat[Depth variation gain.]{
		\includegraphics[width=0.47\linewidth]
		{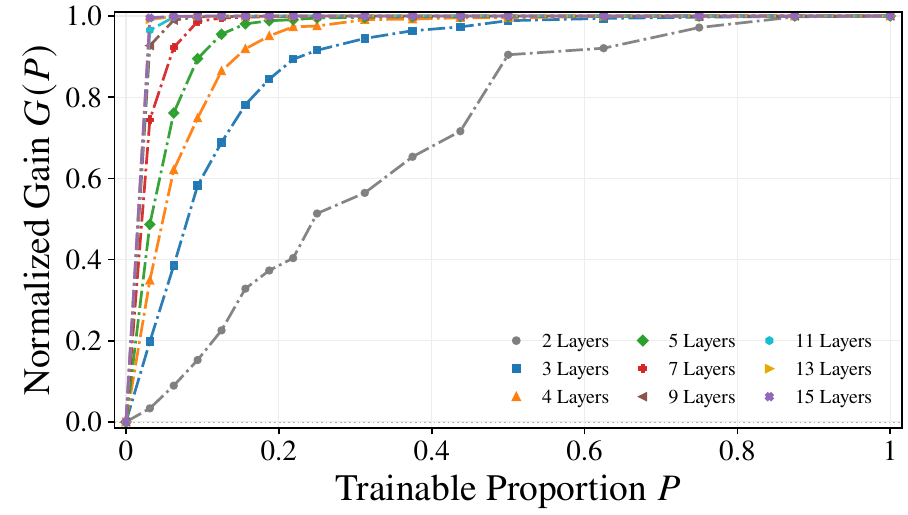}
	}
	
	\subfloat[Width variation loss.]{
		\includegraphics[width=0.47\linewidth]
		{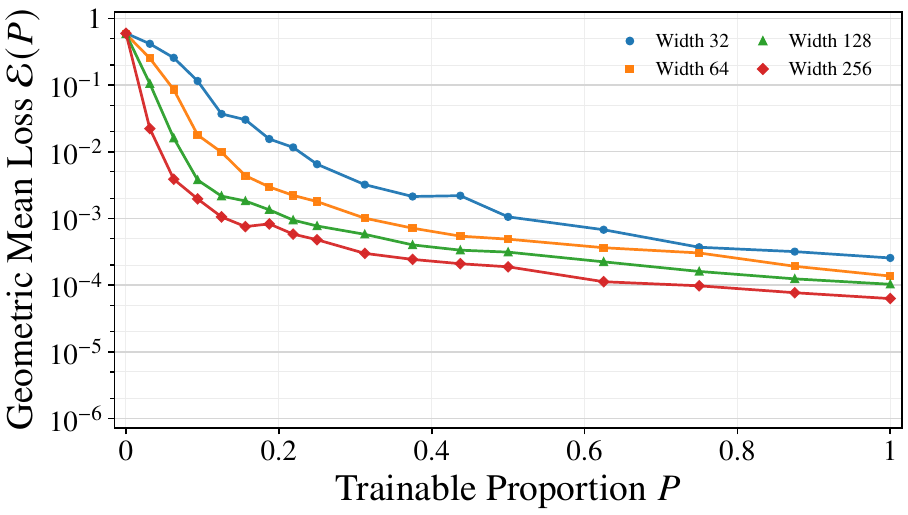}
	}\hfill
	\subfloat[Width variation gain.]{
		\includegraphics[width=0.47\linewidth]
		{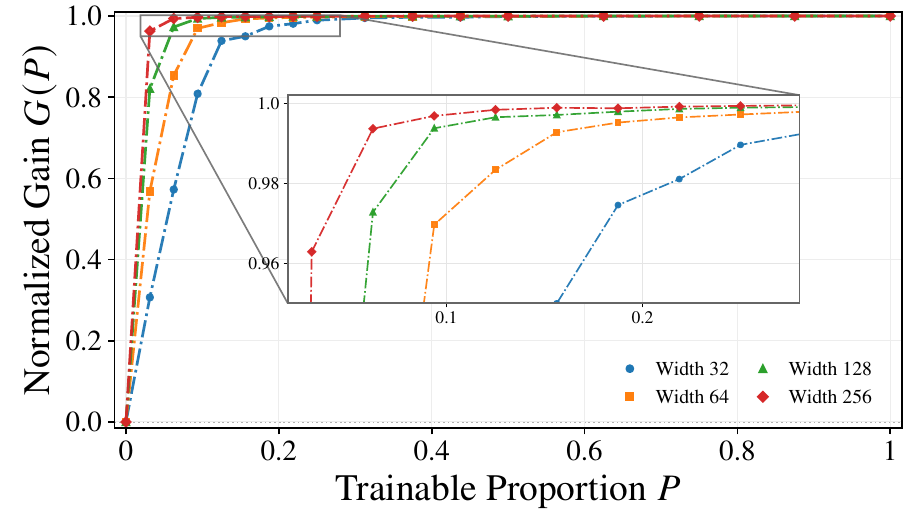}
	}
	\caption{Training results for $f_2^*$ with $\tanh$ activation.}
	\setcounter{subfigure}{0}
	\subfloat[Depth variation loss.]{
		\includegraphics[width=0.47\linewidth]
		{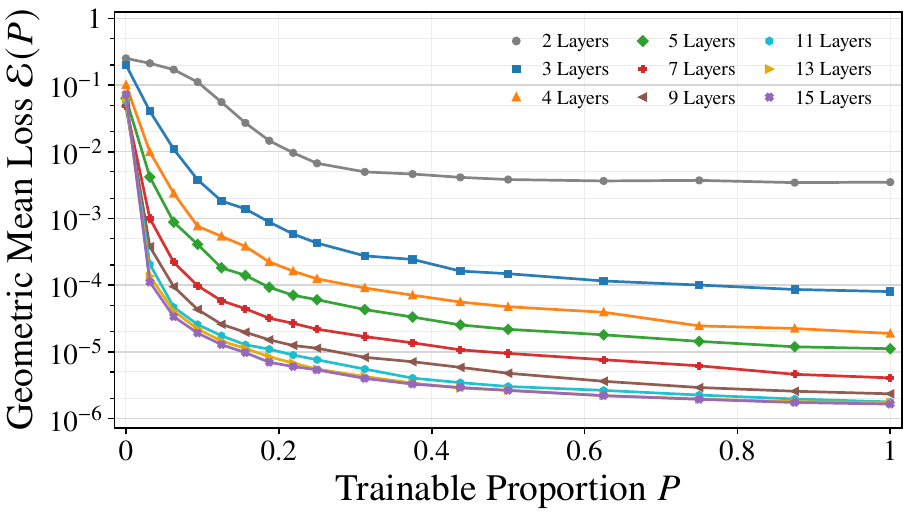}
	}\hfill
	\subfloat[Depth variation gain.]{
		\includegraphics[width=0.47\linewidth]
		{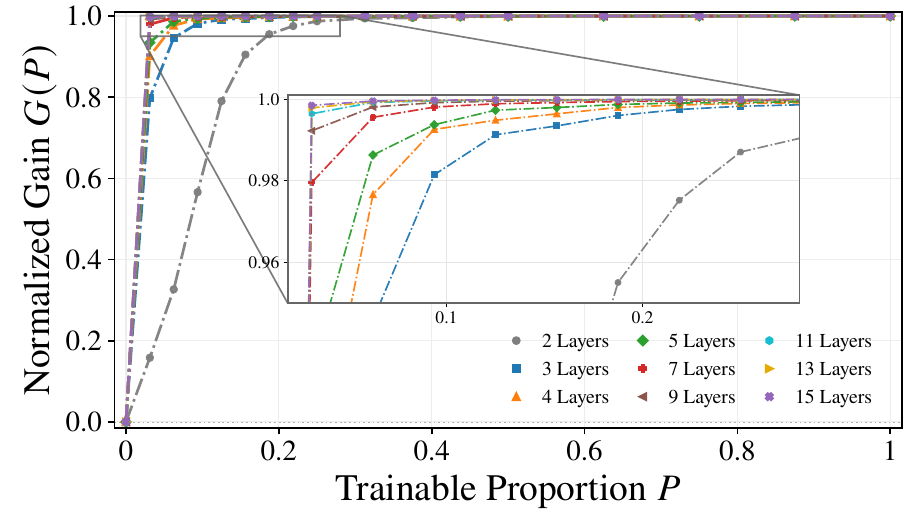}
	}
	
	\subfloat[Width variation loss.]{
		\includegraphics[width=0.47\linewidth]
		{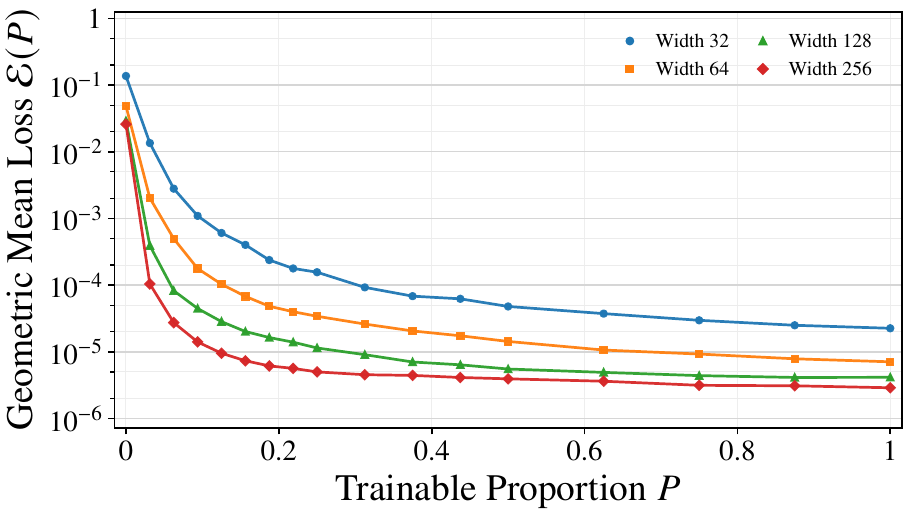}
	}\hfill
	\subfloat[Width variation gain.]{
		\includegraphics[width=0.47\linewidth]
		{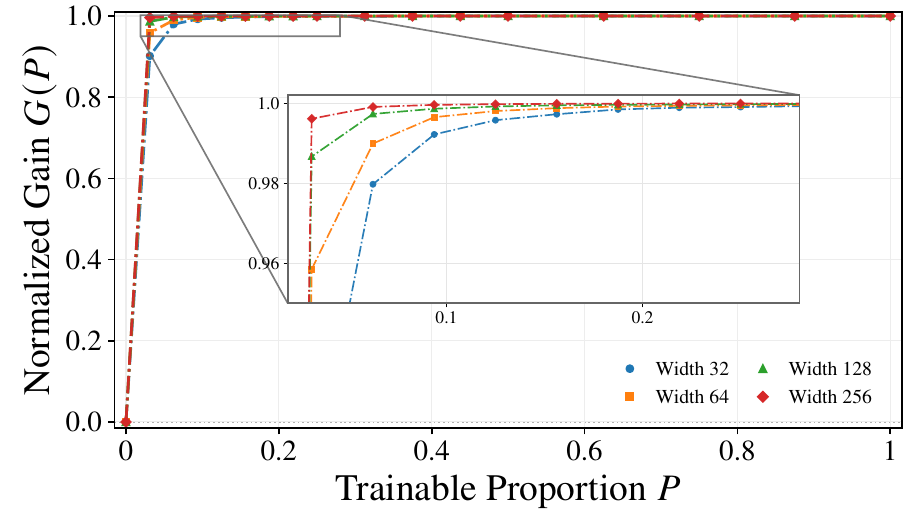}
	}
	\caption{Training results for $f_3^*$ with $\tanh$ activation.}
	\label{fig:tanh-target3}
\end{figure}

	\clearpage
	\begin{ack}
		The authors used ChatGPT (OpenAI) for language polishing and assistance in checking potential errors. The authors take full responsibility for the content of this article.
	\end{ack}
%
	
\bibliographystyle{emss} 
\bibliography{PT}

\end{document}